\documentclass[a4paper]{amsart}

\usepackage[a4paper, margin=2cm]{geometry}
\usepackage[british]{babel}

\usepackage{amsmath}
\usepackage{amssymb}
\usepackage{calc}
\usepackage{enumitem}
\usepackage{graphicx}
\usepackage{tikz-cd}
\usepackage{url}
\usepackage{xcolor}
\usepackage{etoolbox}
\usepackage{tabularx}
\usepackage{mathtools}
\usepackage{xcolor}
\usepackage{hyperref}

\hypersetup{
  colorlinks   = true, 
  urlcolor     = {blue!50!black}, 
  linkcolor    = {blue!50!black}, 
  citecolor   = {red!50!black} 
}

\numberwithin{equation}{section}

\usepackage{cleveref}

\theoremstyle{plain}
\newtheorem{thm}{Theorem}[section]

\newtheorem{prop}[thm]{Proposition}
\newtheorem{lem}[thm]{Lemma}
\newtheorem{fact}[thm]{Fact}
\newtheorem*{lem*}{Lemma}
\newtheorem*{cor*}{Corollary}

\newtheorem{manualthminner}{Theorem}
\newenvironment{manualthm}[1]{%
  \IfBlankTF{#1}
    {\renewcommand{\themanualthminner}{\unskip}}
    {\renewcommand\themanualthminner{#1}}%
  \manualthminner
}{\endmanualthminner}

\newtheorem{manualpropinner}{Proposition}
\newenvironment{manualprop}[1]{%
  \IfBlankTF{#1}
    {\renewcommand{\themanualpropinner}{\unskip}}
    {\renewcommand\themanualpropinner{#1}}%
  \manualpropinner
}{\endmanualpropinner}

\theoremstyle{definition}
\newtheorem{defn}[thm]{Definition}
\newtheorem{eg}[thm]{Example}
\newtheorem{notn}[thm]{Notation}
\newtheorem{qn}[thm]{Question}
\newtheorem{conj}[thm]{Conjecture}
\newtheorem*{qn*}{Question}

\newtheorem*{term*}{Terminology}

\theoremstyle{remark}
\newtheorem{rem}[thm]{Remark}

\DeclareMathOperator{\Age}{Age}
\DeclareMathOperator{\Ao}{Age_\omega}
\DeclareMathOperator{\Ainf}{Age_{inf}}
\DeclareMathOperator{\ar}{ar}
\DeclareMathOperator{\Aut}{Aut}

\DeclareMathOperator{\cod}{cod}
\DeclareMathOperator{\dom}{dom}
\DeclareMathOperator{\dlim}{\underrightarrow{\mathrm{lim}}\,}
\DeclareMathOperator{\dplus}{d_+}

\DeclareMathOperator{\Fin}{Fin}
\DeclareMathOperator{\FrLim}{FrLim}
\DeclareMathOperator{\id}{id}
\DeclareMathOperator{\im}{im}
\DeclareMathOperator{\Iso}{Iso}

\DeclareMathOperator{\Pfin}{\mathcal{P}_{fin}}
\DeclareMathOperator{\RelStr}{RelStr}
\DeclareMathOperator{\rt}{root}

\DeclareMathOperator{\sw}{sw}

\DeclareMathOperator{\tp}{tp}

\newcommand{\N}{\mathbb{N}}
\newcommand{\Q}{\mathbb{Q}}

\newcommand{\Z}{\mathbb{Z}}
\newcommand{\ic}{\mathbin{\bot}}
\newcommand{\mb}[1]{\mathbb{#1}}
\newcommand{\mc}[1]{\mathcal{#1}}

\newcommand{\oa}[1]{\overrightarrow{#1}}
\newcommand{\all}{\forall\,}
\newcommand{\ex}{\exists\,}
\newcommand{\sub}{\subseteq}

\newcommand{\tld}[1]{\tilde{#1}}
\newcommand{\ext}{\text{ext}}
\newcommand{\fin}{\subseteq_{\text{fin\!}}}
\newcommand{\fg}{\subseteq_{\text{f.g.\!}}}

\newcommand{\lex}{\text{lex}}
\newcommand{\rel}{\text{rel}}

\newcommand{\set}{\text{set}}

\newcommand{\ri}{\circ}
\newcommand{\ra}{\rightarrow}
\newcommand{\rra}{\rightrightarrows}

\newcommand{\la}{\leftarrow}
\newcommand{\symdiff}{\mathbin{\bigtriangleup}}
\newcommand{\Mod}[1]{\ (\mathrm{mod}\ #1)}
\newcommand{\eps}{\varepsilon}

\def\ind{\mathrel{\raise0.2ex\hbox{\ooalign{\hidewidth$\vert$\hidewidth\cr\raise-0.9ex\hbox{$\smile$}}}}}

\usepackage[f]{esvect}

\newcommand{\Fr}{Fra\"{i}ss\'{e} }

\newcommand{\Ka}{Kat\v{e}tov }
\newcommand{\NR}{Ne\v{s}et\v{r}il-R\"{o}dl }

\allowdisplaybreaks

\usetikzlibrary{calc}

\tikzset{curve/.style={settings={#1},to path={(\tikztostart)
    .. controls ($(\tikztostart)!\pv{pos}!(\tikztotarget)!\pv{height}!270:(\tikztotarget)$)
    and ($(\tikztostart)!1-\pv{pos}!(\tikztotarget)!\pv{height}!270:(\tikztotarget)$)
    .. (\tikztotarget)\tikztonodes}},
    settings/.code={\tikzset{quiver/.cd,#1}
        \def\pv##1{\pgfkeysvalueof{/tikz/quiver/##1}}},
    quiver/.cd,pos/.initial=0.35,height/.initial=0}

\tikzset{tail reversed/.code={\pgfsetarrowsstart{tikzcd to}}}
\tikzset{2tail/.code={\pgfsetarrowsstart{Implies[reversed]}}}
\tikzset{2tail reversed/.code={\pgfsetarrowsstart{Implies}}}
\tikzset{no body/.style={/tikz/dash pattern=on 0 off 1mm}}

\makeatletter
\def\author@andify{%
  \nxandlist {\unskip ,\penalty-1 \space\ignorespaces}%
    {\unskip {} \@@and~}%
    {\unskip \penalty-2 \space \@@and~}%
}
\makeatother

\author{Aleksandra Kwiatkowska}
\address{\parbox{\linewidth}{Aleksandra Kwiatkowska\\
Institut f\"{u}r Mathematische Logik, Universit\"{a}t  M\"{u}nster\\ Einsteinstraße 62\\
48149 M\"{u}nster, Germany\\
and \\
Instytut Matematyczny, Uniwersytet Wroc{\l}awski\\
pl.\ Grunwaldzki 2/4 \\ 
50-384 Wroc{\l}aw, Poland
}
}
\email{aleksandra.kwiatkowska@uwr.edu.pl}

\thanks{The first and second authors were funded by the Deutsche Forschungsgemeinschaft (DFG, German Research Foundation) under Germany’s Excellence Strategy EXC 2044–390685587, Mathematics M\"{u}nster: Dynamics–Geometry–Structure and CRC 1442 Geometry: Deformations and Rigidity. The second author is additionally funded by Project 24-12591M of the Czech Science Foundation (GA\v{C}R)}

\author{Rob Sullivan}
\address{\parbox{\linewidth}{Rob Sullivan\\
Computer Science Institute, Charles University\\ Malostransk\'{e} n\'{a}m.\ 25,\\Prague, Czech Republic
}
}
\email{robertsullivan1990+maths@gmail.com}

\author{Jeroen Winkel}
\address{Jeroen Winkel}
\email{winkeljeroen+maths@gmail.com}



\title{Group-extensive embeddings into Fra\"{i}ss\'{e} structures and stationary weak independence relations}

\date{\today}

\begin{document}

\begin{abstract}
    Let $M$ be a \Fr structure (a countably infinite ultrahomogeneous structure). We call an embedding $f : A \to M$ \emph{group-extensive} if each automorphism of its image extends to an automorphism of $M$, where the extension map respects composition. We say that $M$ has \emph{group-extensible $\omega$-age} if each substructure admits a group-extensive embedding into $M$.
    
    We investigate the relationship between the following two properties: the presence of a stationary weak independence relation (SWIR) on $M$, and group-extensibility of the $\omega$-age of $M$. We show that linearly ordered \Fr structures with a SWIR have group-extensible $\omega$-age, but also we give examples of \Fr structures where only one of the two properties holds. Finally, we consider whether a wide range of examples of \Fr structures have group-extensible $\omega$-age or a finite SWIR expansion, including all countably infinite ultrahomogeneous oriented graphs (with one exception).
\end{abstract}

\maketitle

\section{Introduction}

Recall that a structure $M$ is \emph{ultrahomogeneous} if each isomorphism between finitely generated substructures of $M$ extends to an automorphism of $M$. We call a countably infinite ultrahomogeneous structure $M$ a \emph{\Fr structure} (see \cite{Mac11}, \cite{Hod93} for background). In this paper, we investigate the following two notions for \Fr structures $M$.

\begin{itemize}
    \item \textbf{Group-extensive embeddings.} We say that an embedding $f : A \to M$ is \emph{group-extensive} if there is a group embedding $\theta : \Aut(f(A)) \to \Aut(M)$ with $g \sub \theta(g)$ for all $g \in \Aut(f(A))$ (informally, each automorphism of the image of $A$ extends to an automorphism of $M$ in a manner that preserves group composition). Define the \emph{$\omega$-age} of $M$, written $\Ao(M)$, to be the class of structures which embed into $M$ (we permit infinite structures). If each $A \in \Ao(M)$ admits an group-extensive embedding into $M$, we say that $\Ao(M)$ is \emph{group-extensible}.
    \item \textbf{Stationary weak independence relations (SWIRs).} These generalise the well-known notion of stationary independence relation (SIR) by dropping the assumption of symmetry. A \emph{stationary weak independence relation} is a ternary relation $\ind$ on the set of finitely generated substructures of a \Fr structure $M$, written as $B \ind_A C$ for $A, B, C \fg M$, which satisfies certain axioms: automorphism-invariance, stationarity, existence and monotonicity.
\end{itemize}

We are concerned with the following three main questions:
\begin{qn} \label{qn extensible}
    For which \Fr structures $M$ is $\Ao(M)$ group-extensible?
\end{qn}
\begin{qn} \label{qn SWIR extensible relationship}
    What is the relationship between group-extensibility of $\Ao(M)$ and the presence of a SWIR on $M$?
\end{qn}
\begin{qn} \label{qn SWIR}
    Which \Fr structures have a SWIR? If a \Fr structure does not have a SWIR, does it have an expansion with one?
\end{qn}

We discuss the background and motivation for these questions below.

\textbf{For brevity, throughout this paper we write \emph{$\circ$-extensive} and \emph{$\circ$-extensible} instead of group-extensive and group-extensible}. ($\circ$ is intended to be a mnemonic for the preservation of group composition.) See \Cref{r: BK} for a discussion of terminology for related notions already used in the literature.

\subsection*{Background: positive examples for \texorpdfstring{\Cref{qn extensible}}{Question 1.1} given by the \Ka tower construction} \hfill

\Cref{qn extensible} has been considered before by a number of authors. The first result in the literature in this area is \cite[Theorem 3.1]{Hen71}, where Henson showed that each countably infinite graph admits a \emph{uniquely extensive} embedding into the random graph $\Gamma$: an embedding such that each automorphism of the image extends \emph{uniquely} to an automorphism of $\Gamma$ (see \Cref{s: unique ext}). His argument adapts straightforwardly to give the extensibility of $\Ao(\Gamma)$. Repurposing a construction of \Ka (\cite{Kat88}), Uspenskij showed in \cite{Usp90} that every separable metric space $A$ admits an isometric embedding $f$ into the complete Urysohn space $\mb{U}$ such that there is a continuous group embedding $\Iso(f(A)) \to \Iso(\mb{U})$ where each isometry of $f(A)$ maps to an extension of itself; this argument can be straightforwardly adapted to show that $\Ao(\mb{U}_\Q)$ is $\circ$-extensible, where $\mb{U}_\Q$ is the rational Urysohn space. Jaligot showed in \cite{Jal07} that each countably infinite tournament admits a uniquely extensive embedding into the random tournament. This result stimulated further interest in \Cref{qn extensible} in subsequent years. Bilge and Melleray modified the argument of Uspenskij and \Ka to show that $\Ao(M)$ is $\circ$-extensible for free amalgamation structures $M$ (\cite[Theorem 3.9]{BM13}), and this argument was extended by M\"{u}ller in \cite{Mul16} to give extensibility of $\Ao(M)$ for \Fr structures $M$ with a stationary independence relation (SIR). The construction in each case is similar, and is sometimes called a \emph{\Ka tower construction}.

Kubi\'{s} and Ma\v{s}ulovi\'{c} observe in \cite{KM17} that the \Ka tower construction appearing in various guises in the above papers is actually an instance of a more general functorial construction on categories. The key ingredient that makes the construction work is the following:

\begin{defn}[{\cite[Definition 2.1]{KM17}}] \label{def Ka functor}
    Let $M$ be a \Fr structure. Recall that $\Age(M)$ is the class of finitely generated structures which are embeddable in $M$. (We allow the empty structure.) We write $\mc{A}(M)$, $\mc{A}_\omega(M)$ for the classes of non-empty structures in $\Age(M)$, $\Ao(M)$. Consider $\mc{A}(M)$, $\mc{A}_\omega(M)$ as categories, where in each case the class of morphisms is given by all embeddings between objects.
    
    A \emph{\Ka functor $K$ for $M$} is a functor $K : \mc{A}(M) \to \mc{A}_\omega(M)$ such that there exists a family of embeddings $\eta = (\eta_A : A \to K(A))_{A \in \mc{A}(M)}$ with the following properties:
    \begin{enumerate}[label=(\roman*)]
        \item \label{Ka nat trans} $\eta$ is a natural transformation: that is, for each embedding $f : A \to B$ in $\mc{A}(M)$ we have $K(f) \circ \eta_A = \eta_B \circ f$;
        \item \label{Ka ope} for each one-point extension $\zeta : A \to B$ in $\mc{A}(M)$, there is an embedding $f : B \to K(A)$ with $f \circ \zeta = \eta_A$.
    \end{enumerate}

    Here, a \emph{one-point extension in $\mc{A}(M)$} is an embedding $\zeta : A \to B$ in $\mc{A}(M)$ where there exists $e \in B \setminus \zeta(A)$ such that $B$ is generated by $\zeta(A) \cup \{e\}$.

    (In fact, \cite{KM17} gives a more abstract version of the above, allowing more general classes of morphisms.)
\end{defn}

A \Ka functor can be thought of as giving a functorial amalgam of all one-point extensions of a structure, for each structure in the age, and this functorial amalgam is enough to carry out a categorical version of the \Ka tower construction:

\begin{fact}[{\cite[Corollaries 3.9 and 3.12, rephrased]{KM17}}] \label{Ka implies extensible}
    Let $M$ be a \Fr structure. If $M$ has a \Ka functor, then each $A \in \Ao(M)$ admits a $\circ$-extensive embedding $f : A \to M$ where, in addition, the group embedding $\theta : \Aut(f(A)) \to \Aut(M)$ is an embedding of topological groups.
    
    In particular, if $M$ has a \Ka functor, then $\Ao(M)$ is $\circ$-extensible. 
\end{fact}

In the above \Cref{Ka implies extensible}, the topology on $\Aut(f(A))$ and $\Aut(M)$ is the usual pointwise-convergence topology. We concern ourselves only with group embeddings in this paper, rather than topological group embeddings, but note that each time we prove the existence of a \Ka functor, \Cref{Ka implies extensible} implies that the group embeddings thereby obtained are also embeddings of topological groups.

We give a very brief sketch of the proof idea for \Cref{Ka implies extensible}. \cite[Theorem 2.2]{KM17} shows that, given a \Ka functor $K : \mc{A}(M) \to \mc{A}_\omega(M)$ with an associated natural transformation $\eta = (\eta_A)_{A \in \mc{A}(M)}$, one may extend $K$, $\eta$ to a functor $\hat{K} : \mc{A}_\omega(M) \to \mc{A}_\omega(M)$ and a natural transformation $\hat{\eta} = (\hat{\eta}_A)_{A \in \mc{A}_\omega(M)}$. Subsequently \cite[Theorem 3.2, Theorem 3.3]{KM17} then show that for each $A \in \mc{A}_\omega(M)$, applying $\hat{K}$ $\omega$-many times gives a copy of the \Fr limit $M$, and as this construction is functorial, we immediately obtain \Cref{Ka implies extensible}.

As a consequence of \Cref{Ka implies extensible}, in order to provide positive examples to answer \Cref{qn extensible}, it suffices to produce \Fr structures $M$ which have \Ka functors. In \cite{KM17} a wide variety of examples of \Fr structures with \Ka functors were already given: those with free amalgamation (\cite[Theorem 2.12]{KM17}), the rational Urysohn space, the countable atomless Boolean algebra, the random poset, the random lattice, the random semilattice, the random digraph and the random tournament.

\begin{rem}
    The reader may wonder if in fact all embeddings are $\circ$-extensive. We show that this is not the case for the random graph, denoted by $M$. As $M$ is universal for countable graphs, it contains a graph $A \cup \{v\}$ with $v \notin A$ where $A$ is an infinite anticlique and the set $C = \{u \in A \mid u \sim v\}$ is infinite/coinfinite in $A$. As there are continuum-many infinite/coinfinite subsets of $A$ and $M \setminus A$ is countable, there is an infinite/coinfinite subset $C' \sub A$ such that there is no $w \in M$ with $C' = \{u \in A \mid u \sim w\}$. Then any automorphism of $A$ sending $C$ to $C'$ does not extend to an automorphism of $M$, so we have produced an embedding of a countably infinite anticlique into $M$ which is not $\circ$-extensive.
\end{rem}

\begin{rem} \label{r: BK}
    There are a number of occurrences in the literature of notions related to $\circ$-extensive embeddings, with varying terminology. We survey some of these here in an attempt to clarify terminology for the reader.

    In the recent paper \cite{BK26}, Barto{\v{s}} and Kubi{\'{s}} define that an embedding $e : X \to Y$ is ``extensible" if, for each $h \in \Aut(X)$, there is $\tld{h} \in \Aut(Y)$ with $\tld{h} \circ e = e \circ h$. In our context where $X = A$ and $Y = M$ is a \Fr structure, this means that each automorphism of $e(A)$ extends to an automorphism of $M$, without assuming that group composition is preserved. As such, this is a weaker notion than the notion of \emph{$\circ$-extensive embedding} used in the present paper. In \cite{HKO11}, the authors define that $X \sub Y$ is ``symmetrically embedded" in $Y$ if each automorphism of $X$ extends to an automorphism of $Y$: this is the same as the ``extensible" notion of Barto{\v{s}} and Kubi{\'{s}} for the inclusion $X \hookrightarrow Y$. The same notion of symmetric embedding also appears for the specific case of graphs in \cite{GK11} (where it is unnamed but denoted by $\sub^A$).

    Our notion of \emph{$\circ$-extensive embedding} appears in \cite[Definition 3.1]{KS95} (and in \cite[Section 1.1]{GK11} in the case of graphs), where it is called a ``bi-embedding". (Note that \cite[Definition 3.1]{KS95} has a typo: there should be an $=$ sign before $\Theta$.)

    See \Cref{r: I_n-free BK strictly weaker} for an example of an embedding $e : A \to M$ which is not $\circ$-extensive but is nonetheless ``extensible" in the sense of Barto\v{s} and Kubi\'{s}.
\end{rem}

\subsection*{Background: negative examples for \texorpdfstring{\Cref{qn extensible}}{Question 1.1} and non-universal automorphism groups} \hfill

Jaligot asked a more abstract version of \Cref{qn extensible}:

\begin{qn*}[\cite{Jal07}]
    Does there exist a \Fr structure $M$ for which $\Aut(M)$ is not universal?
\end{qn*}

Here, the automorphism group $\Aut(M)$ of a \Fr structure $M$ is \emph{universal} if for each $A \in \Ao(M)$ there is an abstract group embedding $\Aut(A) \to \Aut(M)$. Note that, a priori, the group embedding need not have any relation to the structure. It is immediate that extensibility of $\Ao(M)$ implies universality of $\Aut(M)$.

Piotr Kowalski observed that for each prime $p$ the \Fr limit of the class of finite fields of characteristic $p$ has non-universal automorphism group (see \cite{Mul16}), but the question remained open for relational classes until the relatively recent paper \cite{KS20} of Kubi\'{s} and Shelah, which gives examples where universality dramatically fails. In particular, the end of \cite[Section 4]{KS20} gives the example of a \Fr structure $M_\text{KS}$ consisting of a red countably infinite set and a blue copy of $(\Q, <)$ with a generic bipartite graph structure between them (in a signature $\{R, B, <, \sim\}$). The automorphism group of $M_\text{KS}$ is torsion-free (it embeds in $\Aut(\Q, <)$), but the class of automorphism groups of structures in $\Age(M_\text{KS})$ contains all finite symmetric groups.

\subsection*{Background: stationary (weak) independence relations} \hfill

A \emph{stationary independence relation (SIR)}, first defined in \cite{TZ13}, is a ternary relation $\ind$ on the set of finitely generated substructures of a \Fr structure $M$, written as $B \ind_A C$ for $A, B, C \fg M$, which satisfies the axioms of automorphism-invariance, stationarity, existence, monotonicity and symmetry (see \Cref{d:swir}). Examples of \Fr structures with SIRs include those with free amalgamation, the rational Urysohn space, the random poset and many metrically homogeneous graphs (for the latter see \cite{ABH25}). SIRs have numerous applications in the study of \Fr structures (see for example \cite{TZ13}, \cite{Mul16}, \cite{ABH25}, \cite{PS17}, \cite{KS19}, \cite{EHKLZ21}), and were originally used in \cite{TZ13} to determine the normal subgroups of the isometry group of the (complete) Urysohn space.

As $\Ao(M)$ is $\circ$-extensible for $M$ with free amalgamation, and as free amalgamation gives an example of a SIR, it is natural to ask whether the presence of a SIR on $M$ is sufficient for extensibility. In \cite{Mul16}, M\"{u}ller showed that this is indeed the case: if a \Fr structure $M$ has a SIR, then $\Ao(M)$ is $\circ$-extensible. Her proof used a \Ka tower construction directly, rather than the abstract machinery of \Ka functors (the papers \cite{KM17}, \cite{Mul16} were more or less contemporaneous). 

Li generalised the definition of SIR in \cite{Li19}, \cite{Li21} to that of a \emph{stationary weak independence relation (SWIR)}, and used SWIRs to show the simplicity of certain automorphism groups of \Fr structures with asymmetric relations by generalising the main theorem of \cite{TZ13}. A SWIR has the same axioms as a SIR minus the assumption of symmetry. (Another generalisation of SIRs to the asymmetric case was given in \cite{CKT21} and similarly used to prove simplicity of certain automorphism groups of \Fr structures.)

The original motivation for this paper was to see if the result of \cite{Mul16} likewise generalises to the case of SWIRs: this formed our initial approach to \Cref{qn SWIR extensible relationship}. Given the wide variety of applications of SIRs or similar notions appearing in the literature, we also regard the presence of a SWIR on a \Fr structure as having its own inherent interest, leading to \Cref{qn SWIR}. In \Cref{SWIRs and SAOs}, we show that the existence of a SWIR on a \Fr structure is equivalent to a certain natural notion of canonical amalgamation on its age, which we call a \emph{standard amalgamation operator}, and so \Cref{qn SWIR} can be viewed as asking how much information is required in an expansion in order to be able to canonically amalgamate elements of the expanded age. A related question is asked in \cite[Conjecture 7.1]{KS19} in the context of $\omega$-categorical structures and with a similar definition of an independence relation called a CIR (canonical independence relation): Kaplan and Simon conjecture that every $\omega$-categorical structure has an $\omega$-categorical expansion with a CIR. We discuss the relationship between SWIRs and CIRs further in \Cref{SWIRs and SAOs}.

\subsection*{Main results and structure of the paper} \hfill

In \Cref{SWIRs and SAOs}, we show that a \Fr structure $M$ has a SWIR if and only if its age has a certain notion of canonical amalgamation, which we call a \emph{standard amalgamation operator}. In \Cref{s: key exs} we consider three key examples. Two of them demonstrate that neither direction of implication in \Cref{qn SWIR extensible relationship} holds:

\begin{manualthm}{\ref{t: no implication SWIR Ka}}
    There is a relational \Fr structure with a \Ka functor but without a (local) SWIR, and there is a relational \Fr structure with a SWIR but without a universal automorphism group (hence its $\omega$-age is not $\circ$-extensible and it has no \Ka functor).
\end{manualthm}

(Note that the \Fr structure $M_\text{KS}$ of Kubi\'{s} and Shelah mentioned above also has a SWIR and a non-universal automorphism group, but we give the first vertex-transitive and primitive examples with these properties.)

The three examples in \Cref{s: key exs} form part of the proof of \Cref{t: big list of exs} below. There is a \textbf{key idea} underlying these examples, as well as the proofs of the main theorems in \Cref{s: order and tourn exps} and most of the examples in \Cref{s: exs}: the definition of a \Ka functor does \emph{not} require that each one-point extension of $A$ occurs uniquely in $K(A)$, and hence when defining a functorial amalgam $K(A)$ of one-point extensions of $A$, we may add multiple copies of each one-point extension labelled with extra information, and use this extra information to guarantee functoriality -- see the discussion at the beginning of \Cref{s: key exs} and also \Cref{ex: Ka for tourn}.

In \Cref{s: order and tourn exps}, we use the equivalence of SWIRs and standard amalgamation operators, together with the machinery of \Ka functors, to give some general positive results for \Cref{qn extensible}. We first quickly recover the result of M\"{u}ller (\cite{Mul16}): specifically, we show that any \Fr structure with a SIR has a \Ka functor (and hence $\circ$-extensible $\omega$-age). We then show:

\begin{manualthm}{\ref{t: order exp}}
    Let $M$ be a linearly ordered \Fr structure with strong amalgamation and a local SWIR. Then $M$ has a \Ka functor, and hence its $\omega$-age is $\circ$-extensible.
\end{manualthm}
\begin{manualthm}{\ref{t: tournament exp}}
    Let $M$ be a relational \Fr structure with strong amalgamation and a local SIR. Then the generic tournament expansion of $M$ has a \Ka functor, and hence has $\circ$-extensible $\omega$-age.
\end{manualthm}

Using \Cref{t: order exp}, we have for example that the generic ordered graph, the generic ordered rational Urysohn space, the generic $n$-linear order and the generic poset with compatible linear order have \Ka functors and hence $\circ$-extensible $\omega$-age -- see \Cref{ex: ordered SWIR strs}. In \Cref{s: peculiar str}, we give an example of a relational \Fr structure $M$ with a tournament relation and a SWIR whose automorphism group is not universal, which shows that the analogous tournament version of \Cref{t: order exp} does not hold.

In \Cref{s: exs}, we consider further examples of \Fr structures, with a particular focus on Cherlin's complete catalogue of countably infinite ultrahomogeneous oriented graphs (\cite{Che98}), and we determine for each example if a \Ka functor is present and if there is a \emph{finite SWIR expansion}: namely, an ultrahomogeneous expansion in a finite language with a SWIR. (One notable example, the oriented graph $P(3)$, remains open.) Specifically:

\begin{manualthm}{\ref{t: big list of exs}}
    Each of the following structures has a \Ka functor (and hence $\circ$-extensible $\omega$-age and universal automorphism group):
    \begin{itemize}
        \item the generic $3$-hypertournament (\ref{ex: 3-ht});
        \item the generic two-graph (\ref{p three key exs});
        \item the generic bipartite tournament and generic $\omega$-partite tournament (\ref{ex: n-partite tournament});
        \item the circular order on $\Q$ (\ref{betweenness and friends});
        \item the dense local order $S(2)$ and the related oriented graphs $S(3)$, $\widehat{\Q}$ (\ref{S(2) and friends});
        \item the generic $\vec{C}_4$-enlarged tournament $\widehat{T^\omega}$ (\ref{ex: C_4-enlarged});
        \item the dense meet-tree (\ref{meet-trees and meet-tree exps});
    \end{itemize}

    The following structures have $\circ$-extensible $\omega$-age, but no \Ka functor:
    \begin{itemize}
        \item the generic $n$-partite tournament for $2 < n < \omega$ (\ref{ex: n-partite tournament});
        \item the betweenness and separation relations on $\Q$ (\ref{betweenness and friends});
    \end{itemize}
    
    The following structures have non-universal automorphism groups (and hence their $\omega$-ages are not $\circ$-extensible and they have no \Ka functor):
    \begin{itemize}
        \item for $n \geq 3$, the generic $n$-anticlique-free oriented graph (\ref{p three key exs});
        \item the semigeneric $\omega$-partite tournament (\ref{ex: semigeneric}).
    \end{itemize}

    Each of the structures listed above has a finite SWIR expansion.
\end{manualthm}

We also show the following structural results:

\begin{manualthm}{\ref{t: products}}
    We have the following regarding products of structures:
    \begin{itemize}
        \item Let $M, N$ be relational \Fr structures with $M$ transitive, and suppose that $M, N$ have \Ka functors. Then the lexicographic product $M[N]^s$ has a \Ka functor (\ref{ex: lex prod}).
        \item Let $M, N$ be relational \Fr structures such that $M$ is transitive and has strong amalgamation, and suppose that $M, N$ have (local) SWIRs. Then the lexicographic product $M[N]^s$ has a (local) SWIR (\ref{ex: lex prod}).
        \item There are relational \Fr structures $M, N$ with \Ka functors such that their free superposition $M \ast N$ does not have a \Ka functor (\ref{ex: free superposition}).
        \item Let $N$ be the generic meet-tree expansion of a transitive \Fr structure $M$ with free amalgamation (satisfying mild non-triviality conditions). Then $\Aut(N)$ is non-universal (\ref{ex: meet-tree exp of free amalg}).
    \end{itemize}
\end{manualthm}

In \Cref{s: unique ext}, we briefly consider \emph{unique} extensibility. Let $M$ be a \Fr structure, and let $\Ainf(M)$ denote the class of infinite structures which are embeddable in $M$. We say that $\Ainf(M)$ is \emph{uniquely extensible} if, for each $A \in \Ainf(M)$, there exists an embedding $f : A \to M$ such that each automorphism of $f(A)$ extends uniquely to an automorphism of $M$ (note that this gives a group embedding $\Aut(f(A)) \to \Aut(M)$). We show:

\begin{manualprop}{\ref{p: ordered random graph uniq ext}}
    Let $M$ be the generic ordered graph. Then $\Ainf(M)$ is uniquely extensible.
\end{manualprop}

The proof follows by a straightforward combination of ideas from \Cref{s: order and tourn exps} with the proof of unique extensibility for the random graph due to Henson (\cite[Theorem 3.1]{Hen71}). Bilge (\cite{Bil12}) extended Henson's result in his PhD thesis to show unique extensibility of $\Ainf(M)$ for transitive $M$ with free amalgamation in a finite relational language (and such that $M$ is not an indiscernible set). We conjecture that $\Ainf(M^<)$ is uniquely extensible for $M^<$ equal to the generic order expansion of a (non-trivial) \Fr structure with free amalgamation: see \Cref{c: LO free}.

\subsection*{Acknowledgements.} We would like to thank Adam Barto\v{s}, David Bradley-Williams, Alessandro Codenotti, Azul Fatalini and Wies\l{}aw Kubi\'{s} for initial discussions, and Itay Kaplan for his talk in M\"{u}nster in September 2023 on the automorphism group of the Rado meet-tree, which indirectly sparked our interest in universality of automorphism groups. The second author would specifically like to thank Wies\l{}aw Kubi\'{s} for an enlightening discussion on $n$-hypertournaments for $n > 3$ which led to \Cref{p: 2timesodd-ht not univ}, and Sam Braunfeld for a discussion regarding the semigeneric tournament. He would also like to thank David Evans for the idea of considering pointwise-stabilisers which led to \Cref{p: T32 SWIR}, the SWIR expansion of the generic $3$-hypertournament, found as an initial example of more general results in a separate project with the third author and Shujie Yang -- the latter of whom we thank for allowing us to include this example here.

\section{SWIRs and standard amalgamation operators} \label{SWIRs and SAOs}

\subsection{Stationary weak independence relations (SWIRs)}

\begin{notn}
    In the below, we write $AB$ to mean the substructure generated by $A \cup B$. We will also sometimes omit parentheses for function arguments when this does not impede clarity, writing $ga$, $gA$ instead of $g(a)$, $g(A)$.
\end{notn}

\begin{defn}
    Let $M$ be a \Fr structure, and let $A, B, B' \fg M$. We write $B \equiv_A B'$ if there exists $f \in \Aut(M)$ with $fB = B'$ and $f|_A = \id_A$, or equivalently if $B, B'$ have the same quantifier-free type over $A$ in some enumeration. (The two definitions are equivalent by ultrahomogeneity.) In the below stationarity axiom, when we write $B \equiv_A B' \Rightarrow B \equiv_{AC} B'$, both automorphisms agree on $AB$.
\end{defn}

\begin{defn}[{\cite[Definition 3.1.1]{Li21}}] \label{d:swir}
    Let $M$ be a \Fr structure. Let $\ind$ be a ternary relation on the set of finitely generated substructures of $M$, written $B \ind_A C$. We say that $\ind$ is a \emph{stationary weak independence relation (SWIR)} on $M$ if it satisfies the following axioms:
    \begin{itemize}
        \item Invariance (Inv): for $g \in \Aut(M)$ we have $B \ind_A C \Rightarrow gB \ind_{gA} gC$;
        \item Existence (Ex): for all $A, B, C \fg M$,
        \begin{itemize}
            \item there exists $B' \fg M$ with $B \equiv_A B'$ such that $B' \ind_A C$;
            \item there exists $C' \fg M$ with $C \equiv_A C'$ such that $B \ind_A C'$;
        \end{itemize}
        \item Stationarity (Sta): for all $A, B, C \fg M$,
        \begin{itemize}
            \item if $B \ind_A C \wedge B' \ind_A C$ and $B \equiv_A B'$, then $B \equiv_{ AC } B'$;
            \item if $B \ind_A C \wedge B \ind_A C'$ and $C \equiv_A C'$, then $C \equiv_{ AB } C'$;
        \end{itemize}
        \item Monotonicity (Mon):
        \begin{itemize}
            \item $ BD  \ind_A C \Rightarrow B \ind_A C \,\wedge\, D \ind_{ AB } C$;
            \item $B \ind_{A}  CD  \Rightarrow B \ind_A C \,\wedge\, B \ind_{ AC } D$;
        \end{itemize}
    \end{itemize}

    We call a ternary relation $\ind$ defined on the set of non-empty finitely generated substructures of $M$ a \emph{local SWIR} if it satisfies the above axioms. We call a SWIR a \emph{global} SWIR when we particularly wish to emphasise that it is also defined over the empty set, but the reader should assume SWIRs are global unless specified otherwise. Note that a global SWIR restricts to a local SWIR. We give several examples in \Cref{s: exs} of structures with local SWIRs which do not have a global SWIR. Note that we defined \Ka functors in \Cref{def Ka functor} to be local.
\end{defn}
\begin{rem}[{\cite[3.1.2(iii)]{Li21}}] \label{lexinv}
    In fact, left-(Ex) and (Inv) imply right-(Ex). Let $A, B, C \fg M$. By left-(Ex), there is $B' \fg M$ with $B \equiv_A B'$ and $B' \ind_A C$. As $M$ is ultrahomogeneous there exists $g \in \Aut(M)$ fixing $A$ pointwise and sending $B'$ to $B$. So by (Inv), we have $B \ind_A gC$. We have $gC \equiv_A C$, showing right-(Ex).
\end{rem}
\begin{rem}
    In \cite[Definition 3.10]{KS19}, Kaplan and Simon give a similar yet distinct definition of a \emph{canonical independence relation (CIR)} in the more general setting of countable structures. Using \Cref{SWIR tr implies mon}, it is straightforward to see that the relationship between the notions of SWIR and CIR for \Fr structures is as follows:
    \begin{center}
        global SWIR $\Leftrightarrow$ (CIR on finitely generated substructures + stationarity over all $A \fg M$). 
    \end{center}
\end{rem}

\begin{defn}[{\cite[Definition 2.1]{TZ13}}]
    Let $M$ be a \Fr structure, and let $\ind$ be a (local) SWIR on $M$. If in addition $\ind$ is symmetric, i.e.\ for all $A, B, C \fg M$ we have that $B \ind_A C \Leftrightarrow C \ind_A B$, then we call $\ind$ a \emph{(local) stationary independence relation} (SIR).
\end{defn}

\begin{eg} \label{ex: structures with SWIRS}
    We give examples of \Fr structures with SWIRs. In each example, the SWIR is given by a certain canonical notion of amalgamation -- see the remainder of this section below.

    \begin{itemize}
        \item A relational \Fr structure $M$ with free amalgamation has a SIR: for $A, B, C \fin M$, define $B \ind_A C$ if $BA, CA$ are freely amalgamated over $A$.
        \item The random poset has a SIR: define $B \ind_A C$ if for $b \in B \setminus A$, $c \in C \setminus A$ such that there is no $a \in A$ with $b < a < c$ or $b > a > c$, we have $b, c$ incomparable.
        \item The rational Urysohn space has a local SIR: define $B \ind_A C$ if for $b \in B \setminus A$, $c \in C \setminus A$, we have $d(b, c) = \min \{d(b, a) + d(c, a) \mid a \in A\}$.
        \item $\Q$ has a SWIR: define $B \ind_A C$ if for $b \in B \setminus A$, $c \in C \setminus A$ such that there is no $a \in A$ with $b < a < c$ or $b > a > c$, we have $b < c$.
        \item The random tournament has a SWIR: define $B \ind_A C$ if for $b \in B \setminus A$, $c \in C \setminus A$ we have $b \ra c$.
        \item Let $n \geq 3$. The same $\ind$ as for the random tournament gives a SWIR for the \Fr limit of the class of finite oriented graphs omitting $n$-anticliques.
        \item If $M_0$, $M_1$ are the \Fr limits of relational strong amalgamation classes $\mc{C}_0$, $\mc{C}_1$ and $M_0$, $M_1$ have SWIRs, then the \Fr limit of the free superposition of $\mc{A}_0$, $\mc{A}_1$ has a SWIR. (This is straightforward. See \Cref{free sup}.)
    \end{itemize}    
\end{eg}

We now prove a number of basic properties of SWIRs. We phrase them for SWIRs, but the same properties likewise hold for local SWIRs assuming that the base is non-empty.

\begin{lem}[Base triviality] \label{base triviality}
    Let $M$ be a \Fr structure with SWIR $\ind$. Then $A \ind_A B$ and $B \ind_A A$ for all $A, B \fg M$.
\end{lem}
\begin{proof}
    By (Ex) there is $B' \fg M$ with $B \equiv_A B'$ and $A \ind_A B'$. By ultrahomogeneity of $M$ there is $g \in \Aut(M)$ with $\id_A \sub g$ and $gB' = B$, so $A \ind_A B$ by (Inv). The proof of the other statement is similar.
\end{proof}

\begin{lem}[{\cite[3.1.2(ii)]{Li21}}] \label{SWIR mon implies tr}
    Let $M$ be a \Fr structure with SWIR $\ind$. Then we have Transitivity (Tr):
    \begin{align*}
        B \ind_A C \wedge B \ind_{AC} D &\Rightarrow B \ind_A CD; \\
        B \ind_A C \wedge D \ind_{AB} C &\Rightarrow BD \ind_A C.
    \end{align*}
\end{lem}
\begin{proof}
    We show the first statement. Suppose $B \ind_A C \wedge B \ind_{ AC } D$. By (Ex) there is $B' \fg M$ with $B' \ind_A  CD $ and $B \equiv_A B'$. By (Mon) we have $B' \ind_A C \wedge B' \ind_{ AC } D$. As $B \equiv_A B'$ and $B \ind_A C \wedge B' \ind_A C$, by (Sta) we have $B \equiv_{ AC } B'$. By (Sta) we have $B \equiv_{ ACD } B'$ and by (Inv) we have $B \ind_A CD$.
\end{proof}

\begin{lem}[adapted from {\cite[Lemma 3.2]{Bau16}}] \label{SWIR tr implies mon}
    Let $M$ be a \Fr structure with $\ind$ satisfying (Inv), (Ex), (Sta), (Tr). Then $\ind$ satisfies (Mon), and so is a SWIR.    
\end{lem}
\begin{proof}
    We show left-(Mon); right-(Mon) is similar. Suppose $BD \ind_A C$. By (Ex) there is $C' \fg M$ with $C \equiv_A C'$ and $B \ind_A C'$, and by (Ex) there is $C'' \fg M$ with $C' \equiv_{AB} C''$ and $D \ind_{AB} C''$. By (Inv) we have $B \ind_A C''$, so $BD \ind_A C''$ by (Tr). Then $C \equiv_{ABD} C''$ by (Sta), hence $B \ind_A C \wedge D \ind_{AB} C$ by (Inv). 
\end{proof}

\begin{rem}
    The previous two lemmas show that it is equivalent to take (Mon) or (Tr) in the definition of SWIR. When checking the SWIR axioms for a particular definition of $\ind$ on a particular structure, we find that (Mon) is usually easier to check. 
\end{rem}

\begin{lem} \label{base_up}
    Let $M$ be a \Fr structure with SWIR $\ind$. Then $B \ind_A C \Rightarrow AB \ind_A C$ and likewise $B \ind_A C \Rightarrow B \ind_A AC$.
\end{lem}
\begin{proof}
    We show the first statement. By (Ex) there is $C' \fg M$ with $C \equiv_{AB} C'$ and $AB \ind_{AB} C'$. There exists $g \in \Aut(M)$ with $\id_{AB} \sub g$ and $gC' = C$. So by (Inv) we have $AB \ind_{AB} C$ and by (Tr) we have $AB \ind_A C$.
\end{proof}

\begin{lem} \label{SWIR strong amalg}
    Let $M$ be a \Fr structure with SWIR $\ind$. Suppose that $M$ has strong amalgamation. Then if $B \ind_A C$, we have that $B \setminus A$, $C \setminus A$ are disjoint sets.
\end{lem}
\begin{proof}
    By \Cref{base_up} we have $ AB  \ind_A  AC $, so we may assume $A \sub B, C$. Let $D$ be a strong amalgam of two copies $B_0, B_1$ of $B$ over $A$. By (Ex) we may assume $D \ind_A C$. By (Mon) we have $B_0 \ind_A C \wedge B_1 \ind_A C$, and by (Sta) we have $B_0 \equiv_{ AC } B_1 \equiv_{ AC } B$. As $B_0 \setminus A, B_1 \setminus A$ are disjoint and have the same quantifier-free type over $C$, no element of $B_0 \setminus A$ or $B_1 \setminus A$ is equal to any element of $C \setminus A$. Hence $B \setminus A$, $C \setminus A$ are disjoint.
\end{proof}

\subsection{Standard amalgamation operators (SAOs)}

We now define \emph{standard amalgamation operators (SAOs)}, and in the next section we will show their equivalence with SWIRs. (We use the terminology ``standard amalgamation" rather than ``canonical amalgamation" to avoid confusion with the canonical independence relations (CIRs) of Kaplan and Simon in \cite{KS19}.) 

\begin{rem}
    \cite[Section 4.5]{ABH25} sketched the equivalence of SIRs and another version of canonical amalgamation -- we generalise this by dropping the symmetry assumption, and we give a different (but equivalent) definition of canonical amalgam which we believe is easier to check in examples. We also give more details in proofs, making the commutative diagrams explicit.
\end{rem}

\begin{notn} \label{emb notn}
    Let $M$ be a \Fr structure. Let $f_i : A_i \to B_i$, $i = 0, 1$, be embeddings in $\Age(M)$. We write $f : (A_0, A_1) \to (B_0, B_1)$ for the pair $f = (f_0, f_1)$. If $A_0 = A_1 = A$, we write $f : A \to (B_0, B_1)$ and $\dom(f) = A$, and if $B_0 = B_1 = B$ we write $f : (A_0, A_1) \to B$ and $\cod(f) = B$. Let $f : (A_0, A_1) \to (B_0, B_1)$, $f' : (A'_0, A'_1) \to (B'_0, B'_1)$ be pairs of embeddings. An embedding $i : f \to f'$ is defined to be a pair $(i_A, i_B)$ where $i_A : (A_0, A_1) \to (A'_0, A'_1)$, $i_B : (B_0, B_1) \to (B'_0, B'_1)$ and $i_B \circ f = f' \circ i_A$.
\end{notn}

Before giving the formal definition of a SAO, we provide an informal description. For each pair of embeddings $e : A \to (B, C)$ in $\Age(M)$, a standard amalgamation operator $\otimes$ gives an amalgam $B \otimes_A C$ of $e$ such that:

\begin{itemize}
        \item Minimality: $B \otimes_A C$ is generated by the union of the images of $B$ and $C$;
        \item Invariance: given pairs of embeddings $e : A \to (B, C)$, $e' : A' \to (B', C')$ and an isomorphism $i : e \to e'$, we have $B \otimes_A C \cong B' \otimes_A C'$ via an isomorphism that respects $i$;
        \item Transitivity: let $e : A \to (B, C)$ be a pair of embeddings. Then:
        \begin{itemize}
            \item for each embedding $b : B \to B'$ we have $B' \otimes_B (B \otimes_A C) \cong B' \otimes_A C$ via an embedding-respecting isomorphism;
            \item for each embedding $c : C \to C'$ we have $(B \otimes_A C) \otimes_C C' \cong B \otimes_A C'$ via an embedding-respecting isomorphism.
        \end{itemize}
\end{itemize}

We will see in \Cref{SAO mon and assoc} that any SAO is monotonic (given embeddings $B \to B'$, $C \to C'$ we get an embedding $B \otimes_A C \to B' \otimes_A C'$) and associative ($(B \otimes_A C) \otimes_{A'} D \cong B \otimes_A (C \otimes_{A'} D)$).

We also call a SAO \emph{local} if it is only defined for non-empty $A$.

In the formal definition below, we explicitly state which embeddings commute, for the sake of completeness. We then return to a less explicit presentation in the sequel so as to avoid cumbersome notation.

\begin{defn}
    Let $M$ be a \Fr structure. An \emph{amalgamation operator} for $M$ is a map $\otimes$ from the class of pairs of embeddings in $\Age(M)$ with common domain to the class of pairs of embeddings in $\Age(M)$ with common codomain, such that for each pair $(e, f) : A \to (B, C)$, its image $e \otimes f$ is a pair $((e \otimes f)_B, (e \otimes f)_C) : (B, C) \to \cod(e \otimes f)$, and we have $(e \otimes f)_B \circ e = (e \otimes f)_C \circ f$. We write $B \otimes_A C = \cod(e \otimes f)$ for brevity.

    We say that $\otimes$ is a \emph{strong} amalgamation operator if for each pair $(e, f) : A \to (B, C)$ we have \[(e \otimes f)_B(B) \cap (e \otimes f)_C(C) = ((e \otimes f)_B \circ e)(A).\]
\end{defn}

In the below definition, we use the notation for pairs of embeddings from \Cref{emb notn}.

\begin{defn}
    Let $M$ be a \Fr structure with amalgamation operator $\otimes$. We say that $\otimes$ is a \emph{standard amalgamation operator (SAO)} if:
    \begin{itemize}
        \item Minimality: for each pair $(e, f) : A \to (B, C)$, we have that $\cod(e \otimes f)$ is generated by the union of the images of $B$, $C$ in $e \otimes f$;
        \item Invariance (Inv): given pairs of embeddings $(e, f) : A \to (B, C)$, $(e', f') : A' \to (B', C')$ and an isomorphism $i : (e, f) \to (e', f')$, there exists an isomorphism $j : B \otimes_A C \to B' \otimes_{A'} C'$ such that $j \circ (e \otimes f) = (e' \otimes f') \circ i$, as in the below diagram:
        
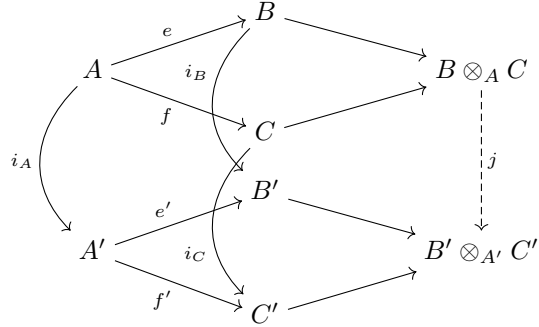
\[\begin{tikzcd}[row sep = tiny, cramped]
	       && B \\
	       \\
	       A &&&& {B \otimes_A C} \\
	       \\
	       && C \\
	       \\
	       && {B'} \\
	       \\
	       {A'} &&&& {B' \otimes_{A'} C'} \\
	       \\
	       && {C'}
	       \arrow[from=1-3, to=3-5]
	       \arrow["{i_B}"'{pos=0.39}, curve={height=24pt}, from=1-3, to=7-3]
	       \arrow["e", from=3-1, to=1-3]
	       \arrow["f"', from=3-1, to=5-3]
	       \arrow["{i_A}"', curve={height=24pt}, from=3-1, to=9-1]
	       \arrow["j", dashed, from=3-5, to=9-5]
	       \arrow[from=5-3, to=3-5]
	       \arrow["{i_C}"'{pos=0.61}, curve={height=24pt}, from=5-3, to=11-3]
	       \arrow[from=7-3, to=9-5]
	       \arrow["e'", from=9-1, to=7-3]
	       \arrow["f'"', from=9-1, to=11-3]
	       \arrow[from=11-3, to=9-5]
        \end{tikzcd}\]
        
        \item Transitivity (Tr): given a pair of embeddings $(e, f) : A \to (B, C)$: 
        \begin{itemize}
            \item for each embedding $b : B \to B'$ there is an isomorphism $i : B' \otimes_B (B \otimes_A C) \to B' \otimes_A C$ such that the left diagram commutes:
            \item for each embedding $c : C \to C'$ there is an isomorphism $j : (B \otimes_A C) \otimes_C C' \to B \otimes_A C'$ such that the right diagram commutes: 
        \end{itemize}
        \[\begin{tikzcd}[sep = tiny, cramped]
	       && {B} && {B'} \\
	       &&& {B' \otimes_B (B \otimes_A C)} \\
	       A && {B \otimes_A C} && {B' \otimes_A C} \\ && {\vphantom{B' \otimes_B (B \otimes_A C)}} &&  \\
	       && {C}
	       \arrow["b", from=1-3, to=1-5]
	       \arrow[from=1-3, to=3-3]
	       \arrow[from=1-5, to=3-5]
	       \arrow[from=1-5, to=2-4]
	       \arrow["i"', dashed, from=2-4, to=3-5]
	       \arrow["e", from=3-1, to=1-3]
	       \arrow["f"', from=3-1, to=5-3]
	       \arrow[from=3-3, to=2-4]
	       \arrow[from=5-3, to=3-3]
	       \arrow[from=5-3, to=3-5]
        \end{tikzcd} \qquad \qquad 
        \begin{tikzcd}[sep = tiny, cramped]
	       && {B} \\
	       &&& {\vphantom{(B \otimes_A C) \otimes_C C'}} \\
	       A && {B \otimes_A C} && {B \otimes_A C'} \\
	       &&& {(B \otimes_A C) \otimes_C C'} \\
	       && {C} && {C'}
	       \arrow[from=1-3, to=3-3]
	       \arrow[from=1-3, to=3-5]
	       \arrow["e", from=3-1, to=1-3]
	       \arrow["f"', from=3-1, to=5-3]
	       \arrow[from=3-3, to=4-4]
	       \arrow["j", dashed, from=4-4, to=3-5]
	       \arrow[from=5-3, to=3-3]
	       \arrow["c"', from=5-3, to=5-5]
	       \arrow[from=5-5, to=3-5]
	       \arrow[from=5-5, to=4-4]
        \end{tikzcd}\]
    \end{itemize}
\end{defn}

\subsection{Equivalence of SWIRs and SAOs} \label{s: SWIR SAO equivalence}

We now show that a SAO induces a SWIR, and vice versa.

\begin{prop} \label{SAO induces SWIR}
    Let $M$ be a \Fr structure with SAO $\otimes$. We define $\ind$ as follows: for $A, B, C \fg M$, define $B \ind_A C$ if there is an isomorphism $ABC \to AB \otimes_A AC$ such that the following diagram commutes (where the hooked arrows in the left diamond are inclusions and the arrows to the right form the standard amalgam):

    
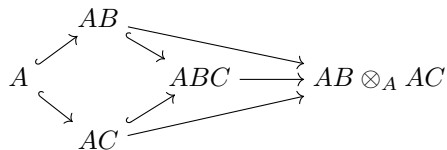
\[\begin{tikzcd}[sep=small, cramped]
	& AB \\
	A && ABC && {AB \otimes_A AC} \\
	& AC
	\arrow[from=2-3, to=2-5]
	\arrow[hook, from=2-1, to=1-2]
	\arrow[hook, from=2-1, to=3-2]
	\arrow[hook, from=3-2, to=2-3]
	\arrow[hook, from=1-2, to=2-3]
	\arrow[from=1-2, to=2-5]
	\arrow[from=3-2, to=2-5]
        \end{tikzcd}\]

        Then $\ind$ is a SWIR on $M$.
\end{prop}
\begin{proof}
    ($\ind$-Inv): suppose $B \ind_A C$. Let $g \in \Aut(M)$. By ($\otimes$-Inv) we have an isomorphism $AB \otimes_A AC \to gAB \otimes_{gA} gAC$, so composing the isomorphisms $gABC \to ABC \to AB \otimes_A AC \to gAB \otimes_{gA} gAC$ we have $gAB \ind_{gA} gBC$.

    ($\ind$-Ex): we show left-(Ex), with right-(Ex) similar. Let $A, B, C \fg M$. We use the extension property of $M$ to embed $AB \otimes_A AC$ over $AC$, giving $AB'C \sub M$, where $B'$ is the image of $B$ under the embedding. Then $B \equiv_A B'$, and by ($\otimes$-Inv) we have $AB' \otimes_A AC \cong AB \otimes_A AC$, so $B' \ind_A C$ as required.

    ($\ind$-Sta): we show left-(Sta). By assumption $B \equiv_A B'$, so by ($\otimes$-Inv) we have $AB \otimes_A AC \cong AB' \otimes_A AC$, and as $B \ind_A C$ and $B' \ind_A C$, we have $ABC \cong AB'C$ via an isomorphism extending $\id_{AC}$, so $B \equiv_{ AC } B'$.

    By \Cref{SWIR tr implies mon}, it now suffices to show ($\ind$-Tr). We show left-($\ind$-Tr), with right-($\ind$-Tr) similar. Suppose $B \ind_A C \wedge B \ind_{AC} D$, so $ABC \cong AB \otimes_A AC$ and $ABCD \cong ABC \otimes_{AC} ACD$. By ($\otimes$-Tr) we have $ABCD \cong AB \otimes_A ACD$, so $B \ind_A CD$.
\end{proof}

\begin{prop} \label{SWIR induces SAO}
    Let $M$ be a \Fr structure with SWIR $\ind$. Define an amalgamation operator $\otimes$ for $M$ as follows. Given a pair of embeddings $(e, f) : A \to (B, C)$, let $(\tld{e}, \tld{f}) : (B, C) \to D$ be an amalgam of $(e, f)$ with $D \sub M$. By ($\ind$-Ex) there is $B' \fg M$ and an isomorphism $g : \tld{e}(B) \to B'$ extending $\id_{\tld{e}e(A)}$ with $B' \ind_{\tld{e}e(A)} \tld{f}(C)$. Let $e \otimes f = (g\tld{e}, \tld{f})$ with $\cod(e \otimes f) = \langle B' \cup \tld{f}(C) \rangle$. Then $\otimes$ is a SAO.
\end{prop}
\begin{proof}
    Minimality is by definition. ($\otimes$-Inv): this follows easily from ($\ind$-Inv), ($\ind$-Sta). We now show left-($\otimes$-Tr), with right-($\otimes$-Tr) similar. By ($\otimes$-Inv) it suffices to consider $A, B, C \fg M$ with $B \otimes_A C = ABC$. We have $B \ind_A C$ by definition of $\otimes$. Let $b : B \to B'$ be an embedding. Then by the extension property for $M$ there is $\tld{B}' \sub M$ and an isomorphism $i : B' \to \tld{B}'$ with $i \circ b = \id_B$ such that $\tld{B}'C \cong B' \otimes_B BC$, and by ($\ind$-Inv) we have $\tld{B}' \ind_B BC$. By ($\ind$-Mon) we have $\tld{B}' \ind_B C$ and thus $\tld{B}' \ind_A C$ by ($\ind$-Tr). So $B' \otimes_B (B \otimes_A C) \cong B' \otimes_A C$.
\end{proof}

\subsection{Properties of SAOs; amalgamation along linear orders}

\begin{lem} \label{SAO mon and assoc}
    Let $M$ be a \Fr structure with SAO $\otimes$. Then $\otimes$ has the following properties:

    \begin{enumerate}[label=(\roman*)]
        \item \label{SAO mon} Monotonicity (Mon): given a pair of embeddings $(e, f) : A \to (B, C)$:
        \begin{itemize}
            \item for each embedding $b : B \to B'$, there is a unique embedding $i : B \otimes_A C \to B' \otimes_A C$ such that the left diagram below commutes (in particular over $B$ and over $C$):
            \item for each embedding $c : C \to C'$, there is a unique embedding $j : B \otimes_A C \to B \otimes_A C'$ such that the right diagram below commutes (in particular over $B$ and over $C$):
        \end{itemize}
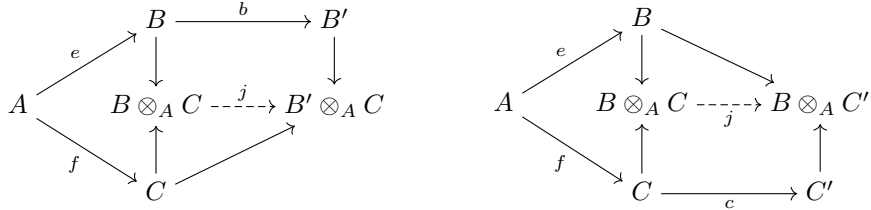
        \[\begin{tikzcd}[sep = small, cramped]
	       && {B} && {B'} \\
	       &&& \\
	       A && {B \otimes_A C} && {B' \otimes_A C} \\ && &&  \\
	       && {C}
	       \arrow["b", from=1-3, to=1-5]
	       \arrow[from=1-3, to=3-3]
	       \arrow[from=1-5, to=3-5]
	       \arrow["i", dashed, from=3-3, to=3-5]
	       \arrow["e", from=3-1, to=1-3]
	       \arrow["f"', from=3-1, to=5-3]
	       \arrow[from=5-3, to=3-3]
	       \arrow[from=5-3, to=3-5]
        \end{tikzcd} \qquad \qquad 
        \begin{tikzcd}[sep = small, cramped]
	       && {B} \\
	       &&& \\
	       A && {B \otimes_A C} && {B \otimes_A C'} \\
	       &&& \\
	       && {C} && {C'}
	       \arrow[from=1-3, to=3-3]
	       \arrow[from=1-3, to=3-5]
	       \arrow["e", from=3-1, to=1-3]
	       \arrow["f"', from=3-1, to=5-3]
	       \arrow["j"', dashed, from=3-3, to=3-5]
	       \arrow[from=5-3, to=3-3]
	       \arrow["c"', from=5-3, to=5-5]
	       \arrow[from=5-5, to=3-5]
        \end{tikzcd}\]
        \item \label{SAO assoc} Associativity (Assoc): given pairs of embeddings $(e, f) : A \to (B, C)$, $(g, h) : A' \to (C, D)$, let $p = ((e \otimes f)_C \circ g) \otimes h$ and $q = e \otimes ((g \otimes h)_C \circ f)$. Then there exists a unique isomorphism $j : \cod(p) \to \cod(q)$ such that the following diagram commutes over $B$, $C$, $D$:
        
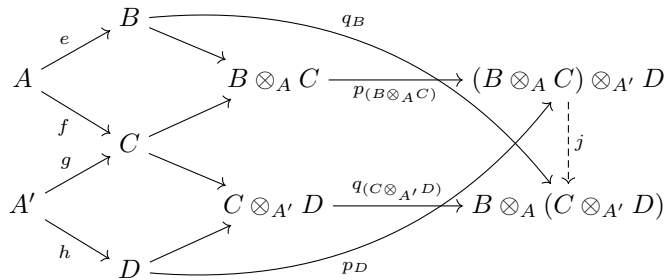
        \[\begin{tikzcd}[cramped, row sep = small]
	& {B} \\
	{A} && {B \otimes_A C} && {(B \otimes_A C) \otimes_{A'} D} \\
	& {C} \\
	{A'} && {C \otimes_{A'} D} && {B \otimes_A (C \otimes_{A'} D)} \\
	& {D}
	\arrow["q_B"{pos=0.4}, curve={height=-34pt}, from=1-2, to=4-5]
	\arrow["p_D"'{pos=0.4}, curve={height=34pt}, from=5-2, to=2-5]
	\arrow["e", from=2-1, to=1-2]
	\arrow["f"', from=2-1, to=3-2]
	\arrow["g", from=4-1, to=3-2]
	\arrow["h"', from=4-1, to=5-2]
	\arrow[from=1-2, to=2-3]
	\arrow[from=3-2, to=2-3]
	\arrow[from=3-2, to=4-3]
	\arrow[from=5-2, to=4-3]
	\arrow["p_{(B \otimes_A C)}"', from=2-3, to=2-5]
	\arrow["q_{(C \otimes_{A'} D)}", from=4-3, to=4-5]
	\arrow["j", dashed, from=2-5, to=4-5]
        \end{tikzcd}\]
    \end{enumerate}

    (Informally:
    \begin{itemize}
        \item (Mon): for $A \to (B, C)$, $B \to B'$, $C \to C'$ there is a unique embedding $B \otimes_A C \to B' \otimes_A C'$;
        \item (Assoc): for $A \to (B, C)$, $A' \to (C, D)$ there is a unique isomorphism $(B \otimes_A C) \otimes_{A'} D \to B \otimes_A (C \otimes_{A'} D)$.)
    \end{itemize}
\end{lem}
\begin{proof}
    In both parts, uniqueness follows by $\otimes$-minimality. \ref{SAO mon}: immediate from (Tr). \ref{SAO assoc}: by (Tr) we have $(B \otimes_A C) \otimes_{A'} D \cong (B \otimes_A C) \otimes_C (C \otimes_{A'} D) \cong B \otimes_A (C \otimes_{A'} D)$, and this isomorphism commutes with the relevant embeddings (we leave the verification of this to the reader).
\end{proof}

\begin{lem} \label{SAO base triviality}
    Let $M$ be a \Fr structure with SAO $\otimes$. Then for each embedding $\zeta : A \to B$ in $\Age(M)$ we have $B \otimes_A A \cong B$ and $A \otimes_A B \cong B$ via unique diagram-respecting isomorphisms.
\end{lem}
\begin{proof}
    This follows immediately by minimality of $\otimes$.
\end{proof}

\begin{lem} \label{M strong implies canon strong}
    Let $M$ be a \Fr structure with SAO $\otimes$. Suppose that $M$ has strong amalgamation. Then $\otimes$ is a strong amalgamation operator.
\end{lem}
\begin{proof}
    This is similar to \Cref{SWIR strong amalg}, and we leave the details to the reader. Alternatively, one can use \Cref{SAO induces SWIR}. 
\end{proof}

\begin{defn} \label{d: n-term amalg}
    Let $M$ be a \Fr structure with SAO $\otimes$. Let $A \in \Age(M)$ and let $n > 2$. For $i < n$ let $\zeta_i : A \to E_i$ be an embedding in $\Age(M)$. We inductively define the \emph{$\otimes$-amalgam of $\zeta_0, \cdots, \zeta_{n-1}$ over $A$}, written $\zeta_0 \otimes \cdots \otimes \zeta_{n-1}$, by $\zeta_0 \otimes \cdots \otimes \zeta_{n-1} := (\zeta_0 \otimes \cdots \otimes \zeta_{n-2}) \otimes \zeta_{n-1}$. We also write $E_0 \otimes_A \cdots \otimes_A E_{n-1}$ for this. Note that by associativity of $\otimes$ (see \Cref{SAO mon and assoc}), we obtain embedding-respecting-isomorphisms between any two bracketings of this $n$-object amalgam, and so it is harmless to drop parentheses in the notation.

    Let $(I, <_\lambda)$ be a finite linear order, and let $\mc{Z} = (\zeta_i \mid i \in I)$ be a family of embeddings in $\Age(M)$ with common domain $A$. We write $\bigotimes^\lambda_A \mc{Z} = \zeta_{i_0} \otimes \cdots \otimes \zeta_{i_{n-1}}$, where $i_0 <_\lambda \cdots <_\lambda i_{n-1}$ is an enumeration of $I$ in strictly increasing $<_\lambda$-order, and call $\bigotimes^\lambda_A \mc{Z}$ the \emph{amalgam of $\mc{Z}$ along $\lambda$}.
\end{defn}

The following \Cref{l: direct limit} is folklore; its proof is straightforward and thus omitted.
\begin{lem} \label{l: direct limit}
    Let $M$ be a \Fr structure. Let $(D, <)$ be a countable directed set, and let $(f_{ij} : C_i \to C_j \mid i, j \in D, i \leq j)$ be a $D$-directed system of embeddings in $\Age(M)$. Then the direct limit $\dlim (C_i)_{i \in D}$ exists and is an element of $\Ao(M)$.
\end{lem}

\begin{lem}
    Let $M$ be a \Fr structure with SAO $\otimes$. Let $(I, <_\lambda)$ be a countable linear order, and let $\mc{Z} = (\zeta_i : A \to E_i)_{i \in I}$ be a family of embeddings in $\Age(M)$ with common domain $A$. For $J \fin I$ write $\mc{Z}_J = (\zeta_j)_{j \in J}$ and let $<_{\lambda_J}$ be the restriction of $<_\lambda$ to $J$.

    Then for $J \sub J' \fin I$, there exists a unique diagram-respecting embedding $\gamma_{J, J'} : \bigotimes^{\lambda_J}_A \mc{Z}_J \to \bigotimes^{\lambda_{J'}}_A \mc{Z}_{J'}$. Thus the embeddings $(\gamma_{J, J'})_{J \sub J' \fin I}$ form a $\Pfin(I)$-directed system, the direct limit of which we denote by $\bigotimes^\lambda_A \mc{Z}$.
\end{lem}
\begin{proof}
    Uniqueness of $\gamma_{J, J'}$ will follow by $\otimes$-minimality, so it suffices to show existence. By induction it suffices to consider the case $J' = J \cup \{i\}$. Let $j_0 <_\lambda \cdots <_\lambda j_{l-1} <_\lambda i <_\lambda j_l <_\lambda \cdots <_\lambda j_{n-1}$ be an enumeration of $J'$ in strictly increasing $<_\lambda$-order, with $j_s \in J$ for $s < n$. (We omit the case where $i$ is the greatest or least element for ease of presentation, but the proof is entirely analogous.)

    Then by \Cref{SAO base triviality}, (Assoc) and (Inv) we have that there is a diagram-respecting isomorphism \[\textstyle{\bigotimes^{\lambda_J}_A \mc{Z}_J} = E_{j_0} \otimes_A \cdots \otimes_A E_{j_{l-1}} \otimes_A E_{j_l} \otimes_A \cdots \otimes_A E_{j_{n-1}} \cong E_{j_0} \otimes_A \cdots \otimes_A E_{j_{l-1}} \otimes_A A \otimes_A E_{j_l} \otimes_A \cdots \otimes_A E_{j_{n-1}}.\] By (Mon), via the embedding $\zeta_i : A \to E_i$ we obtain a diagram-respecting embedding \[E_{j_0} \otimes_A \cdots \otimes_A E_{j_{l-1}} \otimes_A A \otimes_A E_{j_l} \otimes_A \cdots \otimes_A E_{j_{n-1}} \to E_{j_0} \otimes_A \cdots \otimes_A E_{j_{l-1}} \otimes_A E_i \otimes_A E_{j_l} \otimes_A \cdots \otimes_A E_{j_{n-1}} = \textstyle{\bigotimes^{\lambda_{J'}}_A \mc{Z}_{J'}}.\]
\end{proof}
\begin{defn}
    Let $M$ be a \Fr structure with symmetric SAO $\otimes$. Let $\mc{Z} = (\zeta_i : A \to E_i)_{i \in I}$ be a family of embeddings in $\Age(M)$ with common domain $A$, indexed by a countable set $I$. For any two countable linear orders $<_\lambda, <_\mu$ on $I$ and for each $J \fin I$, by the symmetry of $\otimes$ we have a unique diagram-respecting isomorphism $\bigotimes_A^{\lambda_J} \mc{Z}_J \to \bigotimes_A^{\mu_J} \mc{Z}_J$ (here we use the notation of the above lemma), and so there is a unique diagram-respecting isomorphism $\bigotimes_A^\lambda \mc{Z} \to \bigotimes_A^\mu \mc{Z}$. We choose a representative of this isomorphism class and drop the linear order from the notation, writing $\bigotimes_A \mc{Z}$.
\end{defn}

\section{Three key examples} \label{s: key exs}

In this section, we show the following result, giving an answer to \Cref{qn SWIR extensible relationship}:

\begin{thm} \label{t: no implication SWIR Ka}
    There is a relational \Fr structure with a \Ka functor but without a (local) SWIR, and there is a relational \Fr structure with a SWIR but without a universal automorphism group (hence its $\omega$-age is not $\circ$-extensible and it has no \Ka functor).
\end{thm}

This follows from parts \ref{pp two-g} and \ref{pp omit n-antis} of \Cref{p three key exs} below, which collects the results we show in this section for three important examples.

\begin{prop} \label{p three key exs} \hfill
    \begin{enumerate}[label=(\roman*)]
        \item \label{pp 3-ht} The generic $3$-hypertournament has a \Ka functor.
        \item \label{pp two-g} The generic two-graph does not have a local SWIR, but it has a \Ka functor (and hence $\circ$-extensible $\omega$-age).
        \item \label{pp omit n-antis} Let $n \geq 3$. The generic $n$-anticlique-free oriented graph has a SWIR, but its automorphism group is not universal (in particular, the structure does not have a \Ka functor).
    \end{enumerate}    
\end{prop}

We discuss the generic $3$-hypertournament in \Cref{ex: 3-ht}, the generic two-graph in \Cref{ex: two-graph} and the generic $n$-anticlique-free graph in \Cref{ex: n-anticlique-free}; we define each structure in the corresponding section. 

We also show that each structure in the above proposition has a finite SWIR expansion (see \Cref{t: big list of exs}). Note that we do not find a SWIR for the generic $3$-hypertournament itself: the existence of a SWIR for this structure is still open.

The three examples in this section illustrate a \textbf{key idea} which runs through much of this paper: when constructing a \Ka functor, one can add multiple copies of each one-point extension labelled with extra information, and use this extra information (often compared lexicographically) to guarantee a functorial amalgam (conceptually ``canonical" in some sense). In the authors' experience, if this approach does not seem to work, in the process of attempting it one generally comes up with a proof of non-universality of the automorphism group.

Despite some effort, we were not able to produce a general theorem using this idea to cover all the examples we present here, as each example seems to require some specific technical adjustments to the general technique. However, we believe that the key idea should be generally applicable. We informally illustrate the idea on a simple example below, \Cref{ex: Ka for tourn}.

\begin{defn} \label{ope notation}
    Let $M$ be a \Fr structure, and let $A \in \Age(M)$. We define that two one-point extensions $\zeta : A \to B$, $\zeta' : A \to B'$ are \emph{left-isomorphic} if there exists an isomorphism $i : B \to B'$ with $\zeta' = i \circ \zeta$ (we usually omit the prefix ``left-" and just say \emph{isomorphic}). We let $\mc{E}_A$ denote a set of representatives of the isomorphism classes of one-point extensions of $A$.

    Let $\zeta : A \to B$ be a one-point extension, and let $e \in B \setminus \zeta(A)$ with $B = \langle \zeta(A), e \rangle$. For ease of presentation, instead of writing $(\zeta, e)$ for the one-point extension $\zeta$ with distinguished generator $e$, when working with particular examples of relational \Fr structures we will often just write $(A, e)$, conflating the embedding and its image, and we will usually assume that the embedding is set inclusion.
\end{defn}

\begin{notn} \label{sqbrackets relation notn}
    Let $\mc{L} = \{R\}$ with $R$ an $n$-ary relation symbol. Let $A$ be an $\mc{L}$-structure. For $\bar{a} \in A^n$, we write
    \[
    [\bar{a}] =
    \begin{cases}
        0 & \bar{a} \notin R^A;\\
        1 & \bar{a} \in R^A.
    \end{cases}
    \]
    Given a set $A$, we also use this notation to define the relation $R^A$ on $A$: we write $[\bar{a}] = 1$ to mean that we specify $\bar{a} \in R^A$. We will often use this notation to define the relation on an $n$-tuple via the relation on another $n$-tuple (for example, in a tournament, instead of saying that we define $ab$ to have the same orientation as $cd$, we write $[ab] := [cd]$).
\end{notn}

\begin{eg} \label{ex: Ka for tourn}
    Let $M$ be the random tournament. We show that $M$ has a \Ka functor. This result appears in \cite[Example 2.10]{KM17}: its proof inspired much of the present paper. We give the proof from a different viewpoint more amenable to generalisation. We use the following notation: given a tournament $T$ and $v \in T$, we write $v^+ = \{u \in T \mid v \ra u\}$, $v^- = \{u \in T \mid v \la u\}$ and call $v^+$, $v^-$ the \emph{out-neighbourhood} and \emph{in-neighbourhood} of $v$ respectively.

    Let $A$ be a finite tournament. Given two extensions $(A, e)$, $(A, e')$, to define $K(A)$ we must define the orientation of the edge $e e'$, and we need to do this in a functorial way: for each embedding $A \to B$, we must send the extensions of $A$ to extensions of $B$ in such a way that the orientation of $e e'$ is preserved in the image. We do this as follows. To define $K(A)$: for each $(A, e) \in \mc{E}_A$ and each enumeration $\bar{v}$ of $e^-$ in $(A, e)$, place in $K(A)$ a labelled copy $(A, e_{\bar{v}})$ of $(A, e)$ over $A$. (Here the tournament structure is the same, but the extension vertex has an additional label. Note that the enumeration $\bar{v}$ used for the label is arbitrary, and is not defined using the tournament structure of the extension.) We then compare $\bar{v}, \bar{v}'$ lexicographically to define the orientation between labelled extension vertices $e^{}_{\bar{v}}$, $e'_{\bar{v}'}$:
    \begin{itemize}
        \item if $|\bar{v}| < |\bar{v}'|$, define $e^{}_{\bar{v}} \ra e'_{\bar{v}'}$;
        \item if $|\bar{v}| = |\bar{v}'|$, then letting $i$ be least such that $v^{}_i \neq v'_i$, define $[e^{}_{\bar{v}} e'_{\bar{v}'}] = [v^{}_i v'_i]$. (See \Cref{sqbrackets relation notn}.)
    \end{itemize}

    Let $f : A \to B$ be an embedding of finite tournaments. We define $K(f) \supseteq f$ by sending each $e_{\bar{v}} \in K(A) \setminus A$ to the labelled extension vertex in $K(B)$ which has enumerated in-neighbourhood $f(\bar{v})$. It is then immediate that $K(f)$ is a tournament embedding and that $K$ is a \Ka functor.
\end{eg}
\begin{rem}
    Let $M$ be the random tournament. Jaligot showed in \cite{Jal07} that $\Ainf(M)$ is uniquely extensible, by a different method inspired by \cite[Lemma 2.1]{MW92}.
\end{rem}

\subsection{The generic \texorpdfstring{$3$}{3}-hypertournament} \label{ex: 3-ht}

We now show that, with a little more effort, a similar idea to that in \Cref{ex: Ka for tourn} gives a \Ka functor for the $3$-hypertournament.

\begin{defn} \label{d: n-ht}
    Let $n \in \N$, $n \geq 2$. An \emph{$n$-hypertournament} $(T, R)$ is a set $T$ together with an $n$-ary relation $R$ satisfying the condition that each $n$-element substructure of $(T, R)$ has automorphism group equal to the alternating group $A_n$. For example, a $2$-hypertournament is a tournament in the usual sense (an oriented complete graph), and a $3$-hypertournament is a ternary structure where each triple of points $abc$ has exactly one of two possible cyclic orientations: for all $a, b, c \in T$ we have precisely one of $R(a, b, c)$ or $R(b, a, c)$, and also $R(a, b, c) \Rightarrow R(b, c, a) \wedge R(c, a, b)$. If $R(a, b, c)$ holds, we say that $(a, b, c)$ is the \emph{$3$-orientation} of $abc$.

    Let $k \in \N_+$, and let $\bar{n} \in \N^k$ with $n_i \geq 2$ for all $i < k$. An \emph{$\bar{n}$-hypertournament} $(T, (R_i)_{i < k})$ is a set $T$ together with, for each $i$, a relation $R_i$ of arity $n_i$ such that $(T, R_i)$ is an $n_i$-hypertournament. We usually denote an $\bar{n}$-hypertournament $(T, (R_i)_{i < k})$ just by $T$. Let $\mc{T}_{\bar{n}}$ denote the class of finite $\bar{n}$-hypertournaments. It is straightforward to see that $\mc{T}_{\bar{n}}$ has strong amalgamation; denote its \Fr limit by $\mb{T}_{\bar{n}}$.
\end{defn}

\begin{notn}
    For $T$ a tournament, we write the binary relation as $a \ra b$. For $T$ a $3$-hypertournament, we write the ternary relation as $\oa{abc}$.
\end{notn}

We now define several tournament relations that enable us to construct a \Ka functor for $\mb{T}_3$.

\begin{defn} \label{lex tournament}
    Let $V$ be a set. Let $V^{\leq \omega}$ denote the set of sequences of elements of $V$ with sequence length $\leq \omega$, and let $V^{< \omega}$ denote the set of finite sequences of elements of $V$. Let $\ra$ be a tournament relation on $V$. We define the \emph{lexicographic tournament relation} (briefly, lex-tournament relation) $\ra_\lex$ on $V^{\leq \omega}$ induced by $\ra$ as follows. For distinct $\bar{a}, \bar{a}' \in V^{\leq \omega}$:
    \begin{itemize}
        \item if $|\bar{a}| < |\bar{a}'|$, define $\bar{a} \ra_\lex \bar{a}'$;
        \item in the case $|\bar{a}| = |\bar{a}'|$, letting $i$ be least such that $a^{}_i \neq a'_i$, if $a^{}_i \ra a'_i$ define $\bar{a} \ra_\lex \bar{a}'$.
    \end{itemize}
    We also write $(V^{< \omega}, \ra_\lex)$ for the subtournament of $(V^{\leq \omega}, \ra_\lex)$ induced on $V^{< \omega}$, and define $(V^n, \ra_\lex)$ for $n < \omega$ similarly.
\end{defn}

\begin{defn} \label{tourns induced by 3-ht}
    Let $A$ be a $3$-hypertournament and let $a \in A$. Let $V_a = A \setminus \{a\}$. We define a tournament $(V_a, \ra_a)$ as follows: for distinct $u, v \in V_a$, we set $u \ra_a v$ if $(a, u, v) \in R_A$. We then have a lex-tournament relation on $V_a^{<\omega}$ as in \Cref{lex tournament}, which we also denote by $\ra_a$. Let $E_a^{} \sub V_a^2$ be the set of edges of $(V_a^{}, \ra_a^{})$. The lex-tournament relation $\ra_a^{}$ on $V_a^{<\omega}$ induces a tournament relation on $E_a$, which we denote by $\rra_a$. We again have a lex-tournament relation on $E_a^{<\omega}$ induced by $(E_a, \rra_a)$, which we also denote by $\rra_a$.
\end{defn}

\begin{prop} \label{p: 3-ht}
    The generic $3$-hypertournament $\mb{T}_3$ has a \Ka functor.
\end{prop}
\begin{proof}
    Let $A \in \mc{A}(\mb{T}_3)$. For each $(A, e) \in \mc{E}_A$, we put labelled copies of $(A, e)$ over $A$ in $K(A)$ as follows. Let $a \in A$. Recall the notation $V_a, E_a$ from \Cref{tourns induced by 3-ht}. Let 
    \[
        S_a = \{v \in V_a \mid [eav] = 1\},\quad
        T_a = \{(u, v) \in E_a \mid [euv] \neq [auv]\}.
    \]
    For each $a \in A$, each enumeration $\bar{s}$ of $S_a$ and each enumeration $\bar{t}$ of $T_a$, add to $K(A)$ a labelled copy $(A, e_{a, \bar{s}, \bar{t}})$ of $(A, e)$ over $A$. Note that the label $(a, \bar{s}, \bar{t})$ completely specifies the $3$-hypertournament structure up to isomorphism. Let $e_i = e_{a_i, \bar{s}_i, \bar{t}_i}$, $i = 0, 1$ be labelled extension vertices in $K(A)$, and let $v \in A$. We define the $3$-orientation on $v e_0 e_1$ as follows.
    \begin{itemize}
        \item If $v, a_0, a_1$ are distinct, let $[v e_0 e_1] = [v a_0 a_1]$.
        \item If $a_0, a_1$ are distinct and $v = a_i$ for some $i$, let $[v e_i e_{1 - i}] = 1$.
        \item Otherwise $a_0 = a_1$. Let $a = a_0$. We compare $(\bar{s}_0, \bar{t}_0), (\bar{s}_1, \bar{t}_1)$ lexicographically via $(\ra_a, \rra_a)$ (that is, first compare $\bar{s}_0, \bar{s}_1$ via $\ra_a$ and then if equal compare $\bar{t}_0, \bar{t}_1$ via $\rra_a$), and let $i$ be such that $(\bar{s}_i, \bar{t}_i)$ is the start-vertex. Then set $[v e_i e_{1-i}] = 1$.
    \end{itemize}

    Let $e_i = e_{a_i, \bar{s}_i, \bar{t}_i}$, $i \in \Z/3\Z$, be labelled extension vertices in $K(A)$. We define the $3$-orientation on $e_0 e_1 e_2$ as follows.
    \begin{itemize}
        \item If $a_0, a_1, a_2$ are distinct, let $[e_0 e_1 e_2] = [a_0 a_1 a_2]$.
        \item If there is $i \in \Z/3\Z$ with $a_i = a_{i+1}$ and $a_{i+1} \neq a_{i+2}$, let $a = a_i = a_{i+1}$. Compare $(\bar{s}_i, \bar{t}_i), (\bar{s}_{i+1}, \bar{t}_{i+1})$ via $(\ra_a, \rra_a)$, and let $j, k$ be the indices of the start- and end-vertex respectively. Then set $[e_j e_k e_{i+2}] = 1$.
        \item The remaining case is $a_0 = a_1 = a_2$. Let $a = a_0$. We use $(\ra_a, \rra_a)$ to compare $(\bar{s}_i, \bar{t}_i)$, $i \in \Z/3\Z$: this defines a tournament relation $\ra$ on $e_0 e_1 e_2$. Let $i, j, k$ be such that $e_i \ra e_j \ra e_k$, and set $[e_i e_j e_k] = 1$.
    \end{itemize}

    This completes the description of $K(A)$. Let $f : A \to B$ be an embedding in $\mc{A}(\mb{T}_3)$. To define $K(f) : K(A) \to K(B)$, we specify that $K(f)$ extends $f$ and sends each labelled extension vertex $e_{a, \bar{s}, \bar{t}}$ in $K(A)$ to the labelled extension vertex $e_{f(a), f(\bar{s}), f(\bar{t})}$ in $K(B)$. It is then straightforward to check that $K$ is a \Ka functor.
\end{proof}

We now show that there is a finite SWIR expansion of $\mb{T}_3$.

\begin{prop} \label{p: T32 SWIR}
    The generic $3, 2$-hypertournament $\mb{T}_{3, 2}$ is a SWIR expansion of $\mb{T}_3$.
\end{prop}
\begin{proof}
    For $A, B, C \fin \mb{T}_{3,2}$, writing $\tld{B} = B \setminus A$, $\tld{C} = C \setminus A$, we define $B \ind_A C$ if:
    \begin{itemize}
        \item for $b \in \tld{B}$, $c \in \tld{C}$ we have $b \neq c$ and $b \ra c$;
        \item for $b, b' \in \tld{B}$, $c \in \tld{C}$: $b \ra b'$ implies $\oa{bb'c}$;
        \item for $b \in \tld{B}$, $c, c' \in \tld{C}$: $c \ra c'$ implies $\oa{bcc'}$;
        \item for $a \in A$, $b \in \tld{B}$, $c \in \tld{C}$:
        
            $b \ra a \ra c$ implies $\oa{bac}$, $b \la a \la c$ implies $\oa{bca}$, and $(b \ra a \la c \vee b \la a \ra c)$ implies $\oa{abc}$.
        
    \end{itemize}

    It is straightforward to verify that $\ind$ is a SWIR; we leave this to the reader.
\end{proof}
\begin{rem}
    In a separate project, the second and third authors together with Shujie Yang have generalised the above to all $n \geq 2$, giving a finite SWIR expansion of the generic $n$-hypertournament. They then use this to show that $\Aut(\mb{T}_n)$ is simple.
\end{rem}

We end this subsection by observing that it is not the case that all generic $n$-hypertournaments have universal automorphism group:
\begin{prop} \label{p: 2timesodd-ht not univ}
    Let $k \geq 1$ and let $n = 2(2k + 1)$. The automorphism group of the generic $n$-hypertournament $\mb{T}_n$ is not universal.
\end{prop}
\begin{proof}
    By taking $A = \{u, v\}$ and the involution $u \leftrightarrow v$, we see that it is enough to show that $\Aut(\mb{T}_n)$ has no (non-identity) involutions. Suppose for a contradiction that $\tau \in \Aut(\mb{T}_n)$ is an involution, and let $v \in \mb{T}_n$ with $\tau(v) \neq v$. By considering the $n$-edge with vertices $v, \tau(v)$ and any other $n-2$ vertices of $\mb{T}_n$, we see that at least one of these vertices is moved by $\tau$ (as otherwise $\tau$ would be an odd-sign automorphism of the substructure on these $n$ vertices), and so $\tau$ moves infinitely many vertices. Let $U \sub \mb{T}_n$ be such that $|U| = 2k + 1$ and $|U \cup \tau(U)| = n$. Then $\tau$ is an odd-sign automorphism of the substructure on $U \cup \tau(U)$, contradiction.
\end{proof}

\subsection{The generic two-graph} \label{ex: two-graph}

We now give an example with no local SWIR but where the idea of \Cref{ex: Ka for tourn} can still be adapted to give a \Ka functor. Let $\mc{T}$ be the class of finite $3$-hypergraphs $T$ satisfying the condition that the substructure induced on any four vertices of $T$ has an even number of $3$-edges. The elements of $\mc{T}$ are known as finite \emph{two-graphs}. Two-graphs have been extensively studied: see \cite{Sei91}. We first briefly show that $\mc{T}$ is an amalgamation class, following \cite{Sei91} and \cite{AL95}. To do this, we define the \emph{switch} of a graph.

\begin{defn}
    Let $\mc{G}$ denote the class of finite graphs. For $\Gamma \in \mc{G}$ and $U \sub \Gamma$, we define $\sw(\Gamma, U)$, the \emph{switch of $\Gamma$ at $U$}, to be the graph with vertex set $V(\Gamma)$ obtained from $\Gamma$ by switching edges with one vertex in $U$ and one vertex in $U^c$ to non-edges, and vice versa. That is, for $u \in U$, $v \in U^c$, we have $uv \in E(\sw(\Gamma, U))$ iff $uv \notin E(\Gamma)$. (All edges and non-edges whose vertices lie entirely in $U$ or entirely in $U^c$ are left unchanged.)

    It is immediate that $\sw(\Gamma, U) = \sw(\Gamma, U^c)$ and that switching at $U$ is self-inverse. It is straightforward to see that if $\Gamma' = \sw(\Gamma, U)$ and $\Gamma'' = \sw(\Gamma', U')$, then $\Gamma'' = \sw(\Gamma, U \symdiff U')$, and so switching partitions $\mc{G}$ into equivalence classes, which we refer to as \emph{switching-classes}.
\end{defn}

Recall \Cref{sqbrackets relation notn} -- we use this in the below definition.

\begin{defn}
    Let $\Gamma \in \mc{G}$. We define $\tau(\Gamma)$ to be the $3$-hypergraph with vertex set $V = V(\Gamma)$, where for any three vertices $a, b, c \in V$, we add an edge $abc \in E(\tau(\Gamma))$ if $\Gamma$ has an odd number of edges on $\{a, b, c\}$, so $[abc]_{\tau(\Gamma)} \equiv [ab]_\Gamma + [ac]_\Gamma + [bc]_\Gamma \Mod{2}$. For $W \sub V$ with $|W| = 4$, as each pair of vertices is contained in exactly two triples of vertices, we have
    $\sum_{U \in W^{(3)}} [U]_{\tau(\Gamma)} \equiv \sum_{V \in W^{(2)}} 2[V]_\Gamma \equiv 0 \Mod{2}$ and so $\tau(\Gamma)$ is a two-graph. Thus $\tau$ is a map $\mc{G} \to \mc{T}$.
\end{defn}

\begin{lem}[{\cite[Lemma 3.9]{Sei91}}]
    For $\Gamma, \Gamma' \in \mc{G}$, we have $\tau(\Gamma) = \tau(\Gamma')$ iff $\Gamma, \Gamma'$ lie in the same switching-class.
\end{lem}
\begin{proof}[Sketch proof]
    $\Rightarrow:$ Take $v \in \Gamma$. Let $U = \{u \in \Gamma \mid uv \in E(\Gamma) \symdiff E(\Gamma')\}$. Then $\Gamma' = \sw(\Gamma, U)$. $\Leftarrow:$ Switching does not affect the parity of the number of edges of any vertex triple.
\end{proof}

We now give a right inverse for $\tau$.

\begin{defn}
    Let $T \in \mc{T}$, and let $u \in T$. We define the graph $\gamma_u(T)$ to have vertex set $V(T)$ and edge set
    $E(\gamma_u(T)) = \{vw \in V(T)^2 \mid uvw \in E(T)\}$.
    (Note that $u$ is not contained in any edge of $\gamma_u(T)$.) It is easy to check that $\tau(\gamma_u(T)) = T$.
\end{defn}

\begin{lem}[{\cite[Proposition 1]{AL95}}]
    The class $\mc{T}$ of finite two-graphs has the amalgamation property.
\end{lem}
\begin{proof}
    Let $A, B_0, B_1 \in \mc{T}$ with $A = B_0 \cap B_1$. If $A \neq \varnothing$, take $a \in A$. Then $\tau(\gamma_a(B_0) \sqcup_{\gamma_a(A)} \gamma_a(B_1))$ is an amalgam of $B_0, B_1$ over $A$, where $\sqcup_{\gamma_a(A)}$ denotes the free amalgam over $\gamma_a(A)$. If $A = \varnothing$, then take $b_0 \in B_0$, $b_1 \in B_1$, and then $\tau(\gamma_{b_0}(B_0) \sqcup \gamma_{b_1}(B_1))$ is an amalgam of $B_0, B_1$ over $\varnothing$.
\end{proof}

We call the \Fr limit of $\mc{T}$ the \emph{generic two-graph}.

\begin{prop}
    The generic two-graph does not have a local SWIR.
\end{prop}
\begin{proof}
    Suppose for a contradiction that the generic two-graph has a local SWIR. Then $\mc{T}$ has a local SAO $\otimes$. Let $A = \{a, a'\}$, $B_0 = A \cup \{b_0\}$, $B_1 = A \cup \{b_1\}$, where $B_0$ has exactly one $3$-edge $b_0aa'$ and $B_1$ has no $3$-edges. We have $A, B_0, B_1 \in \mc{T}$. Then $B_0 \otimes_A B_1 \in \mc{T}$ must have exactly one additional $3$-edge $b_0b_1a$ or $b_0b_1a'$, violating (Inv) via the automorphism $a \leftrightarrow a'$ of $A$.
\end{proof}

It is however immediate that there is an expansion of the generic two-graph with a SIR: as the generic two-graph is a reduct of the random graph, we may consider the expansion of the random graph by the induced generic two-graph, which has a SIR induced by the SIR for the random graph.

\begin{prop} \label{p: Ka for two-graph}
    The generic two-graph has a \Ka functor.
\end{prop}
\begin{proof}
    Let $A \in \mc{T}$. For each $u \in A$, let $T_u$ be a two-graph containing $A$ obtained by applying $\tau$ to a free amalgam over $\gamma_u(A)$ of all $\gamma_u((A, e))$ for $(A, e) \in \mc{E}_A$. Note that as $\tau$ is a left inverse of $\gamma_u$, each $T_u$ contains a copy $(A, e_u)$ of each $(A, e) \in \mc{E}_A$. We take the $T_u$ to intersect only in $A$, and let $T_A$ be the $3$-hypergraph union over $A$ of the $T_u$, $u \in A$. We define $K(A)$ by adding further $3$-hyperedges to $T_A$ as follows. 
    
    In the below, Latin letters denote elements of $A$ and Greek letters with subscripts denote extension vertices, so for example $\alpha_u$ denotes an extension vertex in $T_u \setminus A$. Distinct letters denote distinct vertices (so for example $\alpha_u \neq \beta_u$, and when we write $\alpha_u \beta_v w$ we assume $w \neq u, v$). We use \Cref{sqbrackets relation notn} for $3$-hyperedges of $T_A$ and to define $3$-hyperedges of $K(A)$, and addition is (mod $2$).

    We extend $T_A$ to $K(A)$ via the following rules:
    \begin{align*}
        [\alpha_u \beta_v u] &:= [\beta_v u v]; &
        [\alpha_u \beta_v w] &:= [uvw] + [\alpha_u uw] + [\beta_v vw];\\
        [\alpha_u \beta_u \gamma_v] &:= [\alpha_u \beta_u u]; &
        [\alpha_u \beta_v \gamma_w] &:= [uvw].
    \end{align*}
    Note that each $T_u$ is a substructure of $K(A)$; the above rules only add $3$-hyperedges including vertices from at least two distinct $T_u, T_v$. Also note that we always have $[\alpha_u \beta_u u] = 0$ by definition; we write the definition as above so as to clarify the below algebraic manipulations.

    We now check that $K(A)$ is a two-graph. Below, we use the following notation: for $|W| = 4$, we define $[[W]] = \sum_{U \in W^{(3)}} [U]$. We now check that for all $W \sub K(A)$, $|W| = 4$, we have $[[W]] \equiv 0 \Mod{2}$. If $W \sub T_u$ for some $u \in A$ then this is immediate, as $T_u$ is a two-graph and a substructure of $K(A)$. We check the remaining cases:
    \begin{align*}
    [[\alpha_u \beta_v u v]] &= 2[\beta_v u v] + 2[\alpha_u u v] \equiv 0 \Mod{2}; &
    [[\alpha_u \beta_v u x]] &= [[\beta_v u v x]] \equiv 0 \Mod{2};\\
    [[\alpha_u \beta_v x y]] &= [[uvxy]] \equiv 0 \Mod{2}; &
    [[\alpha_u \beta_u \gamma_v u]] &= 2[\alpha_u \beta_u u] + 2[\gamma_v u v] \equiv 0 \Mod{2};\\
    [[\alpha_u \beta_u \gamma_v v]] &= [[\alpha_u \beta_u u v]] \equiv 0 \Mod{2}; &
    [[\alpha_u \beta_u \gamma_v x]] &= [[\alpha_u \beta_u u x]] \equiv 0 \Mod{2};\\
    [[\alpha_u \beta_v \gamma_w u]] &= 2[uvw] + 2[\beta_v u v] + 2[\gamma_w u w] \equiv 0 \Mod{2}; &
    [[\alpha_u \beta_v \gamma_w x]] &= [[uvwx]] \equiv 0 \Mod{2};\\
    [[\alpha_u \beta_u \gamma_u \delta_v]] &= [[\alpha_u \beta_u \gamma_u u]] \equiv 0 \Mod{2};
    & [[\alpha_u \beta_u \gamma_v \delta_v]] &= 2[\alpha_u \beta_u u] + 2[\gamma_v \delta_v v] \equiv 0 \Mod{2};\\
    [[\alpha_u \beta_u \gamma_v \delta_w]] &= 2[\alpha_u \beta_u u] + 2[uvw] \equiv 0 \Mod{2}; &
    [[\alpha_u \beta_v \gamma_w \delta_x]] &= [[uvwx]] \equiv 0 \Mod{2}.
    \end{align*}

    Now let $f : A \to B$ be an embedding in $\mc{T}$. We define $K(f) \supseteq f$ as follows. Let $(A, e_u)$ be a labelled extension in $K(A)$. Let $W$ denote the free amalgam of $\gamma_u(A) \hookrightarrow \gamma_u(A, e_u)$ and $\gamma_u(f) : \gamma_u(A) \to \gamma_{f(u)}(B)$, which we take as extending $\gamma_{f(u)}(B)$ by an extension vertex $e'_{f(u)}$. We define $K(f)(e_u)$ to be the extension vertex $e'_{f(u)} \in K(B)$ of the labelled extension $(B, e'_{f(u)})$ isomorphic to $\tau(W)$.

    To show that $K$ respects composition, it suffices to show that $\gamma^{}_{f(u)}((B, e'_{f(u)})) = W$. Let $v \in B \setminus f(A)$. Then as $f(u)$ is isolated in $\gamma_{f(u)}(B)$, there is no edge $f(u)v$ in $\gamma_{f(u)}(B)$, and similarly there is no edge $ue_u$ in $\gamma_u(A, e_u)$. So in $W$ we have that $f(u)ve'_{f(u)}$ is an anticlique, and so $f(u)ve'_{f(u)}$ is not a $3$-hyperedge of $(B, e'_{f(u)}) \cong \tau(W)$. Therefore $\gamma^{}_{f(u)}((B, e'_{f(u)})) = W$ as required.
\end{proof}

\subsection{SWIR but non-universal automorphism group: \texorpdfstring{$n$-anticlique}{n-anticlique}-free oriented graphs} \label{ex: n-anticlique-free}

We now give a particularly straightforward example of a \Fr structure with a SWIR but without a universal automorphism group. Let $n \geq 3$. We let $I_n$ denote an anticlique of size $n$. Let $\mc{C}$ be the \emph{class of finite $n$-anticlique-free oriented graphs}: that is, the class of finite oriented graphs $A$ such that $I_n$ does not embed into $A$ as a substructure (as an induced subgraph, in graph theory terminology). For $A, B, C \in \mc{C}$ with $B \cap C = A$, we define $B \otimes_A C$ to be the structure on $B \cup C$ extending $B, C$ by adding directed edges $(b, c)$ for each $b \in B \setminus A$, $c \in C \setminus A$. It is straightforward to check that $\otimes$ is a SAO. Let $M = \FrLim(\mc{C})$. We call $M$ the \emph{generic $n$-anticlique-free oriented graph}.

We now show that $\Aut(M)$ is not universal (and therefore $M$ has no \Ka functor). We will show that $\Aut(M)$ does not contain $n-1$ pairwise-commuting involutions -- it is straightforward to produce $A \in \mc{C}$ such that $\Aut(A)$ contains such a set of involutions (for example, take $A = \{a_{i, j} \mid i < 2, j < n-1\}$ with $a_{i, j} \ra a_{i', j'}$ for all $i, i'$ and $j < j'$, and the involutions $\tau_j : a_{0, j} \leftrightarrow a_{1, j}$ for $j < n-1$), and thus this suffices to show that $\Aut(M)$ is not universal. (Note that when we use the term involution, we exclude the identity.)

\begin{lem} \label{l: tau fixed pts}
    For each $\tau \in \Aut(M)$, let $F_\tau$ denote the set of fixed points of $\tau$. Let $\tau \in \Aut(M)$ be an involution. Then $F_\tau$ is $I_{n-2}$-free.
\end{lem}
\begin{proof}
    Suppose not, and let $U \sub F_\tau$, $U \cong I_{n-2}$. Let $b \in M$ with $\tau(b) \neq b$. By the extension property of $M$, there exists $e \in M \setminus (U \cup \{b, \tau(b)\})$ with $e \ra b, \tau(b) \ra e$ and no edge between $e$ and $U$. So $\tau(e) \neq e$, and as $\tau$ fixes $U$, we have that there is no edge between $\tau(e)$ and $U$. As $\tau^2 = \id_M$, there is no edge between $e$ and $\tau(e)$. But then $e, \tau(e), U$ form an anticlique of size $n$ -- contradiction.
\end{proof}
\begin{lem} \label{l: anticlique in cover}
    Let $F_0, \cdots, F_{k-1}$ be a finite cover of $M$. Let $r < n$. Then some $F_j$ contains $I_r$.
\end{lem}
\begin{proof}
    It suffices to consider the case where the $F_i$ are disjoint. The $F_i$ then give a vertex-colouring of $M$. As the complement of the unoriented graph reduct of $M$ is the generic $K_n$-free graph, the lemma then follows from the vertex-Ramsey property for the generic $K_n$-free graph, which gives that any vertex-colouring of the generic $K_n$-free graph contains a monochromatic copy of $K_r$. (The vertex-Ramsey property for the generic $K_n$-free graph is Folkman's theorem, originally in \cite{Fol70}; a clear exposition following \cite{Kom86} is \cite[Theorem 12.3]{Pro13}. Alternatively, we may use the \NR theorem.)
\end{proof}

\begin{prop} \label{p: I_n-free not univ}
    $\Aut(M)$ is not universal.
\end{prop}

\begin{proof}
    As noted above, it suffices to show that $\Aut(M)$ does not contain $n - 1$ pairwise-commuting involutions. 
    
    Suppose for a contradiction that $\Aut(M)$ contains $n - 1$ pairwise-commuting involutions $\{\tau_i \mid i < n - 1\}$. Then for distinct $i, j$ each $\tau_i \tau_j$ is also an involution. By \Cref{l: tau fixed pts}, each $F_{\tau_i}$ and each $F_{\tau_i \tau_j}$ is $I_{n-2}$-free, so by \Cref{l: anticlique in cover} the $F_{\tau_i}$ and $F_{\tau_i \tau_j}$ do not cover $M$, and hence there is $v \in M$ with $v, \tau_0(v), \cdots, \tau_{n-2}(v)$ distinct. As each $\tau_i$ and $\tau_i \tau_j$ is an involution, we have that $v, \tau_0(v), \cdots, \tau_{n-2}(v)$ is an $n$-vertex anticlique -- contradiction.
\end{proof}

\begin{rem} \label{r: I_n-free BK strictly weaker}
    As remarked above \Cref{l: tau fixed pts}, there is $A \fin M$ such that $\Aut(A)$ contains $n-1$ pairwise-commuting involutions. By \Cref{p: I_n-free not univ}, the inclusion $A \hookrightarrow M$ is not $\circ$-extensive, but each automorphism of $A$ does extend to an automorphism of $M$, as $A$ is finite and $M$ is ultrahomogeneous. This example shows that our notion of $\circ$-extensive embedding is distinct from the notion of ``extensible embedding" of Bart\v{o}s and Kubi\'{s} -- see \Cref{r: BK}.
\end{rem}

\section{Independence relations, ordered structures and tournament structures} \label{s: order and tourn exps}

In Sections \ref{s: strs with SIR}--\ref{s: SIR tourn free superpos} we now give positive results for \Cref{qn extensible}, giving the existence of \Ka functors for a range of \Fr structures.

\subsection{Structures with a SIR} \label{s: strs with SIR}

Recall that a SIR is defined to be a symmetric SWIR. The main theorem of \cite{Mul16}, Theorem 4.9, states that any \Fr structure with a local SIR has a universal automorphism group. In this subsection, we show that this result follows quickly using the machinery of standard amalgamation operators and \Ka functors. (In \cite[Lemmas 4.4, 4.5]{Mul16} one already sees the idea of canonical amalgamation -- the novelty here is to observe that the general \Ka functor framework enables a quick proof.)

\begin{prop}[see {\cite[Theorem 4.9]{Mul16}}] \label{p: quick Mue}
    Let $M$ be a \Fr structure with strong amalgamation and a local SIR. Then $M$ has a \Ka functor, and hence has $\circ$-extensible $\omega$-age and a universal automorphism group.
\end{prop}
\begin{proof}
    Let $\otimes$ be the symmetric local SAO for $M$ induced by the local SIR on $M$. For each $A \in \Age(M)$, let $K(A) = \bigotimes_A \mc{E}_A$. (As $\otimes$ is symmetric, we do not need to specify a linear order along which to perform the amalgam.)
    
    Let $\theta : A \to B$ be an embedding in $\Age(M)$. We first define a map $\dot\theta : \mc{E}_A \to \mc{E}_B$: for $\zeta : A \to E$ in $\mc{E}_A$, let $\dot\theta(\zeta) : B \to E'$ be the unique element of $\mc{E}_B$ with an isomorphism $E' \to E \otimes_A B$ making the below diagram commute:
    \[\begin{tikzcd}[sep = small, cramped]
	&&& {E \otimes_A B} \\
	E && {E'} \\
	\\
	A && B
	\arrow["(\zeta \otimes \theta)_E", curve={height=-14pt}, from=2-1, to=1-4]
	\arrow["", from=2-1, to=2-3]
	\arrow["\cong", from=2-3, to=1-4]
	\arrow["\zeta^{}"{pos=0.4}, from=4-1, to=2-1]
	\arrow["\theta"', from=4-1, to=4-3]
	\arrow["(\zeta \otimes \theta)_B"', curve={height=14pt}, from=4-3, to=1-4]
	\arrow["{\dot\theta(\zeta)}"'{pos=0.43}, from=4-3, to=2-3]
    \end{tikzcd}\]
    
    To define an embedding $K(\theta) : K(A) \to K(B)$, we define coherent embeddings $\eps_{\mc{F}} : \bigotimes_A \mc{F} \to K(B)$ for each $\mc{F} \fin \mc{E}_A$, and then induce $K(\theta)$ via the universal property of the direct limit.

    Let $\mc{F} = \{\zeta_i : A \to E_i \mid i < n\} \fin \mc{E}_A$. By (Tr) we have \[(E_0 \otimes_A B) \otimes_B \cdots \otimes_B (E_{n-1} \otimes_A B) \cong (E_0 \otimes_A \cdots \otimes_A E_{n-1}) \otimes_A B,\] so there is a unique diagram-respecting embedding $\bigotimes_A \mc{F} \to \bigotimes_B \dot\theta(\mc{F})$ (here we use symmetry of $\otimes$), and we define $\eps_{\mc{F}}$ by composing this with the direct-limit embedding $\bigotimes_B \dot\theta(\mc{F}) \to K(B)$. Uniqueness of the above diagram-respecting embedding gives coherence of the family $\{\eps_{\mc{F}} \mid \mc{F} \fin \mc{E}_A\}$, so we obtain an embedding $K(\theta) : K(A) \to K(B)$ by the universal property of the direct limit. It is then straightforward to check that $K$ is a \Ka functor, and we leave this to the reader.
\end{proof}

\subsection{Free superpositions}

\begin{defn}[{\cite[Definition 3.21]{Bod15}}] \label{free sup}
    Let $\mc{C}_0$, $\mc{C}_1$ be relational \Fr classes in languages $\mc{L}_0$, $\mc{L}_1$ with disjoint signatures, and suppose that $\mc{C}_0$, $\mc{C}_1$ have strong amalgamation. Let $\mc{L} = \mc{L}_0 \cup \mc{L}_1$. We define the \emph{free superposition} $\mc{C}_0 \ast \mc{C}_1$ to be the class of finite $\mc{L}$-structures $A$ such that for $i = 0, 1$, the $\mc{L}_i$-reduct $A|_{\mc{L}_i}$ is in $\mc{C}_i$.

    Let $M_i = \FrLim(\mc{C}_i)$ for $i = 0, 1$. It is straightforward to see that $\mc{C}_0 \ast \mc{C}_1$ also has strong amalgamation, and we denote its \Fr limit by $M_0 \ast M_1$ and call $M_0 \ast M_1$ the \emph{free superposition} of $M_0$ and $M_1$. It is also straightforward to see that $(M_0 \ast M_1)|_{\mc{L}_i} \cong M_i$ for $i = 0, 1$, and so we will assume $(M_0 \ast M_1)|_{\mc{L}_i} = M_i$ when working with free superpositions.
\end{defn}

\begin{defn} \label{d: gen order exp}
    Let $\mc{C}$ be a relational \Fr class with strong amalgamation. Let $\mc{LO}$ denote the class of finite linear orders, and let $\mc{T}$ denote the class of finite tournaments. We call the \Fr limit of $\mc{C} \ast \mc{LO}$ the \emph{generic order expansion} of $\FrLim(\mc{C})$, and we call the \Fr limit of $\mc{C} \ast \mc{T}$ the \emph{generic tournament expansion} of $\FrLim(\mc{C})$.
\end{defn}

\subsection{Ordered structures with a SWIR} \label{s: ordered strs with SWIR}

In this section we prove the following:

\begin{thm} \label{t: order exp}
    Let $M$ be a linearly ordered \Fr structure with strong amalgamation and a local SWIR. Then $M$ has a \Ka functor, and hence its $\omega$-age is $\circ$-extensible.
\end{thm}

Note that here we do not assume that $M$ is a generic order expansion of another \Fr structure -- we only assume that $M$ has a linear order. For an example of an ordered \Fr structure which is not a generic order expansion, see \Cref{ex: ordered SWIR strs}\ref{LO extending PO} below.

\begin{eg} \label{ex: ordered SWIR strs}
    It is straightforward to see that the following structures are ordered structures with a local SWIR, and so by \Cref{t: order exp} they have \Ka functors:
    \begin{enumerate}[label=(\roman*)]
        \item the generic ordered graph;
        \item the generic ordered rational Urysohn space;
        \item the generic $n$-linear order (the free superposition of $n$ instances of $(\Q, <)$);
        \item \label{LO extending PO} the \Fr limit of the class of finite linearly ordered posets where the linear order extends the partial order.
    \end{enumerate}
\end{eg}

We give several definitions needed for the proof of \Cref{t: order exp}.

\begin{defn} \label{lex order}
    Let $V$ be a set. Let $V^{\leq \omega}$ denote the set of sequences of elements of $V$ with sequence length $\leq \omega$. Let $V^{< \omega}$ denote the set of finite sequences of elements of $V$. Let $<$ be a linear order on $V$. We define the \emph{lexicographic order} $\prec$ on $V^{\leq \omega}$ induced by $<$ as follows. For distinct $\bar{a}, \bar{a}' \in V^{\leq \omega}$:
    \begin{itemize}
        \item if $|\bar{a}| < |\bar{a}'|$, define $\bar{a} \prec \bar{a}'$;
        \item in the case $|\bar{a}| = |\bar{a}'|$, letting $i$ be least such that $a^{}_i \neq a'_i$, if $a^{}_i < a'_i$ define $\bar{a} \prec \bar{a}'$.
    \end{itemize}

    We also write $(V^{< \omega}, \prec)$ for the suborder of $(V^{\leq \omega}, \prec)$ induced on $V^{< \omega}$. It is straightforward to check that if $(V, <)$ is a well-order, then so is $(V^{< \omega}, \prec)$. (Note that $(V^{\leq \omega}, \prec)$ is not a well-order in the non-trivial case $|V| \geq 2$.)
\end{defn}

\begin{defn}
    Let $(V, <)$ be a linear order. In the case that $(V, <)$ is a well-order, we define the \emph{power-set order} $<_\set$ on $2^V$ induced by $<$ as follows: for $f, f' \in 2^V$, define $f <_\set f'$ if $f(\min\{v \in V \mid f(v) \neq f'(v)\}) = 0$, where the minimum is taken in the order $<$ on $V$. Note that $<_\set$ is a linear order, but need not be a well-order in general.

    Let $\Fin(V) = \{f \in 2^V \mid f(v) = 1 \text{ for only finitely many } v \in V\}$. In general, even if $(V, <)$ is not a well-order, we may define $<_\set$ on $\Fin(V)$ as above. ($\Fin(V)$ corresponds to the set of finite subsets of $V$, and we pass between subsets and characteristic functions without comment.)
\end{defn}

\begin{defn} \label{rel order}
    Let $(W, <)$ be a well-order, and let $\prec$ be the lexicographic order on $W^{< \omega}$ induced by $<$. As observed above, we have that $\prec$ is a well-order. Let $\prec_\set$ be the power-set order on $2^{W^{< \omega}}$ induced by $\prec$.
    
    Let $\mc{L} = \{R_i \mid i < I\}$ be a first-order relational language with $I \leq \omega$, and let $\RelStr(W)$ be the set of $\mc{L}$-structures on $W$, viewed as a subset of $\prod_{i < I} 2^{W^{\ar(R_i)}} \sub {(2^{W^{< \omega}})}^{\leq \omega}$. Let $\ll$ be the lexicographic order on ${(2^{W^{< \omega}})}^{\leq \omega}$ induced by $\prec_\set$. We define the \emph{relational order} $<_\rel$ on $\RelStr(W)$ to be the suborder of $\ll$ induced on $\RelStr(W)$.
\end{defn}

\begin{defn} \label{str order}
    Let $\mc{L}$ be a first-order language containing a binary relation symbol $\leq$, where $\mc{L}$ may also include other relation, function and constant symbols. We call an $\mc{L}$-structure $M$ an \emph{ordered $\mc{L}$-structure} if $\leq^M$ is a linear order.

    Let $M$ be an ordered \Fr $\mc{L}$-structure with strong amalgamation and a local SWIR $\ind$. Let $\otimes$ be the local SAO induced by $\ind$. For each $A \in \mc{A}(M)$, define a countable set $\mc{F}_A$ of labelled one-point extensions as follows: for each $\zeta : A \to E$ in $\mc{E}_A$, each $S \fin A$ and each $e \in E \setminus \zeta(A)$ such that $E \cong \langle \zeta(S), e \rangle \otimes_{\langle S \rangle} A$, where the isomorphism makes the below diagram commute, put the labelled extension $(\zeta, S, e)$ in $\mc{F}_A$.
    \[\begin{tikzcd}[sep=small, cramped]
	&&& {\langle \zeta(S), e \rangle \otimes_{\langle S \rangle} A} \\
	{\langle \zeta(S), e \rangle} && E \\
	\\
	S && A
	\arrow[curve={height=-5pt}, from=2-1, to=1-4]
	\arrow[hook, from=2-1, to=2-3]
	\arrow["\cong"'{pos=0.3}, from=2-3, to=1-4]
	\arrow["\zeta", from=4-1, to=2-1, yshift=0.1em]
	\arrow[hook, from=4-1, to=4-3]
	\arrow[curve={height=5pt}, from=4-3, to=1-4]
	\arrow["\zeta"', from=4-3, to=2-3, shorten=0.11em]
    \end{tikzcd}\]
    For each $A \in \mc{A}(M)$ we define a linear order $<_{\sigma_A}$ on $\mc{F}_A$, which we call the \emph{structural order}. First, let $x_0, x_1, \cdots$ be an enumeration of the variables of $\mc{L}$, and let $t_0, t_1, \cdots$ be an enumeration of the terms of $\mc{L}$. We will use the same enumerations of variables and terms when defining $<_{\sigma_A}$ for each $A \in \mc{A}(M)$. Let $(\zeta, S, e), (\zeta', S', e') \in \mc{F}_A$ be distinct, and let $E = \cod(\zeta)$, $E' = \cod(\zeta')$. Compare $(\zeta, S, e), (\zeta', S', e')$ in $<_{\sigma_A}$ as follows. 
    \begin{itemize}
        \item First compare $S$, $S' \in \Fin(A)$ using the order $<_\set$ induced by the order $<$ of the structure $A$.
        \item Next, if $S = S'$, let $n = |S|$ and let $\bar{s}$ be the enumeration of $S$ in strictly increasing $<$-order. Let $V$ be the set of terms of $\mc{L}$ with variables amongst $x_0, \cdots, x_n$. Let $C = \langle \zeta(S), e \rangle \sub E$ and $C' = \langle \zeta'(S), e' \rangle \sub E'$. Define an equivalence relation $\sim$ on $V$ by $t(x_0, \cdots, x_n) \sim u(x_0, \cdots, x_n)$ if $t^C(\zeta(\bar{s}), e) = u^C(\zeta(\bar{s}), e)$, and define $\sim'$ similarly (with $\zeta'$, $e'$ and $C'$). The enumeration of the set of terms of $\mc{L}$ induces an enumeration $v_0(\bar{x}), v_1(\bar{x}), \cdots$ of $V$, and hence using \Cref{rel order} we may compare $\sim$, $\sim'$ in the relation order $<_\rel$ induced by the enumeration on $V$.
        \item Finally, if $\sim$, $\sim'$ are equal, then letting $\tld{\mc{L}}$ be the reduct of $\mc{L}$ to its function and constant symbols, we have that the map $\zeta(\bar{s}) \mapsto \zeta'(\bar{s})$, $e \mapsto e'$ extends to an $\tld{\mc{L}}$-isomorphism $C|_{\tld{\mc{L}}} \to C' |_{\tld{\mc{L}}}$. We may thus identify $C|_{\tld{\mc{L}}}$, $C'|_{\tld{\mc{L}}}$ via the $\tld{\mc{L}}$-isomorphism and compare $C$, $C'$ by considering these as two $\mc{L}$-structures on the same domain $D$, where $C|_{\tld{\mc{L}}} = C'|_{\tld{\mc{L}}} = \langle S, e \rangle$, but where the relational reducts of $C, C'$ differ. We enumerate $D$ using the enumeration of $V$: for each $d \in D$, let $\xi(d)$ be the least $i$ with $d = v_i(\bar{s}, e)$. We then have that $(D, \xi)$ is a well-order (with order type $\omega$), and we use the $\xi$-induced relational order $<_\rel$ from \Cref{rel order} to compare the relational reducts of $C, C'$. This completes the definition of $<_{\sigma_A}$.
    \end{itemize}
\end{defn}

\begin{defn} \label{order labelled extension map}
    Let $M$ be an ordered \Fr structure with strong amalgamation and a local SWIR $\ind$. Let $\otimes$ be the local SAO induced by $\ind$. Let $\theta : A \to B$ be an embedding in $\mc{A}(M)$. We define a map $\dot\theta : \mc{F}_A \to \mc{F}_B$, the \emph{labelled extension map induced by $\theta$}, as follows.
    
    Let $\gamma = (\zeta, S, e) \in \mc{F}_A$ and let $E = \cod(\zeta)$. Let $\zeta' : B \to E'$ be the unique element of $\mc{E}_B$ with an embedding-respecting isomorphism to $E \otimes_A B$, the codomain of $\zeta \otimes \theta$, and let $\theta_\zeta : E \to E'$ be the embedding corresponding under the isomorphism to $(\zeta \otimes \theta)_E$ -- see the below diagram. We define $\dot\theta(\gamma) = (\zeta', \theta(S), \theta_\zeta(e))$. 

    \[\begin{tikzcd}[sep = small, cramped]
	&&& {E \otimes_A B} \\
	E && {E'} \\
	\\
	A && B
	\arrow["{(\zeta \otimes \theta)_E}", curve={height=-14pt}, from=2-1, to=1-4]
	\arrow["{\theta_\zeta}", from=2-1, to=2-3]
	\arrow["\cong", from=2-3, to=1-4]
	\arrow["\zeta^{}"{pos=0.4}, from=4-1, to=2-1]
	\arrow["\theta"', from=4-1, to=4-3]
	\arrow["(\zeta \otimes \theta)_B"', curve={height=14pt}, from=4-3, to=1-4]
	\arrow["{\zeta'}"'{pos=0.43}, from=4-3, to=2-3]
    \end{tikzcd}\]
    
    We have $E \otimes_A B \cong (\langle \zeta(S), e \rangle \otimes_{\langle S \rangle} A) \otimes_A B \cong \langle \zeta(S), e \rangle \otimes_{\langle S \rangle} B$ by (Tr), and \[\langle \zeta(S), e \rangle \otimes_{\langle S \rangle} B \cong \langle \theta_\zeta(\zeta(S)), \theta_\zeta(e) \rangle \otimes_{\langle \theta(S) \rangle} B = \langle \zeta'(\theta(S)), \theta_\zeta(e) \rangle \otimes_{\langle \theta(S) \rangle} B\] by (Inv) and the definition of $\zeta'$, so $\dot\theta(\gamma) \in \mc{F}_B$. By the above and the fact that $\theta$, $\theta_\zeta$ are embeddings, we have that for all $\gamma_0 = (\zeta_0, S_0, e_0), \gamma_1 = (\zeta_1, S_1, e_1) \in \mc{F}_A$, if $\gamma_0 <_{\sigma_A} \gamma_1$ then $\dot\theta(\gamma_0) <_{\sigma_B} \dot\theta(\gamma_1)$ -- that is, $\dot\theta$ is a $<_\sigma$-order-embedding $\mc{F}_A \to \mc{F}_B$.
\end{defn}

\begin{proof}[Proof of \Cref{t: order exp}]
    Let $\otimes$ be the local SAO induced by the local SWIR on $M$.

    Let $A \in \mc{A}(M)$. Define $K(A) = \bigotimes_A^{\sigma_A} \mc{F}_A$, where $\bigotimes_A^{\sigma_A} \mc{F}_A$ is the standard amalgam over $A$ of $\mc{F}_A$ along the structure order $<_{\sigma_A}$. (Here we mildly abuse notation: formally we take the amalgam along $\sigma_A$ of the family of extensions $(\zeta_\gamma)_{\gamma \in \mc{F}_A}$, where for each $\gamma \in \mc{F}_A$ we write $\zeta_\gamma$ for the embedding given by the first coordinate of $\gamma$.)

    The one-point extension property \ref{Ka ope} in the definition of \Ka functor is clear. Let $\theta : A \to B$ be an embedding in $\mc{A}(M)$. We define an embedding $K(\theta) : K(A) \to K(B)$ by defining coherent embeddings $\eps_{\tld{\mc{F}}} : \bigotimes_A^{\sigma_A|_{\tld{\mc{F}}}} \tld{\mc{F}} \to K(B)$ for each $\tld{\mc{F}} \fin \mc{F}_A$ and then inducing $K(\theta)$ from the family $\{\eps_{\tld{\mc{F}}} \mid \tld{\mc{F}} \fin \mc{F}_A\}$ using the universal property of the direct limit $K(A)$.

    Let $\tld{\mc{F}} = \{\gamma_i \mid i < n\} \fin \mc{F}_A$ with $\gamma_i = (\zeta_i, S_i, e_i)$, $E_i = \cod(\zeta_i)$ for all $i < n$ and $\gamma_0 <_{\sigma_A} \cdots <_{\sigma_A} \gamma_{n-1}$. For each $\gamma_i$ we have $\dot\theta(\gamma_i) \in \mc{F}_B$, and we have $\dot\theta(\gamma_0) <_{\sigma_B} \cdots <_{\sigma_B} \dot\theta(\gamma_{n-1})$. By (Tr) we have \[(E_0 \otimes_A B) \otimes_B \cdots \otimes_B (E_{n-1} \otimes_A B) \cong (E_0 \otimes_A \cdots \otimes_A E_{n-1}) \otimes_A B, \tag{$\ast$}\] so there is a unique diagram-respecting embedding $\bigotimes_A^{\sigma_A|_{\tld{\mc{F}}}} \tld{\mc{F}} \to \bigotimes_B^{\sigma_B|_{\dot\theta(\tld{\mc{F}})}} \dot\theta(\tld{\mc{F}})$, and we define $\eps_{\tld{\mc{F}}}$ by composing this with the direct-limit embedding $\bigotimes_B^{\sigma_B|_{\dot\theta(\tld{\mc{F}})}} \dot\theta(\tld{\mc{F}}) \to K(B)$. Uniqueness of the diagram-respecting embedding $\bigotimes_A^{\sigma_A|_{\tld{\mc{F}}}} \tld{\mc{F}} \to \bigotimes_B^{\sigma_B|_{\dot\theta(\tld{\mc{F}})}} \dot\theta(\tld{\mc{F}})$ gives coherence of the family $\{\eps_{\tld{\mc{F}}} \mid \tld{\mc{F}} \fin \mc{F}_A\}$, and so by the universal property of the direct limit applied to this family we obtain an embedding $K(\theta) : K(A) \to K(B)$. Property \ref{Ka nat trans} in the definition of \Ka functor is immediate from the construction, and the verification that $K$ is a functor is straightforward diagram-chasing -- we leave this to the reader.
\end{proof}

\subsection{Tournament expansions} \label{s: SIR tourn free superpos}

We show the following:

\begin{thm} \label{t: tournament exp}
    Let $M$ be a relational \Fr structure with strong amalgamation and a local SIR. Then the generic tournament expansion has a \Ka functor, and hence has $\circ$-extensible $\omega$-age.
\end{thm}

\begin{defn} \label{tourn on exts}
    Let $\mc{L}^\ri$ be a relational language, and let $M^\ri$ be a \Fr $\mc{L}^\ri$-structure with strong amalgamation and a local SIR. Let $\mc{L} = \mc{L}^\ri \cup \{\ra\}$ be the language of the generic tournament expansion $M$. (Here we change notation from the statement of the theorem, as indicating the reduct rather than the expansion will give a cleaner presentation.) For $A \in \mc{A}(M)$, we write $A^\ri$ for the $\mc{L}^\ri$-reduct of $A$ and $\ra_A$ for the tournament relation of $A$. Let $\otimes^\ri$ be the symmetric local SAO for $M^\ri$ induced by the local SIR on $M^\ri$, and let $\otimes$ be the local SAO for $M$ given by the free superposition of $\otimes^\ri$ and the standard tournament amalgamation operator.
    
    Let $A \in \mc{A}(M)$. We define a set $\mc{F}_A$ of labelled one-point extensions of $A$. For each $\zeta : A \to E$ in $\mc{E}_A$, and each ordering $\bar{s}$ of each $S \sub A$ such that $E \cong (\zeta(S), e) \otimes_S A$ via a diagram-respecting isomorphism (here $e$ is the extension vertex), put the labelled extension $(\zeta, \bar{s})$ in $\mc{F}_A$.
    
    We define a tournament relation $\twoheadrightarrow_A$ on $\mc{F}_A$. For distinct $(\zeta, \bar{s})$, $(\zeta', \bar{s}')$ in $\mc{F}_A$, we define $\twoheadrightarrow_A$ as follows. Let $C = (\zeta(S), e)$, $C' = (\zeta'(S'), e')$.

    \begin{itemize}
        \item If $\bar{s} \ra_\lex \bar{s}'$ in the lexicographic tournament relation on $A^{\leq \omega}$ induced by $\ra_A$, define $(\zeta, \bar{s}) \twoheadrightarrow_A (\zeta', \bar{s}')$.
        \item In the case $\bar{s} = \bar{s}'$: for notational convenience we identify the domains of $C$, $C'$, denoting the common domain by $D$. Let $n = |\bar{s}|$ and let $<$ be the order $s_0 < \cdots < s_{n-1} < e$ on $D$. We then compare $C, C'$ in the relational order $<_\rel$ induced by $<$: if $C <_\rel C'$ we define $(\zeta, \bar{s}) \twoheadrightarrow_A (\zeta', \bar{s}')$. (See \Cref{rel order} for the definition of the relational order.)
    \end{itemize}
\end{defn}

\begin{defn} \label{tourn labelled extension map}
    We use the notation of the previous definition. Let $\theta : A \to B$ be an embedding in $\mc{A}(M)$. We define a map $\dot\theta : \mc{F}_A \to \mc{F}_B$ similarly to \Cref{order labelled extension map}, as follows. For $\gamma = (\zeta, \bar{s}) \in \mc{F}_A$, let $E = \cod(\zeta)$ and let $\zeta' : B \to E'$ be the unique element of $\mc{E}_B$ with an embedding-respecting isomorphism to $E \otimes_A B$, the codomain of $\zeta \otimes \theta$. Define $\dot\theta(\gamma) = (\zeta', \theta(\bar{s}))$. By an analogous argument to that in \Cref{order labelled extension map}, we have $\dot\theta(\gamma) \in \mc{F}_B$, and also similarly if $\gamma_0 \twoheadrightarrow_A \gamma_1$ then $\dot\theta(\gamma_0) \twoheadrightarrow_B \dot\theta(\gamma_1)$.
\end{defn}

\begin{proof}[Proof of \Cref{t: tournament exp}]
    We use the notation of the previous two definitions. Let $A \in \mc{A}(M)$. To define $K(A)$, we first define an $\mc{L}^\ri$-structure $K(A)^\ri = \bigotimes^\ri_A \mc{F}^\ri_A$, where $\mc{F}^\ri_A$ is the set of $\mc{L}^\ri$-reducts of the elements of $\mc{F}_A$. (Recall that $\bigotimes^\ri$ is symmetric, so we do not need to specify a linear order along which to perform the amalgamation. We also abuse notation similarly to the proof of \Cref{t: order exp}: formally we take the amalgam of the first coordinates of elements of $\mc{F}^\ri_A$.) We then expand $K(A)^\ri$ to an $\mc{L}$-structure as follows. For each $\gamma = (\zeta, \bar{s}) \in \mc{F}_A$, let $\rho_\gamma : \cod(\zeta) \to K(A)^\ri$ be the $\mc{L}^\ri$-embedding into the amalgam and let $e_\gamma \in \cod(\zeta)$ be the extension vertex of $\zeta$. For $\gamma \in \mc{F}_A$ we define the tournament structure on $\im(\rho_\gamma)$ so that $\rho$ is an $\mc{L}$-embedding, and for $\gamma, \gamma' \in \mc{F}_A$ with $\gamma \twoheadrightarrow \gamma'$ we define  $\rho_\gamma(e_\gamma) \ra \rho_{\gamma'}(e_{\gamma'})$. 

    Let $\theta : A \to B$ be an embedding in $\mc{A}(M)$. As in the proofs of \Cref{p: quick Mue} and \Cref{t: order exp}, for each $\tld{\mc{F}} \fin \mc{F}_A$ there is a unique diagram-preserving $\mc{L}^\ri$-embedding $\otimes^\ri_A \mc{F}_A^\ri \to \otimes^\ri_B \dot\theta(\mc{F}_A^\ri)$, and so via the universal property of the direct limit we obtain an $\mc{L}^\ri$-embedding $K(\theta)^\ri : K(A)^\ri \to K(B)^\ri$. As $\dot\theta$ preserves $\twoheadrightarrow$, in fact $K(\theta)^\ri$ is an $\mc{L}$-embedding $K(A) \to K(B)$, and we define $K(\theta)$ to be this embedding. We leave it to the reader to check property \ref{Ka nat trans} in the definition of a \Ka functor, as well as the functoriality of $K$.
\end{proof}

\subsection{A peculiar structure with a tournament and SWIR} \label{s: peculiar str}

In this section, we show that the tournament analogue of \Cref{t: order exp} does not hold: that is, we show there is a \Fr structure with a tournament relation and a SWIR which does not have a universal automorphism group.

Let $\mc{L} = \{E, \rightarrow, D\}$ be a relational language with $E, \rightarrow$ binary and $D$ ternary, and let $\mc{C}$ be the class of finite $\mc{L}$-structures $A$ where $E^A$ is a graph relation, $\ra^A$ is a tournament relation and $A$ satisfies the following:
\begin{align*}
    (\all x, y, z) &(D xyz \Rightarrow D xzy), \\
    (\all x, y, z) &(D xyz \Rightarrow \neg D yzx \wedge \neg D zxy), \\
    (\all u, x, y, z) &(E ux \wedge E uy \wedge E uz \Rightarrow D xyz \vee D yzx \vee D zxy).
\end{align*}

If $(u, v, w) \in D^A$, we say that $u$ \emph{dominates} $v, w$.

It is straightforward to see that $\mc{C}$ is an amalgamation class; let $M = \FrLim(\mc{C})$. For $A, B, C \fin M$, we define $B \ind_A C$ if the following hold:
\begin{enumerate}[label=(\roman*)]
    \item there are no graph edges between $B \setminus A$, $C \setminus A$;
    \item $b \ra c$ for $b \in B \setminus A$, $c \in C \setminus A$;
    \item for $b \in B \setminus A$, $c \in C \setminus A$, $u \in A \cup B \cup C$ such that there exists $a \in A$ with $Eab \wedge Eac \wedge Eau$: if the tournament relation on $bcu$ is a linear order, then the largest element $D$-dominates the other two, and if the tournament relation on $bcu$ is a directed $3$-cycle, then $Dubc$ (note that by (ii) we have $u \in A$);
    \item for $b \in B \setminus A$, $c \in C \setminus A$, $u \in A \cup B \cup C$ such that there is no $a \in A$ with $Eab \wedge Eac \wedge Eau$, we have that there is no $D$-relation on $\{b, c, u\}$.
\end{enumerate}

It is straightforward to check that $\ind$ is a SWIR (taking particular care with base monotonicity), so we have that $M$ is a structure with a tournament relation and a SWIR. We now show that $\Aut(M)$ is non-universal.

\begin{lem} \label{l: peculiar str}
    Let $\tau \in \Aut(M)$ have order $3$. Then $\tau$ has no fixed points.
\end{lem}
\begin{proof}
    First observe that for each $\tau$-fixed point $u \in M$ and each $v \in M$ with $Euv$, we have that $v$ is also $\tau$-fixed: if not then we have $Euv, Eu\tau{v}, Eu\tau^2{v}$, so some vertex of $v, \tau(v), \tau^2(v)$ must $D$-dominate the other two, which is impossible as $\tau$ preserves $D$. Suppose for a contradiction that $\tau$ has a fixed point $u$, and let $w \in M$, $w \neq u$. Then by the extension property of $M$ there is $v \in M$ with $Euv, Evw$, so $w$ is fixed and $\tau = \id_M$, contradiction.
\end{proof}

\begin{prop}
    $\Aut(M)$ is not universal. In particular, $M$ does not have $\circ$-extensible $\omega$-age.
\end{prop}
\begin{proof}
    Let $A$ have vertex set $\{a_{n, i} \mid n \in \N, i \in \Z/3\Z\}$, empty graph and $D$-relations, and tournament relation given by: $a_{n, i} \ra a_{n, i + 1}$ for all $n, i$; $a_{n, i} \ra a_{n', i'}$ for all $i, i'$ and all $n, n'$ with $n < n'$. (Concisely, the tournament structure is the lexicographic product of $\N$ and a directed $3$-cycle.) We have $A \in \mc{A}_\omega(M)$ and $\Aut(A) = C_3^\omega$. Suppose for a contradiction that there is an embedding $\Aut(A) \to \Aut(M)$, and denote its image by $H$. Observe that $H$ contains continuum-many elements of order $3$. Let $v \in M$. Then as the $H$-orbit of $v$ is countable, we have $hv = h'v$ for some distinct $h, h' \in H$. But then $h^{-1}h'$ is of order $3$ and has a fixed point, contradicting \Cref{l: peculiar str}.
\end{proof}

\section{Examples} \label{s: exs}

In this section, we consider the existence of a \Ka functor and a SWIR expansion for a range of structures, including all countably infinite ultrahomogeneous oriented graphs in Cherlin's catalogue (\cite{Che98}), with the exception of $P(3)$. (See \cite{PS20} for another exposition of this catalogue.)

\begin{thm} \label{t: big list of exs}
    Each of the following structures has a \Ka functor (and hence $\circ$-extensible $\omega$-age and universal automorphism group):
    \begin{itemize}
        \item the generic $3$-hypertournament (\ref{ex: 3-ht});
        \item the generic two-graph (\ref{p three key exs});
        \item the generic bipartite tournament and generic $\omega$-partite tournament (\ref{ex: n-partite tournament});
        \item the circular order on $\Q$ (\ref{betweenness and friends});
        \item the dense local order $S(2)$ and the related oriented graphs $S(3)$, $\widehat{\Q}$ (\ref{S(2) and friends});
        \item the generic $\vec{C}_4$-enlarged tournament $\widehat{T^\omega}$ (\ref{ex: C_4-enlarged});
        \item the dense meet-tree (\ref{meet-trees and meet-tree exps});
    \end{itemize}

    The following structures have $\circ$-extensible $\omega$-age, but no \Ka functor:
    \begin{itemize}
        \item the generic $n$-partite tournament for $2 < n < \omega$ (\ref{ex: n-partite tournament});
        \item the betweenness and separation relations on $\Q$ (\ref{betweenness and friends});
    \end{itemize}
    
    The following structures have non-universal automorphism groups (and hence their $\omega$-ages are not $\circ$-extensible and they have no \Ka functor):
    \begin{itemize}
        \item for $n \geq 3$, the generic $n$-anticlique-free oriented graph (\ref{p three key exs});
        \item the semigeneric $\omega$-partite tournament (\ref{ex: semigeneric}).
    \end{itemize}

    Each of the structures listed above has a finite SWIR expansion.
\end{thm}

We also show the following structural results:

\begin{thm} \label{t: products}
    We have the following regarding products of structures:
    \begin{itemize}
        \item Let $M, N$ be relational \Fr structures with $M$ transitive, and suppose that $M, N$ have \Ka functors. Then the lexicographic product $M[N]^s$ has a \Ka functor (\ref{ex: lex prod}).
        \item Let $M, N$ be relational \Fr structures such that $M$ is transitive and has strong amalgamation, and suppose that $M, N$ have (local) SWIRs. Then the lexicographic product $M[N]^s$ has a (local) SWIR (\ref{ex: lex prod}).
        \item There are relational \Fr structures $M, N$ with \Ka functors such that their free superposition $M \ast N$ does not have a \Ka functor (\ref{ex: free superposition}).
        \item Let $N$ be the generic meet-tree expansion of a transitive \Fr structure $M$ with free amalgamation (satisfying mild non-triviality conditions). Then $\Aut(N)$ is non-universal (\ref{meet-trees and meet-tree exps}).
    \end{itemize}
\end{thm}

In the course of proving the above theorems, we will freely switch between SWIRs and SAOs, as each one induces the other -- see \Cref{s: SWIR SAO equivalence}. Also note for the examples in this section that if a \Fr structure $M$ has a local SWIR, then it has a global SWIR in the expansion given by adding a constant to fix a point.

Moreover, in \Cref{ordered strs no SWIR} we give two particularly notable examples of failure of the existence of a local SWIR: in all other negative examples, we demonstrate failure of automorphism-invariance, but here we show that in fact there are classes of ordered structures without a local SWIR. 

\begin{term*}
    A number of examples below discuss involutions -- when we mention involutions we will always assume they are not equal to the identity.
\end{term*}

\subsection{The generic \texorpdfstring{$n$}{n}-partite tournament; \Ka functors for cofinal subclasses} \label{ex: n-partite tournament}

We first show that for a structure to have $\circ$-extensible $\omega$-age, it suffices to have $\circ$-extensive embeddings into a particular subclass with a \Ka functor. We then apply this to the example of the generic $n$-partite tournament.

\begin{lem} \label{l: cofinal}
    Let $M$ be a \Fr structure. Let $\mc{B}$ be a subclass of $\mc{A}(M)$, and let $\mc{B}_\omega$ be the class of colimits of structures in $\mc{B}$. Suppose that $\mc{B}$, $\mc{B}_\omega$ have the following properties:
    \begin{enumerate}[label=(\roman*)]
        \item \label{cof closed under embs} for each embedding $B \to A$ with $B \in \mc{B}$, $A \in \mc{A}(M)$, we have $A \in \mc{B}$;
        \item \label{cof factor embs} each embedding $A \to D$ with $A \in \mc{A}(M)$, $D \in \mc{B}_\omega$ factors into a pair of embeddings $A \to B \to D$ for some $B \in \mc{B}$;
        \item \label{cof B Ka} $\mc{B}$ has a \Ka functor $K : \mc{B} \to \mc{B}_\omega$;
        \item \label{cof ext emb into B_omega} for each $C \in \mc{A}_\omega(M)$, there is an $\circ$-extensive embedding $C \to D$ for some $D \in \mc{B}_\omega$.
    \end{enumerate}
    Then $M$ has $\circ$-extensible $\omega$-age.
\end{lem}
(Note that here we permit non-hereditary $\mc{B}$.)

\begin{proof}
    This is a straightforward adaptation of Theorem 3.3 and Corollary 3.9 of \cite{KM17} -- we sketch the adjustments needed. Specifically: Lemmas 1.1--1.3 of \cite{KM17} hold for $\mc{B}$, Lemma 1.4 holds for $\mc{A}(M)$, Lemma 2.3 holds for $\mc{B}$. In Theorem 3.3: one shows that $K^\omega(D) \cong M$ for each $D \in \mc{B}_\omega$. Let $L = K^\omega(D)$ (we use the notation of the proof in \cite{KM17}). To show that $\mc{A}(M) \sub \Age(L)$, use JEP and \ref{cof closed under embs}. For the proof of the extension property, use \ref{cof factor embs}, amalgamate the one-point extension of $A$ with $A \to B$ and use \ref{cof closed under embs} to get a one-point extension of $B$ in $\mc{B}$. Properties \ref{cof ext emb into B_omega}, \ref{cof B Ka} then give Corollary 3.9: compose the $\circ$-extensive embeddings $C \to D$ and $D \to K^\omega(D)$ to obtain a $\circ$-extensive embedding $C \to K^\omega(D) \cong M$.  
\end{proof}

\begin{defn}
    Let $2 \leq n \leq \omega$. Let $A$ be an oriented graph. For distinct $x, y \in A$, we write $x \ic y$ if there is no oriented edge between $x$ and $y$. We say that $A$ is an \emph{$n$-partite tournament} if $\ic$ is an equivalence relation with $\leq n$ equivalence classes, which we call \emph{parts}. Let $\mc{C}$ be the class of finite $n$-partite tournaments. It is straightforward to show that $\mc{C}$ is an amalgamation class; we let $M = \FrLim(\mc{C})$. We call $M$ the \emph{generic $n$-partite tournament}.
\end{defn}

\begin{prop}
    Let $2 \leq n \leq \omega$, and let $M$ be the generic $n$-partite tournament.
    \begin{enumerate}[label=(\roman*)]
        \item \label{3pt no SAO} For $2 < n < \omega$, the structure $M$ does not have a local SWIR, but has a finite SIR expansion.
        \item \label{2pt local SAO} For $n = 2$, the structure $M$ has a local SWIR but not a global SWIR.
        \item \label{omegapt global SAO} For $n = \omega$, the structure $M$ has a global SWIR.
    \end{enumerate}
\end{prop}
\begin{proof} 
    \ref{3pt no SAO}: Let $A = \{a_0, a_1\}$ with $a_0 \ic a_1$. Let $B = A \cup \{u_i \mid i \in \Z/4\Z\}$ with $u_0 u_1 u_2 u_3$ a directed $4$-cycle and $a_0 \ra u_i, a_1 \ra u_i$ for all $i$. We have $A, B \in \mc{C}$. Let $C$ be an $n-1$-partite tournament containing $A$ with exactly $n-1$ parts. Suppose there exists a local SAO $\otimes$ for $\mc{C}$. Then as $B \otimes_A C$ is $n$-partite, for some $u_i$ and $v \in C \setminus A$ we must have $u_i \ic v$ in the amalgam. But then by applying (Inv) to the automorphism $u_k \mapsto u_{k+1}$ of $B$ we have $u_{i + 1} \ic v$, and so $u_i \ic u_{i + 1}$, contradiction.

    We now briefly sketch an expansion with a global symmetric SAO: expand the tournament language by $n$ unary predicates $P_0, \cdots, P_{n-1}$, and let $\mc{C}'$ be the class of finite structures in this expanded language obtained by expanding each structure in $\mc{C}$ by all possible labellings of its parts using $P_0, \cdots, P_{n-1}$ (interpretations of some $P_i$ may be empty, and each part has exactly one label). For $B, C \in \mc{C}'$ with $B \cap C = A$, we define $B \otimes_A C$ by: for all $b \in B \setminus A$, $c \in C \setminus A$ with $P_i^B(b), P_j^C(c)$, set $b \ic c$ if $i = j$, $b \ra c$ if $i < j$ and $b \la c$ if $i > j$.

    \ref{2pt local SAO}: A similar argument to \ref{3pt no SAO} shows that $\mc{C}$ does not have a global SAO: take $B$ a directed $4$-cycle and $C$ a vertex. We define a local SAO for $\mc{C}$ as follows. Let $A, B, C \in \mc{C}$ with $B \cap C = A \neq \varnothing$. We now define the amalgam $B \otimes_A C$. Let $b \in B \setminus A$, $c \in C \setminus A$. If there exists $a \in A$ with $b \ic a, c \ic a$ or with edges $ab, ac$ (with some orientation) then set $b \ic c$ in the amalgam. Otherwise set $b \ra c$. Checking that $\otimes$ is a SAO is straightforward.

    \ref{omegapt global SAO}: We show that $\mc{C}$ has a global SAO. Let $A, B, C \in \mc{C}$ with $B \cap C = A$. We define $B \otimes_A C$: for $b \in B \setminus A$, $c \in C \setminus A$, if there exists $a \in A$ with $b \ic a, c \ic a$, set $b \ic c$, and otherwise set $b \ra c$.
\end{proof}

\begin{prop} \label{p: Ka npartite}
    Let $2 \leq n \leq \omega$, and let $M$ be the generic $n$-partite tournament.
    \begin{enumerate}[label=(\roman*)]
        \item \label{Ka npartite} If $2 < n < \omega$ then $M$ does not have a \Ka functor, but has $\circ$-extensible $\omega$-age.
        \item \label{Ka 2partite} If $n = 2$ then $M$ has a \Ka functor.
        \item \label{Ka omegapartite} If $n = \omega$ then $M$ has a \Ka functor.
    \end{enumerate}
\end{prop}
\begin{proof}
    \ref{Ka npartite}: We first show that $M$ has no \Ka functor. Let $A = \{a\}$ and let $B = A \cup \bigcup_{i < n-1}\{b_i, b'_i\}$ with $b_i \ic b'_i$, $b_i \ra a$, $b'_i \ra a$ for all $i$ and $b_i b_j b'_i b'_j$ a directed $4$-cycle for $i < j < n-1$. Then $A, B \in \mc{A}(M)$. Define $\sigma \in \Aut(B)$ by $\sigma(a) = a$; $\sigma(b_i) = b_{i+1}$, $\sigma(b'_i) = b'_{i+1}$ for $i < n - 2$; $\sigma(b_{n-2}) = b'_0$, $\sigma(b'_{n-2}) = b_0$. Suppose for a contradiction that $M$ has \Ka functor $K$. Let $e \in K(A)$ with $e \ra a$. Let $f : A \to B$ be the inclusion map. Let $e' = K(f)(e)$. Then as $B$ has $n$ parts and $e' \ra a$, we have $e' \ic b_i$ for some $i$. As $\sigma \circ f = f$ we have $K(\sigma) \circ K(f) = K(f)$, so $e' = K(\sigma)(e') \ic K(\sigma)(b_i) = \sigma(b_i)$, but $\sigma(b_i), b_i$ are not in the same part -- contradiction.

    We now use \Cref{l: cofinal} to show that $M$ has $\circ$-extensible $\omega$-age. Let $\mc{B} \sub \mc{A}_\omega(M)$ consist of the finite $n$-partite tournaments with $n$ parts. Properties \ref{cof closed under embs}, \ref{cof factor embs} in \Cref{l: cofinal} are straightforward, and \ref{cof ext emb into B_omega} follows for $C \in \mc{A}_\omega(M)$ with $k < n$ parts by adding $n-k$ parts and out-edges from each vertex of $C$ to each vertex of the new parts. We now show property \ref{cof B Ka}, which states that $\mc{B}$ has a \Ka functor, and then by \Cref{l: cofinal} it will follow that $\Ao(M)$ is $\circ$-extensible. Let $B \in \mc{B}$. For each $(B, e) \in \mc{E}_B$ and each $b \in B$ with $b \ic e$, add to $K(B)$ a labelled copy $(B, e_b)$ of $(B, e)$ over $B$. To specify the structure between labelled extension vertices $e^{}_b, e'_{b'}$: if $b \ic b'$, set $e^{}_b \ic e'_{b'}$; otherwise give $e^{}_b e'_{b'}$ the same orientation as $b b'$. For $f : B \to B'$ in $\mc{B}$, extend $f$ to $K(f)$ via $K(f)(e^{}_{b\vphantom{f(b)}}) = e'_{f(b)}$, where $(B', e'_{f(b)})$ is the extension with $(f(B), e'_{f(b)}) \cong (B, e^{}_b)$ and $e'_{f(b)} \ra b'$ for $b' \in B' \setminus f(B)$ with $b'$, $f(b)$ in different parts.
    
    \ref{Ka 2partite}: Let $A \in \mc{A}(M)$. For $(A, e) \in \mc{E}_A$ and $a \in A$, add to $K(A)$ a labelled copy $(A, e_a)$ of $(A, e)$ over $A$. For labelled extension vertices $e^{}_{a\vphantom{a'}}, e'_{a'\vphantom{a'}}$, if there is $u \in A$ with $u, e^{}_{a\vphantom{a'}}, e'_{a'\vphantom{a'}}$ in the same part or with edges $e^{}_{a\vphantom{a'}} u$, $e'_{a'\vphantom{a'}}u$, then $e^{}_{a\vphantom{a'}}, e'_{a'\vphantom{a'}}$ are in the same part in any amalgam, so set $e^{}_{a\vphantom{a'}} \ic e'_{a'\vphantom{a'}}$. Otherwise we must have $e^{}_{a\vphantom{a'}}, e'_{a'\vphantom{a'}}$ in different parts in any amalgam, and we now specify the edge relation between them. If $a \ic e^{}_{a\vphantom{a'}}$ and $e'_{a'\vphantom{a'}}, a'$ are in different parts, set $e^{}_{a\vphantom{a'}} \ra e'_{a'\vphantom{a'}}$ (and vice versa). The remaining case is where $a, a'$ are in different parts: give $e^{}_{a\vphantom{a'}} e'_{a'\vphantom{a'}}$ the same orientation as $a a'$. For each embedding $f : A \to B$ in $\mc{A}(M)$, we define $K(f) \supseteq f$ as follows: for $e^{}_a \in K(A)$, define $K(f)(e^{}_{a\vphantom{f(a)}}) = e'_{f(a)}$, where $(B, e'_{f(a)})$ is the extension with $(A, e^{}_{a\vphantom{f(a)}}) \cong (f(A), e'_{f(a)})$ and $e'_{f(a)} \ra b$ for all $b \in B \setminus f(A)$ such that there is no $u \in f(A)$ with $e'_{f(a)} \ic u \ic b$ or with edges $e'_{f(a)}u$, $bu$. 

    \ref{Ka omegapartite}: Let $A \in \mc{A}(M)$. For $(A, e) \in \mc{E}_A$ and $a \in A$ with $a \ic e$, add to $K(A)$ a labelled copy $(A, e_a)$ of $(A, e)$ over $A$ and define the structure between labelled extensions as in \ref{Ka npartite}. For $(A, e) \in \mc{E}_A$ such that $e$ does not lie in any part of $A$, add $(A, e)$ without a label to $K(A)$. For $e, e'$ not lying in any part of $A$, we set $e \ic e'$ in $K(A)$, and for each labelled extension $e_a$ and unlabelled $e'$ we set $e_a \ra e'$. For $f : A \to B$ in $\mc{A}(M)$, define $K(f)$ for labelled extensions as in \ref{Ka npartite}, and for unlabelled $e \in K(A)$ we define $K(f)(e)$ to be the unlabelled extension $(B, e')$ such that $(A, e) \cong (f(A), e')$ and $e' \ra b$ for all $b \in B \setminus f(A)$. 
\end{proof}

\subsection{Reducts of \texorpdfstring{$(\Q, <)$}{(Q, <)}} \label{betweenness and friends}

We now discuss the non-trivial proper reducts of $(\Q, <)$: these are the \Fr structures given by the betweenness relation, the circular order and the separation relation (see \cite[Theorem 6.2.1, Example 2.3.1]{Mac11}, \cite{Cam76}). 

For $(\Q, <)$ itself, \Cref{ex: structures with SWIRS} gives a SWIR, and we define a \Ka functor as follows. For each non-empty finite linear order $A$, there is a unique amalgam $K(A)$ of the one-point extensions of $A$, which correspond to cuts $(u, v)$ with $u, v \in A$ and the endpoint cuts $(-\infty, \min A)$, $(\max A, \infty)$. Given an embedding $f : A \to B$, for each cut $(u, v) \in K(A)$ with $u \in A$, define $K(f)$ to send $(u,v)$ to the cut of $B$ with first coordinate $f(u)$, and define $K(f)$ to send the cut $(-\infty, \min A)$ to the cut of $B$ with second coordinate $f(\min A)$. (See \cite[Example 2.8]{KM17}, which is for monotone maps but can be easily adapted.)

Also note that it is immediate that each reduct of $(\Q, <)$ has a finite SWIR expansion, as $(\Q, <)$ itself has a SWIR. No proper reduct will have a \Ka functor in the strict sense, but this is for a not particularly striking reason: it is impossible to functorially map the extension of a single point. Each proper reduct will have a \Ka functor for structures of large enough size, and this suffices to give $\circ$-extensible $\omega$-age by the below lemma:

\begin{lem} \label{l: subcl at least n with Ka}
    Let $M$ be a \Fr structure. Let $n \geq 1$, and let $\mc{B}$ be the subclass of $\mc{A}(M)$ consisting of the structures in $\mc{A}(M)$ with size $\geq n$. Suppose that $\mc{B}$ has a \Ka functor and that for all $A \in \mc{A}(M)$ there is a $\circ$-extensive embedding $A \to B$  for some $B \in \mc{B}$. Then $M$ has $\circ$-extensible $\omega$-age. 
\end{lem}
\begin{proof}
    This follows from \Cref{l: cofinal}: properties \ref{cof closed under embs}, \ref{cof factor embs} are immediate, and properties \ref{cof B Ka}, \ref{cof ext emb into B_omega} follow by assumption.
\end{proof}

\subsubsection{The betweenness relation on \texorpdfstring{$\Q$}{Q}} \label{ex: betweenness}

Let $\beta$ be the ternary betweenness relation on $\Q$: that is,
$\beta(a, b, c) \Leftrightarrow (a < b < c) \,\vee\, (a > b > c)$.

\begin{prop}
    The betweenness structure $(\Q, \beta)$ does not have a local SWIR or a \Ka functor, but has $\circ$-extensible $\omega$-age.
\end{prop}
\begin{proof}
    Let $\mc{C} = \Age(\Q, \beta)$. Let $A = \{a\}$, $B = \{a, b\}$, $C = \{a, c_0, c_1\}$ with $\beta (c_0, a, c_1)$. Then any amalgam $B \otimes_A C$ does not satisfy (Inv) (using the automorphism of $C$ fixing $a$ and swapping $c_0, c_1$), so $\mc{C}$ does not have a local SAO. Similarly, $(\Q, \beta)$ does not have a \Ka functor: with the same $A, B, C$ as above, consider the embedding $A \to C$ and $B$ as an extension of $A$, and use a similar argument to \Cref{p: Ka npartite}\ref{Ka npartite} with the involution $\sigma : c_0 \leftrightarrow c_1$ of $C$.

    Let $\mc{B}$ be the subclass of $\mc{A}((\Q, \beta))$ consisting of structures of size $\geq 2$. By \Cref{l: subcl at least n with Ka}, to show that $(\Q, \beta)$ has $\circ$-extensible $\omega$-age, it suffices to show that $\mc{B}$ has a \Ka functor (the condition on $\circ$-extensive embeddings is trivial for $|A| = 1$). Let $B \in \mc{B}$. For distinct $u, v \in B$, we say that $\{u, v\}$ is a \emph{cut} of $B$ if there is no $b \in B$ between $u, v$. We say that $b \in B$ is an \emph{endpoint} of $B$ if there do not exist distinct $u, v \in B$ with $b$ between $u, v$. We now define $K(B) \supseteq B$: for each cut $\{u, v\}$ of $B$, add points $(u, v), (v, u)$ to $K(B)$ with $\beta(u, (u, v), (v, u))$ and $\beta(v, (v, u), (u, v))$, and for each endpoint $b \in B$, letting $u \in B$ be the unique element of $B$ with $\{b, u\}$ a cut, add a point $b_\ast$ to $K(B)$ with $\beta(b_\ast, b, u)$. 
    
    Let $f : B \to B'$ be an embedding in $\mc{B}$. We define $K(f) \supseteq f$ as follows. For $(u, v) \in K(B)$ resulting from a cut $\{u, v\}$ of $B$, let $b' \in B'$ be such that $\{f(u), b'\}$ is a cut and $\beta(f(u), b', f(v))$, and define $K(f)((u, v)) = (u, b')$. For $b_\ast \in K(B)$ resulting from an endpoint $b \in B$ with cut $\{b, u\}$: if there is $b' \in B'$ with $\{f(b), b'\}$ a cut and $\beta(b', f(b), f(u))$ define $K(f)(b_\ast) = (f(b), b')$; otherwise, define $K(f)(b_\ast) = f(b)_\ast$.
\end{proof}

\subsubsection{The circular order on \texorpdfstring{$\Q$}{Q}} \label{ex: circular order}

Let $\gamma$ be the circular order on $\Q$. That is, \[\gamma(a, b, c) \Leftrightarrow (a < b < c) \vee (b < c < a) \vee (c < a < b).\]

\begin{prop}
    The circular order $(\Q, \gamma)$ has a local SWIR and a \Ka functor.
\end{prop}
\begin{proof}
    For $A \in \mc{A}((\Q, \gamma))$ and distinct $u, v \in A$, we say that $(u, v)$ is a \emph{cut} of $A$ if there does not exist $a \in A$ with $\gamma(u, a, v)$. We define a local SWIR for $(\Q, \gamma)$: for $A, B, C \fin (\Q, \gamma)$ with $A \neq \varnothing$, in the case $|A| \geq 2$ we define $B \ind_A C$ if, for each cut $(u, v)$ of $A$, each pair $b \in B \setminus A$, $c \in C \setminus A$ with $\gamma(u, b, v) \wedge \gamma(u, c, v)$ satisfies $\gamma(u, b, c)$, and in the case $A = \{u\}$, we define $B \ind_A C$ if $\gamma(u, b, c)$ for all $b \in B \setminus A$, $c \in C \setminus A$. We leave it to the reader to check the SWIR axioms. (To see that no global SAO can exist, consider the amalgam over $\varnothing$ of $B = \{b\}$, $C = \{c_0, c_1\}$.)
    
    We now show that $(\Q, \gamma)$ has a \Ka functor. Let $A \in \mc{A}((\Q, \gamma))$. If $|A| \geq 2$, we define $K(A) \supseteq A$ by adding all cuts $(u, v)$, where we define $\gamma(u, (u, v), v)$. For $A = \{u\}$, we define $K(A) \supseteq A$ by adding a single extra point $(u, \varnothing)$. Given an embedding $f : A \to B$ in $\mc{A}((\Q, \gamma))$ (where we assume $|B| \geq 2$ as the case $|B| = 1$ is trivial), we define $K(f) \supseteq f$ as follows. For $(u, v) \in K(A) \setminus A$ (where possibly $v = \varnothing$), let $b \in B$ be such that $(f(u), b)$ is a cut of $B$, and define $K(f)((u, v)) = (f(u), b)$.
\end{proof}

\subsubsection{The separation relation on \texorpdfstring{$\Q$}{Q}} \label{ex: separation}

Let $\sigma$ be the separation relation on $\Q$: that is,
\[\sigma(a, b; c, d) \Leftrightarrow (\gamma(a, c, b) \wedge \gamma(b, d, a)) \vee (\gamma(a, d, b) \wedge \gamma(b, c, a)),\]
where $\gamma$ is the circular order on $\Q$ (see \Cref{ex: circular order}). If $\sigma(a, b; c, d)$, we say that $c, d$ \emph{separate} $a, b$. Note that the substructure induced on any four points of $\Q$ has automorphism group $D_4$. (We picture $a, b, c, d$ on the unit circle, and $\sigma(a, b; c, d)$ holds iff the chords $ab$, $cd$ intersect.)

\begin{prop}
    The separation structure $(\Q, \sigma)$ does not have a local SWIR or a \Ka functor, but has $\circ$-extensible $\omega$-age.
\end{prop}
\begin{proof}
    To see that $(\Q. \sigma)$ does not have a local SAO, take $A = \{a\}$, $B = \{a, b\}$ and $C = \{a, c_0, c_1, c_2\}$ with $\sigma(a, c_1; c_0, c_2)$. Via the automorphism of $C$ fixing $a, c_1$ which swaps $c_0, c_2$, we see that any amalgam $B \otimes_A C$ does not satisfy (Inv). The argument that shows that $(\Q, \sigma)$ does not have a \Ka functor is similar to that for $(\Q, \beta)$ -- we leave this to the reader.

    Let $\mc{B}$ consist of the elements of $\mc{A}((\Q, \sigma))$ with size $\geq 3$. By \Cref{l: subcl at least n with Ka}, to show that $(\Q, \sigma)$ has $\circ$-extensible $\omega$-age it suffices to show that $\mc{B}$ has a \Ka functor (note that each $A$ of size $2$ embeds $\circ$-extensively in $B$ of size $3$). The \Ka functor is similar to that for $(\Q, \beta)$ in \Cref{ex: betweenness} -- we give a brief description. For $B \in \mc{B}$ and distinct $u, v \in B$, we say that $\{u, v\}$ is a \emph{cut} if no pair of elements of $B$ separates $u, v$. Let $B \in \mc{B}$. We define $K(B) \supseteq B$ by adding $(u, v), (v, u)$ for each cut $\{u, v\}$ of $B$, defining $\sigma(u, (v, u); v, (u, v))$ and $\sigma(b, (u, v); u, v)$ for all $b \in B \setminus \{u, v\}$ (note that this implies $\sigma(b, (v, u); u, v)$ for all $b \in B \setminus \{u, v\}$). Let $f : B \to B'$ be an embedding in $\mc{B}$. We define $K(f) \supseteq f$ as follows. For $(u, v) \in K(B) \setminus B$, if $\{f(u), f(v)\}$ is a cut of $B'$, let $K(f)((u, v)) = (f(u), f(v))$; otherwise, there exists unique $b' \in B'$ with $\{f(u), b'\}$ a cut and $\sigma(f(b), b'; f(u), f(v))$ for all $b \in B \setminus \{u, v\}$, so define $K(f)((u, v)) = (f(u), b')$.    
\end{proof}

\subsection{The dense local order \texorpdfstring{$S(2)$}{S(2)} and related structures} \label{S(2) and friends}

\subsubsection{The dense local order \texorpdfstring{$S(2)$}{S(2)}} \label{ex: S(2)}

Let $\mb{T}$ denote the unit circle in $\mathbb{C}$. For $x \in \mb{T}$, we denote $-x$ by $x'$ and call $x'$ the \emph{antipode} of $x$. We call $x, x'$ an \emph{antipodal pair}. We define an oriented graph on $\mb{T}$: for distinct $x, y \in \mb{T}$, we define $x \ra y$ if the angle subtended at the origin by the anticlockwise arc from $x$ to $y$ is $< \pi$. Let $S(2)$ be the tournament given by the substructure of $\mb{T}$ whose domain consists of the points with rational argument. The structure $S(2)$ is usually called the \emph{dense local order}. (See \cite{Cam90}, \cite{Lac84} for more background.)

\begin{notn} \label{open interval notn}
    Let $(A, \ra)$ be an oriented graph. Let $a, b \in A$ with $a \ra b$. We define the \emph{open interval} $\langle a, b \rangle = \{v \in A \mid a \ra v \ra b\}$, and use analogous notation ($[a, b], [a, b \rangle$ etc.) for other intervals.
\end{notn}
The proof of the below lemma is immediate:
\begin{lem} \label{forced ors in T}
    Let $u, v \in \mb{T}$ with $u \ra v$. Let $a \in \langle u, v \rangle$. Then $a \ra b$ for all $b \in [v, u']$ and $a \la b$ for all $b \in [v', u]$.
\end{lem}

Let $\mc{C} = \Age(S(2))$. For $A \sub S(2)$, let $A' = \{a' \mid a \in A\}$ and $\hat{A} = A \cup A'$. It is immediate that, given $A_0, A_1 \sub S(2)$ and an isomorphism $f : A_0 \to A_1$, we can extend $f$ to an isomorphism $\hat{A}_0 \to \hat{A}_1$ via $a' \mapsto f(a)'$. Thus, given $A \in \mc{C}$, we may define $\hat{A}$ by embedding $A$ in $S(2)$, and the oriented graph $\hat{A}$ is independent of the embedding we take. (We could alternatively define $\hat{A}$ directly: for $a, b \in A$ with $a \ra b$, define $a' \ra b', b \ra a', b' \ra a$. Note that for each $a \in \hat{A}$, the antipodal point $a' \in \hat{A}$ is the only point such that $aa'$ is not an edge.)

\begin{defn}
    Let $A \in \mc{C}$. For each $u, v \in \hat{A}$ with $u \ra v$ and such that there does not exist $w \in \hat{A}$ with $u \ra w \ra v$, we call $(u, v)$ a \emph{cut} of $\hat{A}$. In the case $A = \{a\}$, we define the cuts of $A$ to be $(a, a')$ and $(a', a)$. 
\end{defn}

\begin{prop} \label{p: S(2) props}
    Let $A \in \mc{C}$, and let $(A, e) \in \mc{C}$ be an extension.
    \begin{enumerate}[label=(\roman*)]
        \item \label{S(2) unique cut} There exists a unique cut $(u, v)$ of $\hat{A}$ with $e \in \langle u, v \rangle$ in $\widehat{(A, e)}$.
        \item \label{S(2) cuts determine} Let $f : A \to S(2)$ be an embedding, and extend $f$ to $\hat{f} : \hat{A} \to \mb{T}$. Let $(u, v)$ be the cut of $\hat{A}$ with $e \in \langle u, v \rangle$. Then for any $b \in \langle f(u), f(v) \rangle$, we have $\tp(b / f(A)) = \tp(e / A)$ (where we identify the parameters via $f$).
        \item \label{S(2) ultrahomog} The tournament $S(2)$ is ultrahomogeneous.
    \end{enumerate}
\end{prop}
\begin{proof}
    \ref{S(2) unique cut} is straightforward, considering $A$ as embedded in $S(2)$. \ref{S(2) cuts determine} follows by noting that $(f(u), f(v))$ is a cut of $\hat{f}(\hat{A})$ and using \Cref{forced ors in T}. \ref{S(2) cuts determine} shows that $\mc{C}$ has the extension property, implying \ref{S(2) ultrahomog}.
\end{proof}

The above proposition shows that extensions $(A, e) \in \mc{E}_A$ are in bijection with cuts of $\hat{A}$.

\begin{prop} \label{p: S(2) SWIR Ka}
    The dense local order $S(2)$ has a local SWIR and a \Ka functor.
\end{prop}
\begin{proof}
    Let $A, B, C \fin S(2)$ with $A \neq \varnothing$. For each pair $P = \{(u, v)$, $(u', v')\}$ of opposite cuts of $\hat{A}$, let $B_P = B \cap (\langle u, v \rangle \cup \langle u', v' \rangle)$, $C_P = C \cap (\langle u, v \rangle \cup \langle u', v' \rangle)$. We define $B \ind_A C$ if, for each pair $P$ of opposite cuts of $\hat{A}$, we have that \emph{$C_P$ lies anticlockwise from $B_P$}: that is, 
    \begin{alignat*}{2}
        B \cap \langle u, v \rangle &\ra C \cap \langle u, v \rangle, \quad & B \cap \langle u, v \rangle &\la C \cap \langle u', v' \rangle,\\ B \cap \langle u', v' \rangle &\ra C \cap \langle u', v' \rangle, \quad & B \cap \langle u', v' \rangle &\la C \cap \langle u, v \rangle.
    \end{alignat*}
    It is straightforward to check that $\ind$ is a local SWIR (use \Cref{p: S(2) props}\ref{S(2) cuts determine} for (Ex) and (St)), and we leave this to the reader.

    We now define a \Ka functor for $S(2)$. We denote the extension $(A, e) \in \mc{E}_A$ corresponding to the cut $(u, v)$ of $\hat{A}$ by $(A, e_{(u, v)})$. Let $A \in \mc{A}(S(2))$. To define $K(A) \supseteq A$, take the disjoint union of the extensions $(A, e_{(u, v)})$ for each cut $(u, v)$ of $\hat{A}$. We now define the tournament relation between extension vertices. If the cuts $(u_0, v_0), (u_1, v_1)$ are not an opposite pair, then by \Cref{forced ors in T} the orientation of $e_{(u_0, v_0)}, e_{(u_1, v_1)}$ is already determined. Otherwise, we have $(u_1, v_1) = (u'_0, v'_0)$. Let $i \in \{0, 1\}$ with $u_i \in A$ (note that exactly one of $u_0, u_1 = u_0'$ lies in $A$). We set $e_{(u_i, v_i)} \la e_{(u_{1-i}, v_{1-i})}$. (Informally, one can visualise this as follows: on the anticlockwise arc from $u_i$ to $v_i$ beginning with the ``real" point $u_i$, the ``real" point $e_{(u_i, v_i)}$ occurs before the ``phantom" point $e'_{(u_{1-i}, v_{1-i})}$.)

    For each embedding $f : A \to B$ in $\mc{A}(S(2))$, we define $K(f) \supseteq f$ as follows. For each $e_{(u, v)} \in K(A)$, let $w \in \tld{B}$ be such that $(f(u), w)$ is a cut (note that $w$ is unique and that $(f(u)', w')$ is also a cut), and define $K(f)(e_{(u, v)})$ to be the extension vertex corresponding to the cut $(f(u), w)$ of $B$. It is straightforward to verify that $K$ is a \Ka functor.    
\end{proof}

\subsubsection{The oriented graph \texorpdfstring{$\widehat{\Q}$}{Q-hat}} \label{ex: Q-hat}

We use the notation of \Cref{ex: S(2)}. Let $\widehat{\Q}$ be the oriented graph $S(2) \cup S(2)' \sub \mb{T}$; that is, the oriented graph $\widehat{\Q}$ consists of $S(2)$ together with its antipodal points in the unit circle. (The notation $\widehat{\Q}$ is from \cite{Che87}. The structure $\widehat{T^\omega}$ of \Cref{ex: C_4-enlarged} and the structure $\widehat{\Q}$ can be produced from the random tournament $T^\omega$ and $\Q$ as follows: for $T$ an infinite tournament, the oriented graph $\widehat{T}$ consists of two disjoint copies $T \times \{0\}, T \times \{1\}$ of $T$ together with the additional directed edges $(u, i) \ra (v, j)$ for $i \neq j$ and $v \ra u$.)

\begin{prop}
    The oriented graph $\widehat{\Q}$ has a local SWIR and a \Ka functor.
\end{prop}
\begin{proof}
    We define a local SWIR for $\widehat{\Q}$ analogously to $S(2)$, and only sketch the details: cuts are defined similarly, and for $A \in \Age(\widehat{\Q})$, extensions $(A, e)$ of $A$ are in bijection with cuts of $\hat{A}$ together with $A' \setminus A$. Let $A, B, C \fin \widehat{\Q}$ with $A \neq \varnothing$. We define $B \ind_A C$ if for each pair $P = \{\langle u, v \rangle, \langle u', v' \rangle\}$ of opposite cuts of $\hat{A}$, as in \Cref{p: S(2) SWIR Ka}, we have $C_P$ anticlockwise from $B_P$ (note that $B_P$, $C_P$ may contain antipodal pairs of points). Note that this completely describes the structure on $A \cup B \cup C$: for $b \in B \setminus A$, $c \in C \setminus A$ such that at least one of $b, c$ lies in $A' \setminus A$, there is only one possibility for the relation between $b$ and $c$.

    A \Ka functor for $\widehat{\Q}$ can also be defined analogously to $S(2)$. Let $A \in \mc{A}(\widehat{\Q})$, and to define $K(A)$ take the union of extensions corresponding to cuts and to elements of $A' \setminus A$. The only relation that is not already determined is that between extensions corresponding to opposite pairs of cuts: here, we cannot distinguish between ``real" and ``phantom" points as for $S(2)$, and we simply put no edge between these extensions (so they form an antipodal pair).    
\end{proof}

\subsubsection{The oriented graph \texorpdfstring{$S(3)$}{S(3)}} \label{ex: S(3)}

The structure $S(3)$ is defined in a very similar manner to $S(2)$, and our definition of a local SWIR and \Ka functor will also be similar -- we only sketch the details. We define an oriented graph structure on $\mb{T}$ (different to that in \Cref{ex: S(2)}) as follows. For distinct $x, y \in \mb{T}$, let $x \ra y$ if the angle subtended at the origin by the anticlockwise arc from $x$ to $y$ is $< 2 \pi / 3$. Let $S(3)$ be the substructure of $\mb{T}$ consisting of the points with rational argument. For $a \in \mb{T}$, we define $x', x''$ to be the anticlockwise rotations of $x$ about the origin by $2\pi / 3, 4\pi/3$ respectively. We use the same notation for subsets of $\mb{T}$.

We use the same open interval notation $\langle u, v \rangle$ as in \Cref{ex: S(2)}. Analogously to \Cref{forced ors in T}, we have:

\begin{lem} \label{forced ors in T for S(3)}
    Let $u, v \in \mb{T}$ with $u \ra v$. Let $a \in \langle u, v \rangle$. Then $a \ra b$ for $b \in [v, u']$, $a \ic b$ for $b \in [v', u'']$ and $a \la b$ for $b \in [v'', u]$.
\end{lem}

For $A \sub S(3)$, we define $\hat{A} = A \cup A' \cup A''$. As before, we may also define $\hat{A}$ for $A \in \Age(S(3))$ by embedding $A$ in $S(3)$ (the oriented graph structure does not depend on the embedding). We define cuts of $\hat{A}$ as before, and using \Cref{forced ors in T for S(3)} we have that extensions $(A, e) \in \mc{E}_A$ are in bijection with cuts of $\hat{A}$. To define a local SWIR, for $A, B, C \fin S(3)$ with $A \neq \varnothing$, we consider each triple $X = \langle u, v \rangle \cup \langle u', v' \rangle \cup \langle u'', v'' \rangle$ of $2\pi/3$-rotated cuts of $\hat{A}$, letting $B_X, C_X$ denote the intersections of $B, C$ with $X$, and we define $B \ind_A C$ if for each cut-triple $X$ of $\hat{A}$ we have $C_X$ anticlockwise from $B_X$. The definition of a \Ka functor is also analogous to \Cref{ex: S(2)}. For extensions $e_{(u, v)}, e_{(u', v')}, e_{(u'', v'')}$ of $A \in \mc{A}(S(3))$ corresponding to a triple of $2\pi / 3$-rotated cuts where $u \in A$, in $K(A)$ we set $e_{(u, v)} \ic e_{(u', v')}$, $e_{(u, v)} \la e_{(u'', v'')}$, $e_{(u', v')} \ic e_{(u'', v'')}$ (so on the anticlockwise arc from $u$ to $v$ in $\mb{T}$ we have $e^{}_{(u, v)} \ra e''_{(u', v')} \ra e'_{(u'', v'')}$).

\subsection{The generic \texorpdfstring{$\vec{C}_4$}{C4}-enlarged tournament} \label{ex: C_4-enlarged}

Let $\mc{C}$ be the class of finite oriented graphs $A$ such that $\ic$ is an equivalence relation with all parts (equivalence classes) of size $\leq 2$, and such that for every two pairs $\{u_0, u_1\}$, $\{v_0, v_1\}$ of vertices in different parts, the substructure on $\{u_0, u_1, v_0, v_1\}$ is a directed $4$-cycle. It is easy to see that $\mc{C}$ is an amalgamation class; we call its \Fr limit the \emph{generic $\vec{C}_4$-enlarged tournament}.

\begin{notn} \label{C4 part notn}
    For $A \in \mc{C}$, we write $\mc{P}_A$ for the set of parts of $A$. Let $A \in \mc{C}$ with all parts of size $2$. For $P \in \mc{P}_A$ and $a \in A \setminus P$, let $P(a)$ denote the vertex of $P$ with $a \ra P(a)$. For $a \in A$, let $\tld{a}$ be the vertex of $A$ with $a \ic \tld{a}$. For $A \in \mc{C}$ and $(A, e) \in \mc{E}_A$, let $(A, \tld{e}) \in \mc{E}_A$ be be the extension resulting by reversing all directed edges between $e$ and $A$. 
\end{notn}

\begin{prop} \hfill
    \begin{enumerate}[label=(\roman*)]
        \item \label{C4enlarged SWIR exp} The generic $\vec{C}_4$-enlarged tournament has a finite SWIR expansion.
        \item \label{C4enlarged Ka} The generic $\vec{C}_4$-enlarged tournament has a \Ka functor.
    \end{enumerate}
\end{prop}

Note that it is open whether the generic $\vec{C}_4$-enlarged tournament itself has a SWIR.

\begin{proof}
    \ref{C4enlarged SWIR exp}: Expand the digraph language by a unary predicate $R$, and let $\mc{C}'$ be the class of finite structures $A$ in the expanded language such that the digraph reduct of $A$ is in $\mc{C}$ and $R^A$ holds for exactly one vertex of each part of $A$ of size $2$. It is easy to see that $\mc{C}'$ is an amalgamation class; let $M' = \FrLim(\mc{C}')$. We define a SWIR for $M'$. For $A, B, C \fin M'$, define $B \ind_A C$ if, for all $b \in B \setminus A$, $c \in C \setminus A$ with $b$ in no part of $A$, $c$ in no part of $A$ and $(R(b) \wedge R(c)) \vee (\neg R(b) \wedge \neg R(c))$, we have $b \ra c$. It is straightforward to check that $\ind$ is a SWIR.
    
    \ref{C4enlarged Ka}: Let $M$ be the generic $\vec{C}_4$-enlarged tournament. Let $A \in \mc{A}(M)$. We first observe that there is a unique ``part-completion" $\hat{A}$ of $A$ in $\mc{A}(M)$: that is, there is unique $\hat{A} \in \mc{A}(M)$ with $A \sub \hat{A}$ such that for each $v \in \hat{A}$, there exists $a \in A$ with $v \ic a$ (where uniqueness is up to isomorphism fixing $A$ pointwise). Also observe that there is a bijection between $\mc{E}_{\hat{A}}$ and the set of extensions $(A, e) \in \mc{E}_A$ with no $a \in A$ with $e \ic a$, and observe that, given a \Ka functor $K$ for $M$, if $f : A \to B$ is an embedding in $\mc{A}(M)$ and $(A, e) \in \mc{E}_A$ with $e \ic a$ for some $a \in A$, then $K(f)(e) \ic f(a)$. It therefore suffices to define $K$ for $A \in \mc{A}(M)$ with all parts of size $2$, and then $K$ is uniquely defined for all $A \in \mc{A}(M)$ by the above.

    Let $A \in \mc{A}(M)$ with all parts of size $2$. Recall \Cref{C4 part notn}. For each $(A, e) \in \mc{E}_A$, each $P \in \mc{P}_A$ and each enumeration $\bar{Q}$ of the set $\{Q \in \mc{P}_A \mid Q(e) \neq Q(P(e))\}$, add to $K(A)$ a labelled copy $(A, e_{P, \bar{Q}})$ of $(A, e)$ over $A$. We define the oriented graph structure between distinct labelled extension vertices $e^{}_{P, \bar{Q}}, e'_{P', \bar{Q}'}$ as follows. If $P \neq P'$, orient $e^{}_{P, \bar{Q}} e'_{P', \bar{Q}'}$ as $P(e^{}_{P, \bar{Q}}) P'(e'_{P', \bar{Q}'})$. If $P = P'$, next compare $\bar{Q}, \bar{Q}'$ lexicographically: if $|\bar{Q}| < |\bar{Q}'|$ then set $e^{}_{P, \bar{Q}} \ra e'_{P', \bar{Q}'}$, and if $\bar{Q}, \bar{Q}'$ have the same length and first differ at index $i$, orient $e^{}_{P, \bar{Q}} e'_{P', \bar{Q}'}$ as $Q_i(e^{}_{P, \bar{Q}}) Q'_i(e'_{P', \bar{Q}'})$. If $(P, \bar{Q}) = (P', \bar{Q}')$, then necessarily $(A, e'_{P', \bar{Q}'}) = (A, \tld{e}^{}_{P, \bar{Q}})$, so set $e^{}_{P, \bar{Q}} \ic e'_{P', \bar{Q}'}$.

    Let $f : A \to B$ be an embedding in $\mc{A}(M)$, where again we assume that all parts of $A, B$ have size $2$. To define $K(f)$, we send each $e^{}_{P, \bar{Q}}$ to $e'_{f(P), f(\bar{Q})} \in K(B)$, where $(f(A), e'_{f(P), f(\bar{Q})}) \cong (A, e^{}_{P, \bar{Q}})$ and the oriented graph structure between $e'_{f(P), f(\bar{Q})}$ and each $P' \in \mc{P}_{B \setminus A}$ is given by $P'(e'_{f(P), f(\bar{Q})}) = P'(f(P)(e'_{f(P), f(\bar{Q})}))$.
\end{proof}

\subsection{The semigeneric \texorpdfstring{$\omega$}{omega}-partite tournament} \label{ex: semigeneric}

Let $\mc{C}$ be the class of finite $\omega$-partite tournaments $A$ satisfying the following \emph{parity condition}:
\begin{itemize}
    \item[$(\ast)$] for all $U_0, U_1 \sub A$ with $|U_0| = |U_1| = 2$ and with $U_0, U_1$ in different parts, the number of out-edges from $U_0$ to $U_1$ is even.
\end{itemize}

It is easy to show that $\mc{C}$ is an amalgamation class; we call its \Fr limit the \emph{semigeneric $\omega$-partite tournament}.

\begin{prop} \label{p: semigeneric}
    Let $M$ be the semigeneric $\omega$-partite tournament.
    \begin{enumerate}[label=(\roman*)]
        \item \label{semigeneric no SWIR} $M$ does not have a local SWIR.
        \item \label{semigeneric SWIR exp} $M$ has a finite SWIR expansion.
        \item \label{semigeneric non-univ} $\Aut(M)$ is non-universal.
    \end{enumerate}
\end{prop}
\begin{proof}
    \ref{semigeneric no SWIR}: We show that $\mc{C} = \Age(M)$ does not have a local SAO. Let $A = \{a_0, a_1\}$, $B = A \cup \{u\}$, $C = A \cup \{v_0, v_1\}$ with $a_0 \ic a_1 \ic u$, $v_0 \ic v_1$ and $a_i \ra v_i, a_i \la v_{1-i}$ for $i = 0, 1$. We have $A, B, C \in \mc{C}$. Suppose for a contradiction that $\otimes$ is a local SAO for $\mc{C}$. Consider the amalgam $B \otimes_A C$. By the parity condition for $\{a_0, u\}$ and $\{v_0, v_1\}$, we have that there exists $i \in \{0, 1\}$ with $u \ra v_i$ and $u \la v_{1-i}$. But then taking the automorphisms $a_0 \leftrightarrow a_1$ of $A$, $a_0 \leftrightarrow a_1$ of $B$ and $a_0 \leftrightarrow a_1$, $v_0 \leftrightarrow v_1$ of $C$ and applying (Inv), we have that $u \ra v_{1-i}$, contradiction. 

    \ref{semigeneric SWIR exp}: We first define a family of equivalence relations used to define the expansion. For $A \in \mc{C}$, let $\mc{P}_A = \{(P, Q) \mid P, Q \text{ distinct parts of } A\}$, and for $(P, Q) \in \mc{P}_A$ define an equivalence relation $\sim_{P, Q}$ on $P$ by: $u \sim_{P, Q} v$ if $u, v$ have the same out-neighbourhood in $Q$. By the parity condition, we see that $\sim_{P, Q}$ has at most two equivalence classes. If $\sim_{P, Q}$ has two equivalence classes, we let $E_{P, Q}$ be the set of equivalence classes, and if $\sim_{P, Q}$ has one equivalence class $\eps$, we let $E_{P, Q} = \{\eps, \varnothing\}$. Let $\mc{L}'$ be the expansion of the oriented graph language by a binary predicate $R$. Let $\mc{C}'$ be the class of $\mc{L}'$-structures $A'$ where $A := A'|_{\mc{L}} \in \mc{C}$ and where there exists a family of sets $(\delta_{P, Q} \mid (P, Q) \in \mc{P}_A)$ with $\delta_{P, Q} \in E_{P, Q}$ for $(P, Q) \in \mc{P}_A$ and $R^{A'} = \bigcup_{(P, Q) \in \mc{P}_A} (\delta_{P, Q} \times Q)$. It is not difficult to see that $\mc{C}'$ has strong amalgamation (also see the description of the SWIR below); let $M' = \FrLim(\mc{C}')$. We define a SWIR $\ind$ for $M'$. Let $A', B', C' \fin M'$. We define $B' \ind_{A'} C'$ if, for each $b \in B' \setminus A'$, $c \in C' \setminus A'$:
    \begin{itemize}
        \item if each of $b, c$ is not in the same part as any vertex of $A'$, then $b \ra c \wedge R(b, c) \wedge R(c, b)$;
        \item if $b$ is in the same part as some vertex of $A'$ and $c$ is not, then $R(b, c)$;
        \item if $c$ is in the same part as some vertex of $A'$ and $b$ is not, then $R(c, b)$.
    \end{itemize}
    (Note that if $b, c$ lie in the same part of $A'$ or in different parts of $A'$, then we do not need to specify which relations between $b, c$ hold, as there is only one possibility. Also note that in cases (ii) and (iii), specifying the $R$-relation forces the orientation of $bc$.) We leave it to the reader to check that $\ind$ is a SWIR.
    
    \ref{semigeneric non-univ}: It suffices to show that $\Aut(M)$ does not contain any involutions, as the structure consisting of two vertices and no edges lies in $\mc{C}$. For a contradiction, suppose there is an involution $\tau \in \Aut(M)$. Let $u \in M$ with $\tau(u) \neq u$. We have $\tau(u) \ic u$. If the $u$-part contains a $\tau$-fixed point $v$, take $e \in M$ with $u, v \ra e$ and $\tau(u) \la e$. Then $e \neq \tau(e)$ and the number of out-edges from $\{u, v\}$ to $\{e, \tau(e)\}$ is odd, so there is an edge between $e$ and $\tau(e)$, contradiction. If the $u$-part does not contain a $\tau$-fixed point, then there is $v \in M$ with $v \ic u$, $v \neq \tau(u)$ and $\tau(v) \neq v$. Take $e \in M$ with $u, v, \tau(v) \ra e$ and $\tau(u) \la e$, and we obtain a contradiction as before.
\end{proof}

\subsection{Ordered structures without a local SWIR} \label{ordered strs no SWIR}

We now give two examples showing that failure of automorphism-invariance is not the only obstacle to the presence of a SWIR on a \Fr structure:

\begin{prop} \label{p: ordered two-g no SWIR}
    The generic order expansion of the generic two-graph does not have a local SWIR.
\end{prop}
\begin{proof}
    Let $M$ be the generic order expansion of the generic two-graph. Suppose for a contradiction that $M$ has a local SWIR $\ind$. For $a, b, c \in M$ with $b < a$, $a < c$ and $b \ind_a c$, we have $b < c$ by transitivity, and we shall assume that $[abc] = 0$ (the argument for the case $[abc] = 1$ is analogous). Let $v_0, v_1, v_2 \in M$ with $v_0 < v_1 < v_2$ and $[v_0 v_1 v_2] = 1$. By (Ex) there is $w \in M$ with $w > v_2 > v_1$, $[w v_1 v_2] = 0$ and $v_0 \ind_{v_1 v_2} w$. As $v_1 < v_2 < w$ and $[w v_1v _2] = 0$ we have $v_1 \ind_{v_2} w$, so by (Tr) we have $v_0 v_1 \ind_{v_2} w$, and thus $v_0 \ind_{v_2} w \wedge v_1 \ind_{v_2} w$ by (Mon). So $[v_0 v_2 w] = [v_1 v_2 w] = 0$, but then $\{v_0, v_1, v_2, w\}$ only has one $3$-hyperedge, contradiction.
\end{proof}

\begin{prop} \label{ordered semigen no SWIR}
    The generic order expansion of the semigeneric $\omega$-partite tournament does not have a local SWIR.
\end{prop}
\begin{proof}
    Let $M$ be the generic order expansion of the semigeneric $\omega$-partite tournament, and suppose for a contradiction that $M$ has a local SWIR $\ind$. For all $a, b, c \in M$ in the same part with $a < b, c$ and $b \ind_a c$, we assume $b < c$; the argument for $b > c$ is analogous. Using (Ex), there is a directed $4$-cycle $a_0 b_0 a_1 b_1 \sub M$ and $c \in M$ with $c, a_0, a_1$ in the same part, $b_0 < b_1 < a_0 < a_1 < c$ and $b_0 b_1 \ind_{a_0 a_1} c$. As $a_1 a_0$, $a_1 c$ have a unique amalgam over $a_1$ by transitivity of $<$ and $\ic$, by (Ex) and (Inv) we have $a_0 \ind_{a_1} c$. We also have $a_1 \ind_{a_0} c$ by assumption. By (Mon) applied to $b_0 b_1 \ind_{a_0 a_1} c$ we have $b_0 \ind_{a_0 a_1} c$, and by (Tr) applied to $(b_0 \ind_{a_0 a_1} c) \wedge (a_0 \ind_{a_1} c)$ and then (Mon) we have $b_0 \ind_{a_1} c$. Similarly $b_1 \ind_{a_0} c$. As $a_0 c \cong a_1 c$ and $a_0 b_1 \cong a_1 b_0$, by (St) and (Inv) we have $a_0 b_1 c \cong a_1 b_0 c$, so the edges $b_0 c$, $b_1 c$ have the same orientation. But then there is an odd number of out-edges from $\{b_0, b_1\}$ to $\{a_0, c\}$, contradiction.
\end{proof}

\subsection{The lexicographic product} \label{ex: lex prod}

Throughout this section on the lexicographic product, we let $\mc{L}$ be a relational language.

\begin{defn}
    Let $M$ be an $\mc{L}$-structure. We say that $M$ is \emph{transitive} if all vertices of $M$ are isomorphic as substructures.
\end{defn}
\begin{defn}[{\cite[Definition 1.8]{Mei16}}]
    Let $M, N$ be $\mc{L}$-structures with $M$ transitive. Let $\mc{L}^s$ be an expansion of $\mc{L}$ by a binary relation symbol $s$. The \emph{lexicographic product} of $M$ and $N$ is the $\mc{L}^s$-structure $M[N]^s$ with domain $M \times N$ where:
    \begin{itemize}
        \item for each $R \in \mc{L}$ of arity $k$ and each $(u_i, v_i)_{i < k} \in (M \times N)^k$, we have that $M[N]^s \models R((u_i, v_i)_{i < k})$ iff:
        \begin{itemize}
            \item the $u_i$ are not all equal and $M \models R((u_i)_{i < k})$, or
            \item the $u_i$ are all equal and $N \models R((v_i)_{i < k})$; 
        \end{itemize}
        \item $M[N]^s \models s((u, v), (u', v'))$ iff $u = u'$.
    \end{itemize} 
\end{defn}

\begin{notn}
    For $A \sub M[N]^s$, we write $A^M = \pi_M(A) = \{u \in M \mid \ex v \in N, (u, v) \in A\}$. For $u \in A^M$, we write $A_u = \{v \in N \mid (u, v) \in A\}$.
\end{notn}

The proof of the following two lemmas is straightforward.

\begin{lem}[adapted from {\cite[Lemma 2.13]{Mei16}}] \label{lex emb lem}
    Let $M, N$ be $\mc{L}$-structures with $M$ transitive. Let $A, B \sub M[N]^s$. Then there is a one-to-one correspondence between embeddings $f : A \to B$ and tuples $(f^M, (f_u)_{u \in A^M})$ of embeddings $f^M : A^M \to B^M$ and $f_u : A_u \to B_{f^M(u)}$, $u \in A^M$, given by \[f(u, v) = (f^M(u), f_u(v)).\] We also have that $f$ is an isomorphism iff each element of the tuple $(f^M, (f_u)_{u \in A^M})$ is an isomorphism.
\end{lem}

\begin{lem}
    Let $M, N$ be \Fr $\mc{L}$-structures with $M$ transitive. Then $M[N]^s$ is a \Fr structure.
\end{lem}

We first consider the existence of a SWIR for $M[N]^s$.

\begin{prop}
    Let $M, N$ be \Fr $\mc{L}$-structures such that $M$ is transitive and has strong amalgamation, and let $\ind^M, \ind^N$ be (local) SWIRs for $M$, $N$. Then the following ternary relation $\ind$ defined on the finite substructures of $M[N]^s$ is a (local) SWIR:
    \[B \ind_A C \,\text{ if }\, (B^M \ind^M_{A^M} C^M \text{ and, for all } u \in A^M, \text{ we have } B_u \ind^N_{A_u} C_u).\]
\end{prop}
\begin{proof}
    (Inv): Let $A, B, C \fin M[N]^s$ with $B \ind_A C$. Let $g \in \Aut(M[N]^s)$. By \Cref{lex emb lem} there are $g^M \in \Aut(M)$ and $g_u \in \Aut(N)$, for $u \in M$, such that $g(u, v) = (g^M(u), g_u(v))$. By ($\ind^M$-Inv) we have $g^M B^M \ind^M_{g^M A^M} g^M C^M$. By ($\ind^N$-Inv), for all $u \in A^M$ we have $g_u B_u \ind^N_{g_u A_u} g_u C_u$. It is easy to see that for $D \sub M[N]^s$ and for $u \in M$ we have $(gD)^M = g^M D^M$ and $(gD)_{g^M(u)} = g_u D_u$. So $(gB)^M \ind^M_{(gA)^M} (gC)^M$ and, as $(gA)^M = g^M A^M$, for each $u' \in (gA)^M$ we have $u' = g^M(u)$ for some $u \in A^M$, giving $(gB)_{u'} \ind^N_{(gA)_{u'}} (gC)_{u'}$. Thus $gB \ind_{gA} gC$.

    We now show left existence, stationarity and monotonicity; the proof of the corresponding right-hand properties is similar.
    
    (Ex): Let $A, B, C \fin M[N]^s$. By ($\ind^M$-Ex) there exists $B'^M \sub M$ such that $B'^M \ind^M_{A^M} C^M$ and there is an isomorphism $f^M : B^M \cup A^M \to B'^M \cup A^M$ extending $\id_{A^M}$. Likewise, by ($\ind^N$-Ex), for each $u \in A^M$ there exists $B'_u \sub N$ such that $B'_u \ind^N_{A_u} C_u$ and there is an isomorphism $f_u : B_u \cup A_u \to B'_u \cup A_u$ extending $\id_{A_u}$. For $u \in B^M \setminus A^M$, let $B'_{f^M(u)} = B_u$ and let $f_u = B_u \to B'_{f^M(u)}$ be the identity. Let $B' = \{(u', v') \in M[N]^s \mid u' \in B'^M \wedge v' \in B'_{u'}\}$. Then by the definition of $\ind$, we have $B' \ind_A C$, and by \Cref{lex emb lem} the map $f : B \cup A \to B' \cup A$, $f(u, v) = (f^M(u), f_u(v))$ is an isomorphism extending $\id_A$.
    
    (Sta): suppose $B \ind_A C$, $B' \ind_A C$ and $\tp(B/A) = \tp(B'/A)$. To show that there is an isomorphism taking $B$ to $B'$ and fixing $AC$, by \Cref{lex emb lem} it suffices to show that (i) there is an isomorphism $f^M$ taking $B^M$ to $B'^M$ and fixing $A^M C^M$, and that (ii) for all $u \in B^M A^M C^M$ there is an isomorphism taking $B_u$ to $B'_{f^M(u)}$ fixing $A_u C_u$. (i) follows from ($\ind^M$-Sta), and by ($\ind^N$-Sta) we have (ii) for all $u \in A^M$. As $B^M \ind^M_{A^M} C^M$ and $B' \ind^M_{A^M} C^M$, by \Cref{SWIR strong amalg} we have that $C_u = \varnothing$ for $u \in B^M \setminus A^M$ and $B_u = B'_{f^M(u)} = \varnothing$ for $u \in C^M \setminus A^M$, so we have (ii) for all $u \in B^M A^M C^M$. 
    
    (Mon): Suppose $BD \ind_A C$. By the definition of $\ind$, it is immediate that $B \ind_A C$. To show that $D \ind_{AB} C$, it suffices to consider the case where $B$ is a singleton $\{(s, t)\}$  (as then the general case follows by induction). We have $(s, t)D \ind_A C$, so $sD^M \ind^M_{A^M} C^M$ and $(s, t)_u D_u \ind^N_{A_u} C_u$ for all $u \in A^M$. By ($\ind^M$-Mon) we have $D^M \ind^M_{sA^M} C^M$, and thus $D^M \ind^M_{((s, t)A)^M} C^M$. By ($\ind^N$-Mon) we have $D_u \ind^N_{((s, t)A)_u} C_u$ for all $u \in A^M$. If $s \in A^M$, we are done. If not, then we also must show that $D_s \ind^N_t C_s$. As $M$ has strong amalgamation and $sD^M \ind^M_{A^M} C^M$, by \Cref{SWIR strong amalg} we have $s \notin C^M$, so $C_s = \varnothing$. By \Cref{base triviality}, we have $D_s \ind^N_t t$, so by ($\ind^N$-Mon) we have $D_s \ind^N_t \varnothing$ as required. 
\end{proof}

\begin{prop}
    Let $M, N$ be \Fr $\mc{L}$-structures with \Ka functors and with $M$ transitive. Then $M[N]^s$ has a \Ka functor.
\end{prop}
\begin{proof}
    It suffices to define a \Ka functor for embeddings of finite substructures of $M[N]^s$ -- we do this for ease of exposition (working with an abstract definition of $\mc{A}(M[N]^s)$ is straightforward but less concise). Let $P, Q$ be \Ka functors for $M, N$ (where we assume that if $A \fin M$, then $A \sub P(A) \sub M$, and likewise for $Q$ and $N$). For $A \fin M$, write $P_{\ext}(A) = P(A) \setminus A$. Let $v \in N$ be arbitrary.
    
    Let $A \fin M[N]^s$, $A \neq \varnothing$. We define
    $K(A) = ((P_{\ext}(A^M) \times \{v\}) \,\cup \bigcup_{u \in A^M} (\{u\} \times Q(A_u))$.
    
    We have $K(A)^M = P(A^M)$ and $K(A)_u = Q(A_u)$ for each $u \in A$. Let $A \cup \{(s, t)\} \sub M[N]^s$ be an extension of $A$. By \Cref{lex emb lem}, if $s \in A^M$ then as $Q$ is a \Ka functor we have that $A \cup \{(s, t)\}$ embeds into $K(A)$, and if $s \notin A^M$ we use the fact that $P$ is a \Ka functor to obtain the embedding. Let $f : A \to B$ be an embedding of non-empty finite substructures of $M[N]^s$. To define $K(f)$, we again use \Cref{lex emb lem} with $K(f)^M = P(f^M)$ and $K(f)_u = Q(f_u)$ for $u \in A^M$, $K(f)_u = \id_{\{v\}}$ for $u \in K(A)^M \setminus A^M$. We leave the routine verification that $K$ is indeed a functor to the reader. 
\end{proof}

By the previous two propositions (and the results in \Cref{ex: Ka for tourn}, \Cref{betweenness and friends}, \Cref{ex: S(2)}), we immediately have the following regarding the infinite structures from Cherlin's catalogue of oriented graphs (\cite{Che98}) which result from lexicographic products of other structures in the catalogue:

\begin{prop}
    Let $n \leq \omega$, and let $I_n$ denote an anticlique of size $n$.
    \begin{enumerate}[label=(\roman*)]
        \item Let $T = \Q$ or the random tournament. Then the structures $T[I_n]$, $I_n[T]$ each have a global SWIR and a \Ka functor.
        \item The structures $S(2)[I_n]$, $I_n[S(2)]$ each have a local SWIR and a \Ka functor.
        \item Let $\oa{C}_3$ denote the directed $3$-cycle. Then the structures $\oa{C}_3[I_n]$, $I_n[\oa{C}_3]$ each have a local SWIR and a \Ka functor.
    \end{enumerate}
\end{prop}

We leave the routine verifications to the reader. Note that in the above, as $I_n$ has no edges and the other factor in the product is a tournament, the binary relation $s$ is definable and it is not necessary to add it to the language.

\subsection{Free superpositions} \label{ex: free superposition}

We give an example showing that the free superposition of \Fr classes with \Ka functors need not have a \Ka functor. In fact, we show the stronger statement that the free superposition need not have a universal automorphism group.

Let $\mc{L}_0 = \{\gamma\}$ with $\gamma$ ternary, and let $\mc{A}_0 = \mc{A}(\Q, \gamma)$ be the class of non-empty finite circularly ordered $\mc{L}_0$-structures (see \Cref{ex: circular order}). Let $\mc{L}_1 = \{R, B, <\}$ be a relational language with $R, B$ unary and $<$ binary, and let $\mc{A}_1$ consist of non-empty finite $\mc{L}_1$-structures whose points are coloured red/blue by $R$/$B$, where the red points are linearly ordered by $<$ and the blue points have no additional structure. In \Cref{ex: circular order}, we saw that $\mc{A}_0$ has a \Ka functor, and it is straightforward to see that $\mc{A}_1$ has a \Ka functor. Let $\mc{A}$ be the free superposition of $\mc{A}_0, \mc{A}_1$, and let $M = \FrLim(\mc{A})$. Let $M_i$ be the $\mc{L}_i$-reduct of $M$, $i = 0, 1$ (so $M_i = \FrLim(\mc{A}_i)$). Then $\Aut(M)$ has no involutions: suppose otherwise, and let $\tau \in \Aut(M)$ be an involution. Then $\tau$ is an involution of $M_0$, so has no fixed points (if $a$ is fixed and $v$ is moved by $\tau$, then we have $\gamma(a, v, \tau(v)) \wedge \gamma(a, \tau(v), v)$, which is impossible), but also $\tau$ is an involution of $M_1$, so has a fixed point -- contradiction. So the cyclic group $C_2$ does not embed into $\Aut(M)$. But there exists $A \in \mc{A}$ with $\Aut(A) = C_2$: take the structure consisting of two blue points. So $\Aut(M)$ is not universal.

\subsection{Meet-trees and meet-tree expansions} \label{meet-trees and meet-tree exps}

\subsubsection{Meet trees} \label{ex: meet-tree}
Let $\mc{L} = \{<, \wedge\}$ with $\leq$ a binary relation symbol and $\wedge$ a binary function symbol. We say that an $\mc{L}$-structure $A$ is a \emph{meet-tree} if:
\begin{itemize}
    \item $(A, <)$ is a poset;
    \item for each $a \in A$, $\{b \in A \mid b \leq a\}$ is linearly ordered by $<$;
    \item for $a, b \in A$, the set $\{c \in A \mid c \leq a, b\}$ is non-empty and has greatest element $a \wedge b$.
\end{itemize}

Let $\mc{C}$ be the class of finite meet-trees. It is straightforward to see that $\mc{C}$ is a \Fr class with strong amalgamation. Let $M = \FrLim(\mc{C})$. We call $M$ the \emph{dense meet-tree}. We show that $M$ has a local SWIR and a \Ka functor -- we give a brief proof, leaving some straightforward checking to the reader.

We call the least vertex of a finite meet tree the \emph{root}. Let $A$ be a finite meet tree, and let $B$ be a meet tree with $A \sub B$. Let $a_0, a_1 \in A$. If $a_0 < a_1$ and there is no $a \in A$ with $a_0 < a < a_1$, then we call $(a_0, a_1)$ a \emph{cut} of $A$, and we say that each vertex $v \in B$ with $a_0 < v < a_1$ lies in the cut $(a_0, a_1)$. We also consider $(-\infty, \rt(A))$ to be a cut of $A$. For $a \in A$, we say that $v \in B$ \emph{branches from $A$ at $a$} if $a < v$ and for each $a' \in A$ with $a' > a$ we have $a' \wedge v = a$. For a cut $(a_0, a_1)$ of $A$, we say that $v \in B$ \emph{branches from $A$ at the cut $(a_0, a_1)$} if $v \wedge a_1$ lies in $(a_0, a_1)$. 

\begin{prop}
    The dense meet-tree has a local SWIR.
\end{prop}
\begin{proof}
    Let $A, B, C \fin M$ with $A \neq \varnothing$. Define $B \ind_A C$ if for all $b \in B \setminus A$, $c \in C \setminus A$ the following hold:
    \begin{itemize}
        \item if there is a cut $(a, a')$ of $A$ with $b \wedge a'$, $c \wedge a'$ inside it, then $b \wedge a' < c \wedge a'$;
        \item if $b, c$ branch from $A$ at the same point $a \in A$, then $b \wedge c = a$.
    \end{itemize}
    It is not difficult to check that $\ind$ is a SWIR, and we leave this to the reader.
\end{proof}
\begin{rem}
    This is essentially identical to the definition of cone-independence in \cite[Definition 2.1]{KRR25}, but we work with a local SWIR rather than fixing a point. See also \cite[Lemma 2.6]{KRR25}, where independence is checked in detail.
\end{rem}

\begin{prop}
    The dense meet-tree has a \Ka functor.
\end{prop}
\begin{proof}
    Let $A \in \mc{A}(M)$. Each one-point extension $(A, e)$ of $A$ is exactly one of the following three types:
    \begin{itemize}
        \item there is a unique vertex $a \in A$ such that $e$ branches from $A$ at $a$;
        \item there is a unique cut $(u, v)$ of $A$ such that $e$ branches from $A$ at $(u, v)$, and $e \wedge v = e$;
        \item there is a unique cut $(u, v)$ of $A$ such that $e$ branches from $A$ at $(u, v)$, and $e \wedge v \neq e$.
    \end{itemize}

    We define $K(A)$ to be the amalgam of $\mc{E}_A$ over $A$ where, for each cut $(u, v)$ of $A$, the extension vertex $e$ of the $(u, v)$-branching extension $(A, e)$ with $e \wedge v = e$ is less than $e' \wedge v$, where $e'$ is the extension vertex of the $(u, v)$-branching extension $(A, e')$ with $e' \wedge v \neq e'$ -- note that this completely specifies the amalgam. For $f : A \to B$ in $\mc{A}(M)$, we define $K(f)$ by sending each extension vertex $e$ of $(A, e)$ in $K(A)$ to the extension vertex in $K(B)$ corresponding to $(A, e) \otimes_A B$, where $\otimes$ is the local SAO induced by the local SWIR for $M$ defined in the previous proposition. It is straightforward to check that $K$ is a \Ka functor.
\end{proof}

\subsubsection{Meet-tree expansions of free amalgamation classes} \label{ex: meet-tree exp of free amalg}

Let $\mc{L}^\ri$ be a relational language without unary predicates, and let $\mc{A}^\ri$ be a non-trivial free amalgamation class of finite $\mc{L}^\ri$-structures, where \emph{non-trivial} means that the \Fr limit is not an indiscernible set. Let $M^\ri = \FrLim(\mc{A}^\ri)$, and let $M$ be the free superposition of $M^\ri$ with the dense meet-tree.

\begin{lem}
    The structure $M$ defined above has no involutions.
\end{lem}
\begin{proof}
    Let $\tau \in \Aut(M)$ be an involution. We first show that $\tau$ has no fixed points. Suppose for a contradiction that $a \in M$ is a $\tau$-fixed point. Then, as the set of predecessors of $a$ is linearly ordered, it is fixed pointwise by $\tau$. Let $c \in M$ with $\tau(c) \neq c$. As $\mc{A}^\ri$ is a non-trivial free amalgamation class and $\mc{L}^\ri$ has no unary predicates, there is $B_0^\ri \in \mc{A}^\ri$ with $\{c\} \sub B_0^\ri$ such that there is $b \in B_0^\ri$, $b \neq c$, with $b, c$ $R$-related for some $R \in \mc{L}^\ri$. Take an amalgam $B_1^\ri$ of $B_0^\ri$ and $\{a, c\}^\ri$ over $\{c\}$, and then take the free amalgam $B_2^\ri$ of $B_1^\ri$ and $\{a, c, \tau(c)\}^\ri$ over $\{a, c\}^\ri$. In $B_2^\ri$, we have that $b, c$ are $R$-related but $b, \tau(c)$ are not. As $\mc{A}$ is a free superposition, using the extension property of $M$ we thus have that there exists $b' \in M$, $b' < a$, such that $b', c$ are $R$-related but $b', \tau(c)$ are not. Then $\tau(b') \neq b'$, contradiction. So each point of $M$ is moved by $\tau$.

    Now take $a \in M$, and let $v = a \wedge \tau(a)$. Then $v < a, \tau(a)$, so $\tau(v) < a$. But then $v, \tau(v)$ are $<$-comparable, and $\tau(v) \neq v$. As $\tau$ is an involution, this gives a contradiction.
\end{proof}
\begin{rem}
    \cite[Lemma 5.13]{KRR25} states that the only automorphism of $M$ fixing an interval pointwise is the identity, which immediately gives that involutions of $M$ have no fixed points. We give a direct proof here for the convenience of the reader.
\end{rem}

\begin{prop}
    Let $M$ be the \Fr limit of the free superposition of the class of finite meet trees and a non-trivial free amalgamation class in a language without unary predicates. Then $\Aut(M)$ is not universal.
\end{prop}
\begin{proof}
    By the previous lemma, it suffices to produce $A \in \Age(M)$ such that $\Aut(A)$ contains an involution. Take $A \in \Age(M)$ to consist of three distinct vertices $a, b, a \wedge b$ with no $\mc{L}^\ri$-relations.
\end{proof}

\section{Unique extensibility} \label{s: unique ext}

We now briefly consider unique extensibility. For $M$ a \Fr structure, we say that the class $\Ainf(M)$ of infinite structures embeddable in $M$ is \emph{uniquely extensible} if each $A \in \Ainf(M)$ admits a \emph{uniquely extensive embedding} into $M$: an embedding $f : A \to M$ such that each automorphism of $f(A)$ extends to a unique automorphism of $M$. Note that by uniqueness, group composition is respected, so a uniquely extensive embedding $f : A \to M$ induces a group embedding $\Aut(f(A)) \to \Aut(M)$. 

Henson (\cite[Theorem 3.1]{Hen71}) showed that $\Ainf(M)$ is uniquely extensible for $M$ the random graph (also shown independently by Macpherson and Woodrow in \cite{MW92}), and Bilge (\cite{Bil12}) extended this result to any transitive \Fr structure $M$ with free amalgamation in a finite relational language (and satisfying the non-triviality condition that $M$ is not an indiscernible set). For the particular case of the generic $K_n$-free graph, see \cite{BJ12}. In \cite{KSW25} the present authors showed via new techniques that $\Ainf(M)$ is uniquely extensible for $M$ equal to the random poset.

Bilge observed (\cite[Proposition 8.33]{Bil12}) that $\Ainf((\Q, <))$ is not uniquely extensible. Indeed, no countably infinite linear order $A$ admits a uniquely extensive embedding $f$ into $(\Q, <)$: if such an $f$ were to exist, then by uniqueness $f(A)$ must be dense in $\Q$, and one can construct an automorphism of $f(A)$ for which any extension would be forced to send a rational to an irrational. We show in this section that this failure of unique extensibility for $(\Q, <)$ does not transfer to free superpositions: 

\begin{prop} \label{p: ordered random graph uniq ext}
    Let $M$ be the generic ordered graph. Then $\Ainf(M)$ is uniquely extensible.
\end{prop}

The proof is via a simple generalisation of Henson's original construction, and we only provide a sketch.

\begin{proof}[Proof sketch.]
    Let $A$ be a countably infinite ordered graph. Let $M_0 = A$. We inductively construct a chain $M_0 \sub M_1 \sub \cdots$ of ordered graphs. Suppose $M_{i-1}$ is already given; we now define $M_i$. For each $S \fin M_{i-1}$ with $|S \cap A| = i$ and each $v \in M_{i-1}$, add two new vertices $u_{S, v^-}$, $u_{S, v^+}$, each adjacent precisely to $S$, with the order relation between new vertices and $M_{i-1}$ defined by:
    \begin{itemize}
        \item $u_{S, v^-} < v$ and $u_{S, v^-} > w$ for all $w \in M_{i-1}$ with $w < v$;
        \item $u_{S, v^+} > v$ and $u_{S, v^+} < w$ for all $w \in M_{i-1}$ with $w > v$.
    \end{itemize}
    We call $S$ the \emph{graph base} and $v$ the \emph{order base} of $u_{S, v^-}$, $u_{S, v^+}$.
    
    We define the ordered graph structure between new verties $u, u'$ as follows. We define that $u$, $u'$ are not adjacent, and letting $v, v'$ be the order bases of $u, u'$:
    \begin{itemize}
        \item if $v < v'$ we define $u < u'$ and if $v > v'$ we define $u > u'$;
        \item if $v = v'$ then, letting $S, S'$ be the graph bases of $u, u'$, we compare the underlying finite sets $S, S'$ lexicographically in the order on $M_{i-1}$ and use this to determine the order relation between $u, u'$.
    \end{itemize}

    This completes the inductive construction. Let $M_\omega = \bigcup_{i < \omega} M_i$. It is then straightforward to check (in a manner similar to the constructions of Henson and Bilge) that:
    \begin{itemize}
        \item each element of $\Aut(A)$ extends to $\Aut(M_\omega)$: given $g \in \Aut(M_{i-1})$, extend to $\hat{g} \in \Aut(M_i)$ by defining $\hat{g}(u_{S, v^\pm}) = u_{gS, gv^{\pm}}$ (note that as $g$ is in particular an order-automorphism, it preserves the lexicographic order on finite sets);
        \item for each $g \in \Aut(A)$ there is a unique $h \in \Aut(M_\omega)$ extending $g$: this follows from the observation that any such $h$ must stabilise each $M_i$ setwise, as for each vertex of $M_\omega$ we have that $h$ preserves the number of points in $A$ adjacent to that vertex;
        \item $M_\omega$ is isomorphic to the random ordered graph: one checks the one-point extension property (witnessing the one-point extension using vertices $u_{S, v^\pm} \in M_j$ for $j$ sufficiently large). \qedhere
    \end{itemize}
\end{proof}

We conjecture the following generalisation of \Cref{p: ordered random graph uniq ext}:

\begin{conj} \label{c: LO free}
    Let $M^<$ be the generic order expansion of a transitive \Fr structure $M$ with free amalgamation in a finite relational language, where $M$ is not an indiscernible set. We conjecture that $\Ainf(M^<)$ is uniquely extensible. (See \Cref{d: gen order exp} for the definition of the generic order expansion.)
\end{conj}

As a potential proof approach, we suggest that it may be possible to combine the idea of the proof of \Cref{p: ordered random graph uniq ext} with the proof of Bilge's result \cite[Theorem 8.30]{Bil12} which gives unique extensibility for non-trivial free amalgamation structures. (We note that a step of the proof of \cite[Theorem 8.30]{Bil12} involves building infinite structures with trivial automorphism groups -- see \cite[Proposition 8.24]{Bil12}. With a linear order available, this is straightforward: one just orders the infinite structure as $(\N, <)$.)

\section{Further questions}

For two structures considered in this paper, we found a finite expansion with a SWIR, but the existence of a SWIR for the structure itself remains open:

\begin{qn}
    Do the generic $3$-hypertournament (\ref{ex: 3-ht}) or the generic $C_4$-enlarged tournament (\ref{ex: C_4-enlarged}) have a SWIR?
\end{qn}

We were also not able to find a \Ka functor for $P(3)$:

\begin{qn}
    Does the ultrahomogeneous oriented graph $P(3)$ have a \Ka functor? (See \cite{Che98} for a description of this structure.)
\end{qn}

The second and third author, together in a project with Shujie Yang, have found a finite SWIR expansion of $P(3)$ (given by adding a constant to fix a point), and used this to show that $\Aut(P(3))$ is simple.

It would also be interesting to investigate generalisations of \Ka functors to continuous logic. In \cite{BY14}, Ben Yaacov constructs the Gurarij space via an analogue of the \Ka tower construction, and perhaps similar results for a range of continuous structures are achievable. In the conclusion of \cite{KM17}, the authors mention that they have successfully obtained such a generalisation, but we could not find any further references in the literature. We would also be interested in analogues of unique extensibility in this context:

\begin{qn}
    Does the Gurarij space satisfy a suitable analogue of unique extensibility?
\end{qn}

It would also be nice to consider unique extensibility more generally. In \cite{KSW25} the authors considered the case of the generic poset, unique extensibility for which was obtained via a relatively technical construction -- it seemed that the presence of a SIR did not lead to a straightforward generalisation of the results by Bilge for free amalgamation structures.

Our final question concerns the two papers \cite{GK11}, \cite{HKO11}, which define that $A \sub M$ is \emph{symmetrically embedded} in $M$ if each element of $\Aut(A)$ extends to $\Aut(M)$ (a weaker definition than that of $\circ$-extensive embedding, as the map $\Aut(A) \to \Aut(M)$ need not be a homomorphism), and which give Ramsey-type results regarding \emph{symmetric indivisibility}: they show for several \Fr structures $M$ that for any finite colouring of $M$ there is a monochromatic symmetrically-embedded copy of $M$. Also, the lexicographic product is used to give a quick proof that each countable linear order has a symmetrically-embedded copy in $\Q$ (\cite[Lemma 2.8]{HKO11}). We ask:
\begin{qn}
    Can some of the results from \cite{GK11}, \cite{HKO11} be unified with the \Ka functor framework?
\end{qn}

\bibliographystyle{alpha}
\bibliography{references}

\end{document}